\documentclass[11pt, a4paper]{amsart}
\usepackage[dvips]{epsfig}
\usepackage{amsmath}
\usepackage{amssymb}
\usepackage{amsfonts}
\usepackage{amsthm}
\usepackage{amsbsy}
\usepackage{amsgen}
\usepackage{amscd}
\usepackage{amsopn}
\usepackage{amstext}
\usepackage{amsxtra}
\usepackage{enumerate}
\usepackage[bookmarks=false]{hyperref}

\usepackage{times, graphicx}
\usepackage{eso-pic}
\usepackage{lastpage}
\usepackage{mathtools}

\usepackage{mathrsfs}

\newif\ifanswers
\answerstrue

\newtheorem{theorem}{Theorem}[section]
\newtheorem{proposition}[theorem]{Proposition}
\newtheorem{lemma}[theorem]{Lemma}
\newtheorem{corollary}[theorem]{Corollary}
\theoremstyle{definition}
\newtheorem{definition}[theorem]{Definition}
\newtheorem{notation}[theorem]{Notation}
\newtheorem{remark}[theorem]{Remark}

\newcommand{\supp}{\operatorname{supp}}
\newcommand{\vol}{\operatorname{Vol}}
\newcommand{\prob}{\mathcal{P}}
\newcommand{\diam}{\operatorname{diam}}

\newcommand{\cov}{\operatorname{Cov}}

\numberwithin{equation}{section}

\newcommand {\E} {\mathbb{E}}
\newcommand {\M} {\mathcal{M}}
\newcommand {\R} {\mathbb{R}}
\newcommand {\N} {\mathbb{N}}
\newcommand {\Z} {\mathbb{Z}}
\newcommand {\C} {\mathbb{C}}

\newcommand {\Cc} {\mathcal{C}}
\newcommand {\Ec} {\mathcal{E}}
\newcommand {\Tc} {\mathcal{T}}
\newcommand {\Fc} {\mathcal{F}}
\newcommand {\Sc} {\mathcal{S}}
\newcommand {\Dc} {\mathcal{D}}
\newcommand {\Gc} {\mathcal{G}}
\newcommand {\Zc} {\mathcal{Z}}
\newcommand {\Pc} {\mathcal{P}}
\newcommand {\Hc} {\mathcal{H}}

\newcommand{\gfr}{\mathfrak{g}}

\newcommand {\Zf} {\mathfrak{Z}}

\newcommand{\nod}{\mathcal{N}}
\newcommand{\Ac}{\mathcal{A}}
\newcommand{\Bc}{\mathcal{B}}

\DeclareMathAlphabet{\pazocal}{OMS}{zplm}{m}{n}

\begin{document}

\title{Topologies of nodal sets of random band limited functions}

\author{Peter Sarnak}
\thanks{Research of P.S. is supported by an NSF grant.}
\address{Department of Mathematics, Princeton University and the Institute for Advanced Study, US}

\author{Igor Wigman}
\thanks{
The research leading to these results has received funding from the
European Research Council under the European Union's Seventh
Framework Programme (FP7/2007-2013), ERC grant agreement
n$^{\text{o}}$ 335141 (I.W.)}

\address{Department of Mathematics, King's College London, UK}


\begin{abstract}
It is shown that the topologies and nestings of the zero and nodal
sets of random (Gaussian) band limited functions have universal
laws of distribution. Qualitative features of the
supports of these distributions are determined. In
particular the results apply to random monochromatic waves and to
random real algebraic hyper-surfaces in projective space.
\end{abstract}

\maketitle

\section{Introduction}

\label{sec:intro}

Nazarov and Sodin ~\cite{NS,So} and very recently in ~\cite{NS2015}
have developed some powerful general techniques
to study the zero (``nodal") sets of functions of several variables coming from
Gaussian ensembles. Specifically they show that the number of connected components
of such nodal sets obey an asymptotic law. In ~\cite{Sa} we pointed
out that these may be applied to ovals of a random real plane curve, and in
~\cite{LL} this is extended to real hypersurfaces in $\mathbb{P}^{n}$. In
~\cite{GW} the barrier technique from ~\cite{NS} is used
to show that ``all topologies" occur with positive probability in the
context of real sections of high tensor powers of a holomorphic line
bundle of positive curvature, on a real projective manifold.

\subsection{Gaussian band-limited functions}

In this paper we apply these techniques to study the laws of distribution of
the topologies of a random band limited function.
Let $\M=(\Sc^{n},g)$ denote the $n$-sphere with a smooth Riemannian metric $g$.
Choose an orthonormal basis $\{ \phi_{j}\}_{j=0}^{\infty}$ of eigenfunctions of its
Laplacian
\begin{equation}
\label{eq:Delta phi+t^2phi = 0}
\Delta \phi_{i}+t_{i}^{2}\phi_{i}=0,
\end{equation}
\begin{equation*}
0=t_{0}<t_{1}\le t_{2} \ldots .
\end{equation*}

Fix $\alpha\in [0,1]$ and denote by $\mathcal{E}_{\M,\alpha}(T)$
($T$ a large parameter) the finite dimensional Gaussian ensemble of functions on $\M$ given
by
\begin{equation}
\label{eq:f = sum cj phij}
f(x) =f_{\alpha;T}(x)= \sum\limits_{\alpha T \le t_{j} \le T} c_{j}\phi_{j}(x),
\end{equation}
where $c_{j}$ are independent real Gaussian variables of mean $0$ and variance $1$. If $\alpha = 1$,
which is the important case of ``monochromatic"
random functions, we interpret \eqref{eq:f = sum cj phij} as
\begin{equation}
\label{eq:f = sum cj phij mono}
f(x) = \sum\limits_{T-\eta(T) \le t_{j} \le T} c_{j}\phi_{j}(x),
\end{equation}
where $\eta(T) \rightarrow \infty$ with $T$, and $\eta(T)=o(T)$.
The Gaussian ensembles $\mathcal{E}_{\M,\alpha}(T)$ are our
$\alpha$-band limited functions, and they do not depend on the choice
of the o.n.b. $\{\phi_{j}\}$. The aim is to study the nodal sets of a typical
$f$ in $\mathcal{E}_{\M,\alpha}(T)$ as $T\rightarrow\infty$.

\subsection{Nodal set of $f$ and its measures}

Let $V(f)$ denote the nodal set of $f$, that is
$$V(f) = \{ x:\: f(x)=0\}.$$ For almost all $f$'s in $\mathcal{E}_{\M,\alpha}(T)$
with $T$ large, $V(f)$ is a smooth $(n-1)$-dimensional compact manifold. We decompose
$V(f)$ as a disjoint union $\bigsqcup\limits_{c\in\mathcal{C}(f)}c$ of its connected
components. The set $\Sc^{n}\setminus V(f)$ is a disjoint union of connected components
$\bigsqcup\limits_{\omega\in\Omega(f)}\omega$, where each $\omega$ is a smooth
compact $n$-dimensional manifold with smooth boundary. The components $\omega$ in
$\Omega(f)$ are called the nodal domains of $f$. The nesting relations between
the $c$'s and $\omega$'s are captured by the tree $X(f)$ (see \S\ref{sec:bas conventions}),
whose vertices are the points $\omega\in\Omega(f)$ and edges $e$ run
from $\omega$ to $\omega'$ if $\omega$ and $\omega'$ have a (unique!) common
boundary $c\in\mathcal{C}(f)$ (see Figure \ref{fig:nesting tree}). Thus the edges $E(X(f))$ of $X(f)$ correspond to
$\mathcal{C}(f)$.

\begin{figure}[ht]
\centering
\includegraphics[height=50mm]{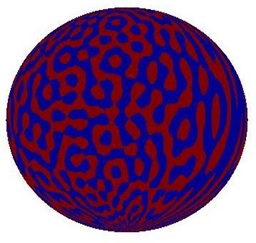}
\caption{A nodal picture of a spherical harmonic. The blue and red are positive
and negative domains respectively, and the nodal set is the interface between these.}
\label{fig:sphere}
\end{figure}

\begin{figure}[ht]
\centering
\includegraphics[height=54mm]{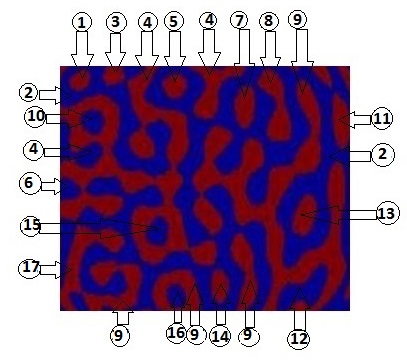}
\includegraphics[height=54mm]{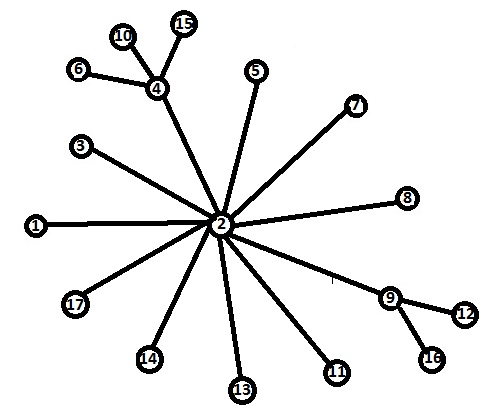}
\caption{
To the right: the nesting tree $X(f)$ corresponding to a fragment of the nodal picture
in Figure \ref{fig:sphere}, to the left, containing
$17$ nodal domains (where we neglected some small ones lying next to the boundary).
Figure \ref{fig:sphere} is essential
for deciding which components merge on the sphere outside of the fragment.}
\label{fig:nesting tree}
\end{figure}

As mentioned above, Nazarov and Sodin have determined the asymptotic law for the
cardinality $|\mathcal{C}(f)|$ of $\mathcal{C}(f)$ as $T\rightarrow\infty$.
There is a {\em positive} constant $\beta_{n,\alpha}$ depending on $n$ and $\alpha$
(and not on $\M$) such that, with probability tending to $1$ as $T\rightarrow\infty$,
\begin{equation}
\label{eq:|C(f)| beta c vol T^n}
|\mathcal{C}(f)|\sim \beta_{n,\alpha}\frac{\omega_{n}}{(2\pi)^{n}}\vol(\M) T^{n},
\end{equation}
here $\omega_{n}$ is the volume of the unit $n$-ball.
We call these constants $\beta_{n,\alpha}$ the Nazarov-Sodin constants. Except for
$n=1$ when the nodal set is a finite set of points and \eqref{eq:|C(f)| beta c vol T^n}
can be established by the Kac-Rice formula
($\beta_{1,\alpha} = \frac{1}{\sqrt{3}}\cdot \sqrt{1+\alpha+\alpha^{2}}$), these numbers are not known explicitly.

In order to study the topologies of the components of $\Cc(f)$ and of the nesting trees $X(f)$,
we introduce two measure spaces. Let $H(n-1)$ be the countable\footnote{That $H(n-1)$ is countable follows, for example, from Cheeger Finiteness Theorem ~\cite[Theorem 7.11 on p. 340]{Cha} i.e. that there is only finitely many differomorphism types satisfying certain geometric conditions,
see also Lemma \ref{lem:Cheeger} below and its derivation from Cheeger Finiteness Theorem in section \ref{sec:Cheeger proof}.} set of diffeomorphism types of compact
$(n-1)$-dimensional manifolds that can be embedded in $\Sc^{n}$ and let $\Tc$ denote the set of finite rooted
trees. As discrete spaces they carry measures, and define the discrepancy $D$ between $\mu$ and $\nu$ by
\begin{equation}
\label{eq:D dist prob meas def}
D(\mu,\nu)=\sup\limits_{A} \left|\mu(A)-\nu(A) \right|,
\end{equation}
where the supremum is taken w.r.t. all subsets $A$ of $H(n-1)$ (resp. $\Tc$).

Associating to $c\in\Cc(f)$ its topological type\footnotemark $t(c)$ gives a map from $\Cc(f)$ to $H(n-1)$. Similarly,
each $c\in\Cc(f)$ is an edge in the tree $X(f)$ so that removing it leaves two rooted trees. We let $e(c)$
be the smaller one (if they are equal in size choose any one of them) and call it the end of $X(f)$ corresponding
to $c$. Hence $e$ gives a map from $\Cc(f)$ to $\Tc$.

\footnotetext{Throughout the paper by the ``topological type" of $c$ we mean ``diffeomorphism type".}

With these associations we define the key probability measures $\mu_{\Cc(f)}$ and $\mu_{X(f)}$ which
measure the distribution of the topologies of the components of $V(f)$ and of the nesting ends of $X(f)$ by
\begin{equation}
\label{eq:muCf def}
\mu_{\Cc(f)}:=\frac{1}{|\Cc(f)|}\sum\limits_{c\in\Cc(f)}\delta_{t(c)}
\end{equation}
and
\begin{equation}
\label{eq:muXf def}
\mu_{X(f)}:=\frac{1}{|\Cc(f)|} \sum\limits_{c\in\Cc(f)}\delta_{e(c)},
\end{equation}
where $\delta_{\xi}$ is a point mass at $\xi$.

\subsection{Statement of the main result}

Our main result is that as $T\rightarrow\infty$ and for typical $f\in\Ec_{\M,\alpha}(T)$ the above
measures converge to Universal Laws (that is {\em probability} measures) on $H(n-1)$ and $\Tc$ respectively:

\newpage

\begin{theorem}

\label{thm:main thm lim meas}

\begin{enumerate}

\item

\label{it:lim muf exist}
There are universal probability measures $\mu_{\Cc,n,\alpha}$ on $\hbox{H(n-1)}$ and $\mu_{X,n,\alpha}$ on $\Tc$,
depending only on $n$ and $\alpha$ but not on $\M$, such that for every $\epsilon>0$,
\begin{equation}
\label{eq:lim muf exist}
\prob\left\{f\in\Ec_{\M,\alpha}(T):\: \max \left[D(\mu_{\Cc(f)},\mu_{\Cc,n,\alpha}),
D(\mu_{X(f)},\mu_{X,n,\alpha})\right]>\epsilon \right\}
\end{equation}
tends to $0$ as $T\rightarrow\infty$.

\item

\label{it:lim muf supp}

The support of $\mu_{\Cc,n,\alpha}$ is
$$\supp\mu_{\Cc,n,\alpha} = H(n-1)$$ and the support of $\mu_{X,n,\alpha}$ is
$$\supp\mu_{X,n,\alpha} = \Tc.$$

\end{enumerate}

\end{theorem}

\begin{remark}
The most difficult case of part \eqref{it:lim muf supp} of Theorem \ref{thm:main thm lim meas}
is the monochromatic case $\alpha=1$. The proof of this for $n>2$
is given in the companion paper ~\cite{CS}.
\end{remark}

\begin{remark}
Though formulated for the sphere with arbitrary smooth metric,
Theorem \ref{thm:main thm lim meas} holds on general compact smooth Riemannian manifolds with no boundary.
Though the Jordan-Brouwer Theorem might fail for other manifolds (hence the nesting graph might fail to be
a tree), it still holds {\em locally}, which is sufficient for all our needs.
\end{remark}

\vspace{5mm}

The measures $\mu_{\Cc(f)}$ and $\mu_{X(f)}$ carry a lot of information. If $$F:H(n-1)\rightarrow P$$ is a $P$-valued
topological invariant then one can define the $P$-distribution of $$f\in \Ec_{\M,\alpha}(T)$$ to be
\begin{equation*}
\mu_{F(f)} = \frac{1}{|\Cc(f)|}\sum\limits_{c\in\Cc(f)}\delta_{F(t(c))},
\end{equation*}
i.e. $\mu_{F(f)}$ is the pushforward of $\mu_{\Cc(f)}$ to $P$. Theorem \ref{thm:main thm lim meas} then
gives the universal distribution of $\mu_{F(f)}$, namely it is simply the pushforward of $\mu_{\Cc,n,\alpha}$.
The same applies to any $Q$-valued map $G:\Tc\rightarrow Q$. Of special interest in this connection
are Betti distributions and domain connectivities.

The first is the vector of Betti numbers given by
\begin{equation*}
F(c) = \text{Betti}(c) = (\beta_{1}(c),\ldots,\beta_{k}(c)),
\end{equation*}
where $\beta_{j}(c)$ is the $j$-th Betti number of $c$. Here
we are assuming $n=2k$ or $2k+1$ with $k>0$ (these Betti numbers give the rest since $c$ is connected and
applying Poincar\'{e} duality). The image of $H(n-1)$ under $F$ can be shown to be $P_{n}$, which is $(\Z_{\ge 0})^{k}$
if $n$ is odd and $(\Z_{\ge 0})^{k-1}\times (2 \Z_{\ge 0})$ if $n$ is even. Thus Theorem \ref{thm:main thm lim meas}
yields the universal distribution of the vector of Betti numbers according to a Law $\mu_{\text{Betti},\alpha,n}$
on $P_{n}$ and whose support is $P_{n}$. \label{pag:Betti}We show that $\mu_{\text{Betti},\alpha,n}$ has finite total mean
(which is not a formal consequence of Theorem \ref{thm:main thm lim meas}):
\begin{equation*}
\sum\limits_{y\in P_{n}}\left(\sum\limits_{j=1}^{k}y_{j}\right)\mu_{\text{Betti},n,\alpha}\left(\left\{ y \right\} \right) < \infty.
\end{equation*}

For the domain connectivity distributions let $G:\Tc\rightarrow\N$ be the function which assigns to each
rooted tree one plus the degree (i.e. number of neighbours) of the root. Now the root of $e(c)$ in $\Tc$
corresponds to a nodal domain $\omega$ of $f$, and $G(e(c))$ is the connectivity $m(\omega)$ of $\omega$,
that counts the number of its boundary components. The measure $\mu_{G(f)}$ is essentially the connectivity measure:
\begin{equation}
\label{eq:muGamma def}
\mu_{\Gamma(f)}:=\frac{1}{|\Omega(f)|} \sum\limits_{\omega\in\Omega(f)}\delta_{m(\omega)}
\end{equation}
on $\N$.

Theorem \ref{thm:main thm lim meas} yields a universal distribution for these connectivities according to a
Law $\mu_{\Gamma,n,\alpha}$ on $\N$, and whose support is $\N$. Since the means over $\N$ of each measure in
\eqref{eq:muGamma def} is (see \S\ref{sec:top V(f) rem}) $$2-\frac{2}{|\Omega(f)|},$$ it follows that the
mean of the law $\mu_{\Gamma,n,\alpha}$ is at most $2$.

\subsection{Applications} The extreme values of $\alpha$, namely $0$ and $1$ are the most
interesting. The case $\alpha=1$ is the monochromatic random wave
(and also corresponds to random spherical harmonics) and it has been suggested
by Berry \cite{Berry 1977} that it models the individual eigenstates of the quantization
of a classically chaotic Hamiltonian. The examination of the count of nodal domains  (for $n=2$)
in this context was initiated by ~\cite{BGS}, and ~\cite{BS}, and the latter suggest some interesting
possible connections to exactly solvable critical percolation models.

The law $\mu_{\Gamma, 2, 1}$ gives
the distribution of connectivities of the nodal domains for monochromatic waves.
Barnett and Jin's numerical experiments ~\cite{BJ} give the following values for its mass on atoms.

\vspace{5mm}

\begin{tabular}{|c|c|c|c|c|c|c|c|c|c|}
\hline
connectivity & 1 & 2 & 3 & 4 & 5 &6 & 7 \\
\hline
$\mu_{\Gamma,2,1}$ &.91171 &.05143 &.01322 &.00628 & .00364 &.00230 &.00159 \\
\hline
\end{tabular}

\vspace{3mm}

\begin{tabular}{|c|c|c|c|c|c|c|c|c|c|}
\hline
connectivity &8 &9 & 10 & 11 & 12 & 13 & 14\\
\hline
$\mu_{\Gamma,2,1}$ &.00117 &.00090 &.00070 &.00058 &.00047 &.00039 & .00034  \\
\hline
\end{tabular}

\vspace{3mm}

\begin{tabular}{|c|c|c|c|c|c|c|c|c|c|}
\hline
connectivity &15 &16 &17 &18 & 19 & 20 & 21\\
\hline
$\mu_{\Gamma,2,1}$ &.00030 &.00026 &.00023 &.00021 &.00018 &.00017 &.00016  \\
\hline
\end{tabular}

\vspace{3mm}

\begin{tabular}{|c|c|c|c|c|c|c|c|c|c|}
\hline
connectivity & 22 & 23 &24 & 25 & 26\\
\hline
$\mu_{\Gamma,2,1}$ &.00014 & .00013 &.00012 &.000098 &.000097 \\
\hline
\end{tabular}

\vspace{5mm}

The case $\alpha=0$ corresponds to the algebro-geometric setting of a random real
projective hypersurface. Let $W_{n+1,t}$ be the vector space of real homogeneous
polynomials of degree $t$ in $n+1$ variables. For $f\in W_{n+1,t}$,
$V(f)$ is a real projective hypersurface in $\mathbb{P}^{n}(\R)$.
We equip $W_{n+1,t}$ with the ``real Fubini-Study" Gaussian coming from the inner
product on $W_{n+1,t}$ given by
\begin{equation}
\label{eq:<f,g> FS}
\langle f,\, g \rangle = \int\limits_{\R^{n+1}} f(x)g(x) e^{-|x|^{2}/2} dx
\end{equation}
(the choice of the Euclidian length $|x|$ plays no role ~\cite{Sa}).
This ensemble
is essentially $\Ec_{\M,0}(t)$ with $\M=(\Sc^{n},\sigma)$ the sphere with its round
metric (see ~\cite{Sa}).

Thus the laws $\mu_{\Cc,n,0}$ describe
the universal distribution of topologies of a random real projective hypersurface
in $\mathbb{P}^{n}$ (w.r.t. the real Fubini-Study Gaussian).
It is interesting to compare
this with the more familiar case of complex hypersurfaces. For those the generic (i.e. on
a Zariski open set) hypersurface is smooth connected and of a fixed topology. Over $\R$ these
hypersurfaces are very complicated and have many components. The main theorem asserts that if
``generic" is replaced by ``random" then order is restored in that the distribution of the topologies
and Betti numbers is universal, at least when $t\rightarrow\infty$.

If $n=2$ the Nazarov-Sodin constant $\beta_{2,0}$ is such that the random oval
is about $4\%$ Harnack, that is it has about $4\%$ of the maximal number of
components that it can have (~\cite{Na}, \cite{Sa}).
The measure $\mu_{\Gamma,2,0}$ gives the distribution of the connectivities of the nodal
domains of a random oval. Barnett and Jin's Montre-Carlo simulation for these yields:

\vspace{5mm}

\begin{tabular}{|c|c|c|c|c|c|c|c|c|c|}
\hline
connectivity & 1 & 2 & 3 & 4 & 5 &6 & 7 \\
\hline
$\mu_{\Gamma,2,0}$ &.94473 &.02820 &.00889 &.00437 & .00261 &.00173 &.00128 \\
\hline
\end{tabular}

\vspace{3mm}

\begin{tabular}{|c|c|c|c|c|c|c|c|c|c|}
\hline
connectivity &8 &9 & 10 & 11 & 12 & 13 & 14\\
\hline
$\mu_{\Gamma,2,0}$ &.00093 &.00072 &.00056 &.00048 &.00039 &.00034 & .00029 \\
\hline
\end{tabular}

\vspace{3mm}

\begin{tabular}{|c|c|c|c|c|c|c|c|c|c|}
\hline
connectivity &15 &16 &17 &18 & 19 & 20 & 21\\
\hline
$\mu_{\Gamma,2,0}$ &.00026 &.00025 &.00021 &.00019 &.00016 &.00014 &.00013  \\
\hline
\end{tabular}

\vspace{3mm}

\begin{tabular}{|c|c|c|c|c|c|c|c|c|c|}
\hline
connectivity & 22 & 23 &24 & 25 & 26\\
\hline
$\mu_{\Gamma,2,0}$ &.00011 & .00011 &.00009 &.00008 &.00008 \\
\hline
\end{tabular}

\vspace{5mm}

From these tables it appears that the decay rates of $\mu_{\Gamma,2,1}(\{m \})$ and
$\mu_{\Gamma,2,0}(\{m \})$ for $m$ large are power laws $m^{-\gamma}$, with $\gamma$
approximately $2.149$ for $\alpha=1$ and $2.057$ for $\alpha=0$. These are close to the universal
Fisher constant $187/91$ which governs related quantities in critical percolation ~\cite{KZ}.

\vspace{5mm}

We note that the only cases for which there is an explicit description of $H(n-1)$ are $n=2$ and $n=3$.
For $n=2$ $H(n-1)$ is a point, namely the circle $\Sc^{1}$. For $n=3$, $H(2)$ consists of all the orientable
compact surfaces and these are determined by their genus $g$, a non-negative integer. Thus $H(2)=\Z_{\ge 0}$,
and $\mu_{\Cc,3,\alpha}$ is a measure on $\Z_{\ge 0}$ which has finite mean (see \S\ref{sec:bas conventions}). It would be very
interesting to Monte-Carlo the distributions $\mu_{\Cc,3,0}$ and $\mu_{\Cc,3,1}$ and to learn more
about their profiles. The only features that we know about the Universal Laws are that they are
{\em probability measures} and that they charge every atom and in some special case that their ``means" are
finite.

\subsection{Outline of the paper}

\label{sec:muC esc outline}

We turn to an outline of the proof of Theorem \ref{thm:main thm lim meas}, and also the content of the various sections.
Most probabilistic calculations for the Gaussian Ensembles $\Ec_{\M,\alpha}(T)$ start with the covariance kernel
\begin{equation}
\label{eq:Kalpha covar f}
K_{\alpha}(T;x,y) : = \E_{\Ec}[f(x)f(y)] = \sum\limits_{\alpha T \le t_{j}\le T} \phi_{j}(x)\phi_{j}(y)
\end{equation}
(with suitable modifications if $\alpha=1$).

The behaviour of $K_{\alpha}$ as $T\rightarrow \infty$ is decisive in the analysis and it can be studied
using micro-local analysis and the wave equation on $\M\times \R$, see section \ref{sec:bas conventions}
for more details. We have
\begin{equation}
\label{eq:Kalpha asymp}
\widetilde{K_{\alpha}} (T;x,y) := \frac{1}{D_{\alpha}(T)}K_{\alpha} (T;x,y) = B_{n,\alpha}(T\cdot d(x,y)) + O\left(T^{-1}\right)
\end{equation}
uniformly for $x,y\in\M$, where $d(x,y)$ is the (geodesic) distance in $\M$ between $x$ and $y$,
\begin{equation*}
D_{\alpha}(T)=\frac{1}{\vol(\M)} \int\limits_{\M}K_{\alpha}(T;x,x)d\vol(x),
\end{equation*}
and for $w\in\R^{n}$
\begin{equation}
\label{eq:Bnalp clean}
B_{n,\alpha}(w) = B_{n,\alpha}(|w|) = \frac{1}{|A_{\alpha}|} \int\limits_{A_{\alpha}}e(\langle w,\xi  \rangle) d\xi
\end{equation}
with $$A_{\alpha}=\left\{ w:\: \alpha\le |w|\le 1 \right\}.$$
Thus for $x$ and $y$ roughly $1/T$ from each other $\widetilde{K_{\alpha}}$ is approximated by the universal
kernel $B_{n,\alpha}$, whereas of $x$ and $y$ are a bit further apart, then the correlation is small.

\vspace{3mm}

The estimate \eqref{eq:Kalpha asymp} allows one to compute local quantities for the typical $f$ in $\Ec_{\M,\alpha}(T)$,
for example, the density of critical points (see the `Kac-Rice' formula in section \ref{sec:bas conventions}),
or the $(n-1)$-dimensional volume of $V(f)$. While these are interesting and, in fact, useful for us, for example
in bounding from above the Betti numbers using Morse Theory, the topology of $V(f)$ is global and lies beyond
these purely local quantities. It was a major insight of Nazarov and Sodin ~\cite{So} that most components $c$
of $V(f)$ and $\omega$ of $\Omega(f)$ are small, that is are of diameter at most $O(1/T)$, and that the components
that are further than this scale of $1/T$ apart are (almost) independent. This simultaneously semi-localizes the
problem of the distribution of the topologies, and also explains the concentration feature that most $f$'s have
the same distribution as well as the universality.

Moreover it separates the analysis into different parts. The first being the study of the problem for the scaling limits
determined by \eqref{eq:Kalpha asymp} and \eqref{eq:Bnalp clean}. We call these the ``scale invariant models", and
they are translation invariant isotropic Gaussian fields on $\R^{n}$ determined by \eqref{eq:Bnalp clean}. We denote
them by $\mathfrak{g}_{n,\alpha}$ and review their properties in section \ref{sec:bas conventions}. Once these are
understood one has to couple the scale invariant theory with the global analysis after decomposing $\M$ into pieces
of size $1/T$.

\vspace{3mm}

In section \ref{sec:bas conventions} we review some `standard' theory such as properties of the fields $\mathfrak{g}_{n,\alpha}$,
the Kac-Rice formula and some elementary topology. Sections \ref{sec:scal invar mod}, \ref{sec:top nest not leak}
and \ref{sec:full supp} are concerned with proving the analogue of Theorem \ref{thm:main thm lim meas} for the fields
$\mathfrak{g}_{n,\alpha}$. The measure spaces $H(n-1)$ and $\Tc$ are non-compact so that weak limits of probability measures
need not be probability measures. In terms of technical novelty this issue is a central one for us. In section \ref{sec:scal invar mod}
we prove the existence as weak limits of $\mu_{\Cc,n,\alpha}$ and $\mu_{X,n,\alpha}$ associated with the
$\mathfrak{g}_{n,\alpha}$'s. The proof is relatively soft and follows closely the component counting analysis of \cite{NS} and
~\cite{So}. The main difference is that our random variables are conditioned to count a given topological type (respectively tree end).
This requires a number of modifications and extensions, especially of various inequalities
(see \S\ref{sec:Int Geom sand}, \S\ref{sec:bi(R-1,Fx)<=bi(R,fxL)<=bi(R+1,Fx)}, \S\ref{sec:glob upper bound}).

Section \ref{sec:top nest not leak} is devoted to a proof that the universal limit measures $\mu_{\Cc,n,\alpha}$ and $\mu_{X,n,\alpha}$
are in fact probability measures. This requires establishing some tightness properties for the tails of our families of measures
(see Proposition \ref{prop:grad bounded away}). The Kac-Rice formula allows us to show that most components $c$ of a typical $V(f)$
are gotten from $\Sc^{n-1}$ with a bounded number of surgeries. However this is not sufficient to control their topologies uniformly.
To limit these we examine further the geometries of the components. We show that for most components there are uniform bounds from above
for their volumes, diameter and curvatures. With this, versions of Cheeger's finiteness theorem (see Lemma \ref{lem:Cheeger}
in \S\ref{sec:proof top not leak}) allows us
to restrict their topological types and hence to establish the desired tightness property. For the case of tree ends a similar
but much simpler analysis gives the desired control on the geometry of nodal domains, yielding the corresponding tightness.

Section \ref{sec:full supp} is concerned with a proof that for the $\mathfrak{g}_{n,\alpha}$'s the limit measures $\mu_{\Cc,n,\alpha}$
and $\mu_{X,n,\alpha}$ have full support (i.e. that they charge every atom). For the cases $0\le \alpha <1$ this is straightforward
(see \S\ref{sec:liminf>0 inter prf}). The case $\alpha=1$ presents us with a second technically novel problem. We resolve it for $n=2$ and for
the measure $\mu_{X,2,1}$ in this section. This makes use of some auxiliary lemmas about approximating functions in $C(K)$, $K\subseteq \R^{2}$
compact, by solutions $u$ of $$\Delta u +u=0$$ in $\R^{2}$. This is combined with a combinatorial analysis of the zero sets of perturbations
of $$u(x_{1},x_{2})=\sin(x_{1})\cdot \sin(x_{2}).$$ The general case of the support of these measures for $n\ge 3$ is established in the companion
paper ~\cite{CS}.

\vspace{3mm}

Section \ref{sec:loc res Riem} gives `semi-local'
analysis concerning the $\mathfrak{g}_{n,\alpha}$'s with a decomposition of $\M$ into pieces at scale $1/T$ to study
the typical members of $\Ec_{\M,\alpha}(T)$. Section \ref{sec:glob res Riem}
ends with a proof of Theorem \ref{thm:main thm lim meas} by combining or `gluing' the `semi-local' pieces of $\M$.
Again we follow closely the analysis of the counting of the number
of components in ~\cite{NS} and ~\cite{So}. This requires a number of modifications and extensions
(for example see the proof of Proposition \ref{prop:bi(R-1,Fx)<=bi(R,fxL)<=bi(R+1,Fx)} in
\S\ref{sec:bi(R-1,Fx)<=bi(R,fxL)<=bi(R+1,Fx)} or the proof of Proposition \ref{prop:glob upper bound}
in \S\ref{sec:glob upper bound}), and we spell these out in some detail.

\subsection{Acknowledgements}

We would like to thank Mikhail Sodin
for sharing freely early versions of his work with Fedor Nazarov and
in particular for the technical discussions with one of us (Wigman) in Trondheim 2013,
and Ze\'{e}v Rudnick for many stimulating discussions.
In addition I.W. would like to thank Dmitri Panov and Yuri Safarov for sharing their expertise on
various topics related to the proofs.
We also thank Alex Barnett for carrying out the numerical experiments
connected with this work and for his figures which we have included, as well
as P. Kleban and R. Ziff for formulating and examining the ``holes of
clusters" in percolation models. We are grateful to Nalini Anantharaman, Yaiza Canzani and Curtis McMullen
for their valuable comments on some earlier versions of this manuscript
and to S. Weinberger for clarifying issues of decidability of topological types in higher dimensions.
Finally, it is
a pleasure to thank the anonymous referee for his careful reading and constructive criticism on
an earlier version of this manuscript, that, in particular, improved its readability significantly.

\section{Basic conventions}

\label{sec:bas conventions}

We review some material that will be used in the text. We begin with a quantitative local Weyl law for $\M$.

\subsection{Quantitative local Weyl law}

The modern treatments of Weyl's law with remainder involve construction of a parametrix for the wave equation on $\M\times\R$
as developed first in ~\cite{Lax} and ~\cite{Horm}. For the spectral window that we treat a parametrix for a small fixed time
interval suffices. The recent papers ~\cite{CH} and ~\cite{CH2}, which we will use as a reference, go beyond what we need in that
they allow $\eta(T)$ to be bounded in the case $\alpha=1$. Their goal is a remainder term of $o(T^{n-1})$, and for that
they assume some properties of the geodesic flow. Since we assume that $\eta(T)\rightarrow\infty$ (in fact, that
$\eta(T)=T^{\beta}$ for some $0<\beta<\frac{1}{2}$) and we are content with a bound of $O(T^{n-1})$ for
the remainder, the analysis from ~\cite{CH} and ~\cite{CH2} is simpler, and we don't need to impose any conditions on
$\M=(\Sc^{n},g)$.

Specifically, as pointed out in equation (5) of ~\cite{CH2}, the remainder $R(x,y,T)$ (our $T$ is their
$\lambda$) in their equation (4) is $O(T^{n-1})$, and this is proved without any assumptions on the geodesic flow (see
Theorem $4.4$ of ~\cite{Horm}). In their analysis of the main term in (4) leading to their Theorem $1$, there is a parameter $\epsilon$,
which they allow to go to zero, and which makes use of the non self-focal condition. If we fix $\epsilon$ a large constant
so that the parametrix constructed in their section $2$ is used only for $|t|< \frac{1}{\epsilon}$, then no condition on the geodesic
flow are used and one obtains their Theorem $1$ with the weaker $O(1)$ replacing their $o(1)$. This leads to:
uniformly for $x,y\in\M$
\begin{equation*}
\begin{split}
K_{T}(x,y) &:= \sum\limits_{t_{j}\le T} \phi_{j}(x)\phi_{j}(y)
\\&=T^{n}B_{n}(T\cdot d(x,y)) + O(T^{n-1}),
\end{split}
\end{equation*}
where
\begin{equation*}
B_{n}(w)=B_{n}(|w|) = \int\limits_{|\xi|\le 1}e(\langle w,\xi \rangle ) d\xi.
\end{equation*}

Hence if $0\le T'<T$,
\begin{equation}
\label{eq:KT-KT' asymp}
\begin{split}
K_{T}(x,y) - K_{T'}(x,y) &=\sum\limits_{T'\le t_{j}\le T} \phi_{j}(x)\phi_{j}(y)
\\&=T^{n}B_{n,T'/T}(T\cdot d(x,y)) + O(T^{n-1}),
\end{split}
\end{equation}
where
\begin{equation*}
B_{n,\gamma}(w)= \int\limits_{\gamma\le |\xi|\le 1}e(\langle w,\xi \rangle ) d\xi.
\end{equation*}
In particular if $0\le \alpha<1$ is fixed then we obtain the desired Weyl asymptotics that
we will use for the $\Ec_{\M,\alpha}(T)$'s.

For the monochromatic case $\alpha=1$, $$T'/T = 1-\eta(T)/T=1-T^{\beta-1}, $$ and hence
\begin{equation}
\label{eq:KT-KT' asymp alpha=1}
\begin{split}
&K_{T}(x,y) - K_{T'}(x,y)
\\&=n\eta(T)T^{n-1}B_{n,1}(T\cdot d(x,y)) + O\left(T^{n-1}+\eta^{2}T^{n-2}(Td(x,y))\right),
\end{split}
\end{equation}
where
\begin{equation*}
B_{n,1}(w)=B_{n,1}(|w|)= \int\limits_{|\xi|=1}e(\langle w,\xi \rangle ) d\nu_{1}(\xi),
\end{equation*}
and $\nu_{1}$ is the spherical Haar measure on $\Sc^{n-1}$.
We need also to control the derivatives of $K_{T}(x,y)$ w.r.t. $x$ and $y$ when $x$ and $y$ are very close
(within $1/T$ of each other). Using geodesic normal coordinates about a point $x_{0}$ in $\M$ and
the exponential map from $T_{x_{0}}(\M)$ identified with $(\R^{n},0)$ to $\M$, ~\cite{CH} show say in the case $\alpha=1$
(which is the most difficult one) that
\begin{equation}
\label{eq:KT-KT' asymp der}
K_{T} \left( \exp_{x_{0}}\left(\frac{u}{T}\right),\exp_{x_{0}}\left(\frac{v}{T}\right)   \right) = n\eta T^{n-1}B_{n,1}(u-v)+O(T^{n-1}),
\end{equation}
and this holds uniformly together with any fixed number of derivatives w.r.t. $u$ and $v$. They establish this
for $\M$ real analytic in ~\cite{CH}, and for the general $C^{\infty}$-case in ~\cite{CH2}.

Our discussion above
does not apply directly to the interesting special case of $\M$ being the standard round sphere $\Sc^{n}$ and $\Ec_{\M,1}(T)$
being the space of spherical harmonics of a given degree. In that case the classical Mehler-Heine asymptotics for
Gegenbauer polynomials yield a version of \eqref{eq:KT-KT' asymp der}, and all the results in this paper apply equally well to these
``monochromatic waves".

\subsection{Scale invariant fields}

\label{sec:scal invar gfr}

The semi-classical approximations \eqref{eq:KT-KT' asymp}, \eqref{eq:KT-KT' asymp alpha=1}
and \eqref{eq:KT-KT' asymp der} lead to the study of the scaled Gaussian fields
at a point $x_{0}\in\M$, and these in turn to the scale invariant Gaussian fields $\gfr_{n,\alpha}$ on $\R^{n}$.
They are defined as follows.

Let $\psi_{j}$, $j=1,2,\ldots$ be a real valued o.n.b. of
$L^{2}(A_{\alpha},d\nu_{\alpha})$ with $0\le \alpha<1$ and $L^{2}(\Sc^{n-1},d\nu_{1})$ if $\alpha=1$
(here $A_{\alpha}$ is the annulus $\alpha \le |\xi|\le 1$ and $d\nu_{\alpha}$ its Haar measure).
Set
\begin{equation*}
\widehat{\psi_{j}} = \int\limits_{\R^{n}} \psi_{j}(\xi)e(\langle x,\xi\rangle)d\nu_{\alpha}:=C_{j}(x)+S_{j}(x)
\end{equation*}
to be the real valued cosine and sine transforms. Define a random $f(x)$ by
\begin{equation}
\label{eq:f sim sum ajCj+bjSj}
f(x)\sim \sum\limits_{j=1}^{\infty} (a_{j}C_{j}(x)+b_{j}S_{j}(x) ) ,
\end{equation}
where $a_{j}, b_{j}$ are independent $N(0,1)$ Gaussian variables. The space of such $f$'s is denoted by
$G_{n,\alpha}$, and the corresponding probability measure on the measurable subsets of $G_{n,\alpha}$
by $P_{n,\alpha}$. The Gaussian fields $\gfr_{n,\alpha}$ are then $(G_{n,\alpha},P_{n,\alpha})$.

From the definitions one checks that for $x,y\in\R^{n}$,
\begin{equation}
\label{eq:reprod kern covar}
\begin{split}
&\sum\limits_{j=1}^{\infty} \left( C_{j}(x)C_{j}(y)+S_{j}(x)S_{j}(y)   \right) =
\int\limits_{\R^{n}} e(\langle x-y, \xi\rangle)d\nu_{\alpha}(\xi)\\&=\widehat{\nu_{\alpha}}(x-y):=\cov_{n,\alpha}(x-y)
:=\E_{f}[f(x)\cdot f(y)].
\end{split}
\end{equation}
It follows that $\gfr_{n,\alpha}=(G_{n,\alpha},P_{n,\alpha})$ is a translation and rotation invariant random field (~\cite{Adler-Taylor,AW}).
That is, if $B$ is a (measurable) subset of $G_{n,\alpha}$, then
$$P_{n,\alpha}(B) = P_{n,\alpha}(RB)$$ for any rigid motion $R$ of $\R^{n}$ (here $Rf(x):=f(Rx)$). Moreover, since $\widehat{\nu_{\alpha}}(x)$
is analytic in $x$, almost all $f$'s in $G_{n,\alpha}$ represent analytic functions on $\R^{n}$. This follows for example from the series
\eqref{eq:f sim sum ajCj+bjSj} converging uniformly on compacta in $x$ for a.a. choices of the Gaussian variables $a_{j}$, $b_{j}$, see
Appendix \ref{apx:measurability} for more details.

Of special interest is $\alpha=1$ for which one can choose as o.n.b. of $L^{2}(\Sc^{n-1},\nu_{1})$ the special harmonic $Y_{m}^{l}$,
$l\ge 0$ and $m=1,\ldots d_{l}$. A computation ~\cite{CS} shows that the ensemble $\gfr_{n,1}$ has a representation
\begin{equation*}
f(x)\sim (2\pi)^{n/2}\sum\limits_{l=0}^{\infty}\sum\limits_{m=1}^{d_{l}}b_{l,m}Y_{m}^{l}\left(\frac{x}{|x|}\right) \frac{J_{l+\nu}(|x|)}{|x|^{\nu}}
\end{equation*}
with $$\nu = \frac{n-2}{2}$$ and $b_{l,m}$ independent $N(0,1)$ Gaussians.

\subsection{Kac-Rice}

To illustrate the use of the covariance and its asymptotic approximation above, we review the computation of expected zero volume
and or critical points number for random fields. These are computed using the ``Kac-Rice" formula.
Let $m\le n$, $H:\Dc\rightarrow\R^{m}$ be a smooth random field
on a domain $\Dc\subseteq \R^{n}$, and $\Zc(H;\overline{\Dc})$ be either the $(n-m)$-volume of $H^{-1}(0)$
(for $m<n$), or the number of the discrete zeros (for $m=n$).

We set $J_{H}(x)$ to be the (random) Jacobi matrix of $H$ at $x$
and define the ``zero density" of $H$ at $x\in \Dc$ as the conditional Gaussian expectation
\begin{equation}
\label{eq:K1 density def}
K_{1}(x) = K_{1;H}(x)=\varphi_{H(x)}(0)\cdot \E[|\det J_{H}(x)| \big| H(x)=0],
\end{equation}
where $\varphi_{H(x)}(0)$ is the probability density function of $H(x)$ evaluated at $0\in \R^{m}$.
With the notation above the Kac-Rice formula (meta-theorem) states that, under some non-degeneracy condition on $H$,
$$\E[\Zc(H;\Dc)]=\int\limits_{\Dc}K_{1}(x)dx.$$ Concerning the sufficient conditions that guarantee that \eqref{eq:K1 density def} holds,
a variety of results is known ~\cite{Adler-Taylor,AW}. The following version of Kac-Rice merely requires the non-degeneracy
of the values of $H(x)$ (vs. the non-degeneracy of $(H(x),J_{H}(x))$ in the appropriate sense, as in
the other sources), to our best knowledge, the mildest sufficient condition.

\begin{lemma}[Standard Kac-Rice ~\cite{AW}, Theorem $6.3$]
\label{lem:Kac-Rice precise}
Let $H:\Dc\rightarrow\R^{m}$ be an a.s. smooth Gaussian field, such that for every $x\in \Dc$ the distribution of
the random vector $H(x)\in \R^{m}$ is non-degenerate Gaussian. Then
\begin{equation}
\label{eq:E[ZH]=int K1}
\E[\Zc(H;\Dc)]=\int\limits_{\Dc}K_{1}(x)dx
\end{equation}
with the zero density $K_{1}(x)$ as in \eqref{eq:K1 density def}.
\end{lemma}


The Gaussian density \eqref{eq:K1 density def} is a Gaussian integral that, in principle, could be evaluated explicitly.
However, in practice, it is not easy to control $K_{1}(x)$ uniformly in both $x$ and the field $H$.
The following lemma was
proved in ~\cite{So}; it offers an easy and explicit upper bound for the discrete number of zeros (in case $m=n$),
uniformly w.r.t. $H$.

\begin{lemma}[{\cite[Lemma $2$]{So}}]
\label{lem:Kac-Rice upper bnd gen}
Let $m\ge 1$, $B\subseteq \R^{m}$ a ball and
$H:\overline{B}\rightarrow \R^{m}$ an a.s. $C^{1}$-smooth random Gaussian field. Then we have
\begin{equation}
\label{eq:Kac-Rice upper bnd gen}
\E[\Zc(H;\overline{B})] \le C\sup\limits_{x\in\overline{B}}\frac{\E[|\nabla H(x)|^{2}]^{m/2}}{(\det\E[H(x)\cdot H(x)^{t}])^{1/2}}\cdot \vol(B).
\end{equation}
for some universal constant $C=C(m)>0$, where $$|\nabla H(x)|^{2} = \sum\limits_{i,j}|\partial_{i}H_{j}|^{2}$$ is the Hilbert-Schmidt norm.
\end{lemma}

Note that both the denominator and the numerator of the r.h.s. of \eqref{eq:Kac-Rice upper bnd gen}
may be expressed in terms of the covariance matrix $(r_{ij}(x,x))_{ij}$ and its second mixed derivatives
$(\partial_{lm}r_{ij}(x,x))_{ijlm}$ evaluated on the diagonal $x=y$. If $H$ is stationary, then
the fraction on the r.h.s. of \eqref{eq:Kac-Rice upper bnd gen} does not depend on $x$,
so that in the latter case \eqref{eq:Kac-Rice upper bnd gen} is
\begin{equation}
\label{eq:Kac-Rice upper bnd stat}
\E[\Zc(H;\overline{B})] \le C(H)\cdot \vol(B)
\end{equation}
with $C(H)$ expressed in terms of the covariance of $H$
and a couple of its derivatives at the origin $x=0$.

We now apply Lemma \ref{lem:Kac-Rice upper bnd gen} for counting critical points of a given stationary field $F$
by using $H=\nabla F$.

\begin{corollary}[Kac-Rice upper bound]
\label{cor:Kac Rice crit ball sphere}

Let $\Dc\subseteq\R^{m}$ be a domain and $F:\Dc\rightarrow\R^{m}$ an a.s. $C^{2}$-smooth stationary Gaussian random field,
such that for $x\in \Dc$ the distribution of $\nabla F(x)$ is non-degenerate Gaussian.

\begin{enumerate}

\item For $r>0$ let $\Ac(F;r)$ be the number of critical points of $F$ inside $B(r)\subseteq\Dc$. Then
$$\E[\Ac(F;r)] = O(\vol(B(r))),$$ where the constant involved in the `$O$'-notation depends on the law of $F$ only.

\item For $r>0$ let $\widetilde{\Ac}(F;r)$ be the number of critical points of
the restriction $F|_{\partial B(r)}$ of $F$ to the sphere $\partial B(r)\subseteq\Dc $. Then
$$\E[\widetilde{\Ac}(F;r)]= O(\vol(\partial B(r))),$$
where the constant involved in the `$O$'-notation depends on the law of $F$ only.

\end{enumerate}

\end{corollary}

Note that the total number $\nod_{\Cc}(F;r,u) $ of nodal components of $F$ lying in $B_{u}(r)$ is bounded by
the number $\Ac(F;r,u)$ of critical points of $F$ in $B_{u}(r)$. Hence Corollary \ref{cor:Kac Rice crit ball sphere}
allows us to control the expected number of the former by the volume of the ball $B(r)$
(bearing in mind the stationarity of $F$). Similarly, the second part of Corollary \ref{cor:Kac Rice crit ball sphere}
allows us to control the expected number of nodal components intersecting $\partial B_{u}(r)$ in terms of
the volume of $\partial B_{u}(r)$. This approach is pursued in \S\ref{sec:scal invar exist proof}.

\begin{proof}[Proof of Corollary \ref{cor:Kac Rice crit ball sphere}]

The first part of Corollary \ref{cor:Kac Rice crit ball sphere} is merely an
application\footnote{Alternatively, it follows directly from Lemma \ref{lem:Kac-Rice precise}.} of
\eqref{eq:Kac-Rice upper bnd stat} (following from Lemma \ref{lem:Kac-Rice upper bnd gen}) on the stationary field $H=\nabla F$.
For the second part we decompose the sphere $\partial B(r)$ into (universally) finitely many coordinate patches, thus
reducing the problem to the Euclidian case, and apply Lemma \ref{lem:Kac-Rice upper bnd gen} on
the restrictions of the gradient of $F|_{\partial B(r)}$ on each of the coordinate patches separately.
Note that the total volume is of the same order of magnitude $r^{n-1}$ as $\partial B(r)$,
so that the second statement of the present corollary follows from summing up the individual estimates
\eqref{eq:Kac-Rice upper bnd gen}, bearing in mind that upon passing to the Euclidian coordinates we are losing stationarity
of the underlying random field (though the non-degeneracy of the gradient stays unimpaired).

\end{proof}

\subsection{Some remarks on topology of $V(f)$}

\label{sec:top V(f) rem}

We end this background material section with some elementary remarks about the topology of $V(f)$. For the random
$f\in\Ec_{T,\alpha}(\M)$ (and $T$ large) a component $c$ of $V(f)$ is a smooth hypersurface in $\Sc^{n}$; hence it
can be embedded in $\R^{n}$ and gives a point in $H(n-1)$. It is known that $c$ separates $\Sc^{n}$ into two connected
components ~\cite{E.Lima}. From this it follows that the nesting graph $X(f)$ is a tree and that
\begin{equation*}
|X(f)|=|\Omega(f)|=|E(X(f))|+1=|\Cc(f)|+1.
\end{equation*}

The mean of the connectivity measure $\mu_{\Gamma(f)}$ from \eqref{eq:muGamma def} is equal to
\begin{equation*}
\sum\limits_{m=1}^{\infty}m\cdot \mu_{\Gamma(f)}(m) = \frac{1}{|\Omega(f)|} \sum\limits_{v\in X(f)} d(v),
\end{equation*}
where $d(v)$ is the degree of $v$. Now
\begin{equation*}
2|E(X(f))| = \sum\limits_{v\in X(f)} d(v),
\end{equation*}
and hence
\begin{equation*}
\sum\limits_{m=1}^{\infty}m\cdot \mu_{\Gamma(f)}(m) = \frac{2|E(X(f))|}{|\Omega(f)|} = 2-\frac{2}{|\Omega(f)|}.
\end{equation*}
It follows that the means of the universal domain connectivity measures $\mu_{\Gamma,n,\alpha}$ are at most $2$.
We do not know whether these are equal to $2$ or not.

The proof that the universal Betti measures $\mu_{\text{Betti},\alpha,n}$ (page \pageref{pag:Betti})
has finite (total) mean
also follows from a finite individual bound. If $\M$ is $\Sc^{n}$ with its round metric, the eigenfunctions are
spherical harmonics and any element of $\Ec_{\M,\alpha}(T)$ is a homogeneous polynomial of degree $t=t(T)$.
According to ~\cite{Milnor} the total Betti number of the full zero set $V(f)$ of $f$
(which we can assume is nonsingular) is at most $t^{n}$. Nazarov and Sodin's Theorem \eqref{eq:|C(f)| beta c vol T^n}
asserts that the number of connected components of a typical $V(f)$ for $f$ in $\Ec_{\M,\alpha}(T)$ is
$c_{n,\alpha}\cdot t^{n}$, from which the finiteness of the total mean of the $\mu_{\text{Betti},\alpha,n}$
on $(\Z_{\ge 0})^{k}$ follows.

\section{Scale-invariant model}

\label{sec:scal invar mod}

\subsection{Statement of the main result}

Let $F:\R^{n}\rightarrow \R$ be an a.s. smooth stationary Gaussian random field. Here the relevant
limit is considering the restriction $F|_{B(R)}$ of $F$ to the centred radius-$R$ ball
$B(R)$, and taking $R\rightarrow \infty$.
The covariance function of $F$ is $r_{F}:\R^{n}\times \R^{n}\rightarrow \R$,
defined by the standard abuse of notation as
\begin{equation*}
r_{F}(x,y)=r_{F}(x-y)=\E[F(x)F(y)],
\end{equation*}
and the spectral measure (density) $d\rho_{F}$ is the Fourier transform of $r_{F}$ on $\R^{n}$.

\begin{notation}
\label{not:counting vars}

Let $\Gamma\subseteq\R^{n}$ be a (deterministic) smooth hypersurface and $R>0$ a (large) parameter.

\begin{enumerate}

\item For $H \in H(n-1)$ let $\nod_{\Cc}(\Gamma,H;R)$ be the number of connected components of $\Gamma$ lying
entirely in $B(R)$, diffeomorphic to $H$.

\item For $c\in \Cc(\Gamma)$ let $e(c)$ be the rooted subtree of $X(\Gamma)$ cut by $c$,
with vertices corresponding to domains $\omega\in\Omega(\Gamma)$ lying {\em inside} $c$.

\item For $G\in \Tc$ let $\nod_{X}(\Gamma,G;R)$ be the number of edges $c\in\Cc(f)$ in the nesting tree of $\Gamma$,
corresponding to components $c\in \Cc(\Gamma)$ lying entirely in $B(R)$ with $e(c)$ isomorphic to $G$.

\item We use the shorthand $$\nod_{\cdot}(F;\cdot,\cdot):=\nod_{\cdot}(F^{-1}(0),\cdot,\cdot)$$ in either of the cases above.

\end{enumerate}

\end{notation}

Our principal result of this section (Theorem \ref{thm:scal invar count} below)
asserts that, under some assumptions on $F$, as $R\rightarrow\infty$, the numbers
$$\nod_{\cdot}(F;\cdot,R),$$ suitably normalized, converge {\em in mean} (i.e. in $L^{1}$).

\begin{definition}[Axioms on $F$]
\label{def:rho1-4 axioms}

\begin{enumerate}

\item [$(\rho 1)$] The measure $d\rho$ has no atoms.

\item [$(\rho 2)$] For some $p>6$, $$\int\limits_{\R^{n}}\|\lambda\|^{p}d\rho(\lambda) < \infty.$$

\item [$(\rho 3)$] $\supp \rho$ does not lie in a linear hyperplane.

\end{enumerate}

\end{definition}

Axiom $(\rho 2)$ implies \cite[page 30]{AW} that $F$ is a.s. $C^{3}$-smooth, and axiom $(\rho 3)$
implies that the distribution of $(F(x),\nabla F(x))$ is a.s. non-degenerate.
Finally, axiom $(\rho 1)$ guarantees that
the action of translations (or shifts) on $F$ are {\em ergodic},
which is a crucial ingredient in Nazarov-Sodin theory
(see Theorem \ref{thm:ergodicity} below).
From \eqref{eq:reprod kern covar} it is clear that axioms $(\rho 1)$, $(\rho 2)$ and $(\rho 3)$ hold for the
$\gfr_{n,\alpha}$'s considered in \S\ref{sec:scal invar gfr}.

For this {\em scale-invariant} model Nazarov and Sodin proved \cite[Theorem $1$]{So}
that under axioms $(\rho 1)-(\rho 3)$ on $F$ there exists
a constant $\beta=\beta(F)$ such that
\begin{equation}
\label{eq:NS L1 conv scal invar}
\E\left[\left|\frac{\nod_{\Cc}(F;R)}{\vol(B(R))} -\beta\frac{\omega_{n}}{(2\pi)^{n}}\right|\right] \rightarrow 0,
\end{equation}
and in particular, for every $\epsilon>0$
\begin{equation}
\label{eq:NS scal invar}
\prob\left\{\left|\frac{\nod_{\Cc}(F;R)}{\vol(B(R))} -\beta\frac{\omega_{n}}{(2\pi)^{n}}\right|>\epsilon\right\} \rightarrow 0,
\end{equation}
and gave some sufficient conditions on $F$ for the positivity of $\beta$.
The following theorem refines the latter result;
it will imply the existence of the limiting measures in Theorem \ref{thm:scal invar no leak}, part \eqref{it:conv prob} below.

\begin{theorem}[cf. {\cite[Theorem $1$]{So}}]
\label{thm:scal invar count}

Let $F:\R^{n}\rightarrow\R$ be a random field whose spectral density $\rho$
satisfies the axioms $(\rho 1)-(\rho 3)$ above.
Then for every $H\in H(n-1)$ and $G\in \Tc$
there exist constants, $c_{\Cc}(H)=c_{\Cc;F}(H) = c_{\Cc;\rho}(H)$ and
$c_{X}(G)=c_{X;F}(G) = c_{X;\rho}(G)$ so that as $R\rightarrow\infty$,
\begin{equation}
\label{eq:cnt scal invar lim}
\begin{split}
&\E\left[\left| \frac{\nod_{\Cc}(F,H;R)}{\vol(B(R))} - c_{\Cc;F}(H) \right|\right]\rightarrow  0,\\
&\E\left[\left| \frac{\nod_{X}(F,G;R)}{\vol(B(R))} - c_{X;F}(G) \right|\right]\rightarrow  0  .
\end{split}
\end{equation}

\end{theorem}

The statement \eqref{eq:cnt scal invar lim} is to say that, as $R\rightarrow\infty$, we have the limits
\begin{equation*}
\frac{\nod_{\cdot}(F,\cdot;R)}{\vol(B(R))}\rightarrow c_{\cdot;F}(\cdot)
\end{equation*}
in $L^{1}$. Using the same methods as in the present manuscript (and ~\cite{So})
it is possible to prove that these limits also valid a.s.; however, we were not able to infer the analogues of the latter statement
for the Riemannian case \eqref{eq:f = sum cj phij}.
The rest of the present section is dedicated to the proof of Theorem \ref{thm:scal invar count}, eventually given in
\S\ref{sec:scal invar exist proof}.

\subsection{Integral-Geometric sandwiches}

\label{sec:Int Geom sand}

Let $\tau_{u}$ be
the translation operator $$\tau_{u}:\R^{n}\rightarrow\R^{n}$$
acting on (random) functions, by $(\tau_{u}F)(x)=F(x-u)$.
More precisely, we consider a Gaussian random field $F:\R^{n}\rightarrow\R^{n}$ as a probability space
$(\Delta=C(\R^{n}),\mathscr{A},\prob)$, where $\mathscr{A}$ is the $\sigma$-algebra
generated by the cylinder sets of the form $$ \{f\in \Delta :\: f(x_{j})\in A_{j}:\: j=1,\ldots k \}$$ with some
$x_{j}\in \R^{n}$, and intervals $A_{j}\subseteq \R$, $j=1,\ldots, k$,
and $\prob$ the corresponding Gaussian measure, as prescribed by Kolmogorov's Theorem. Under axiom $(\rho 2)$, $\prob$
is supported on the smooth functions (e.g. $C^{1}(\R^{n})$).

In this section we reduce the various nodal counts into a purely ergodic
question; the latter is addressed using the following result
(after Wiener, Grenander-Fomin-Maruyama, see ~\cite[Theorem 3]{So}
and references therein):

\begin{theorem}
\label{thm:ergodicity}

\begin{enumerate}

\item Let $F$ be a random stationary Gaussian field with spectral measure $d\rho$. Then
if $d\rho$ contains no atoms, the action of the translations group
$$(\tau_{u}F)(x) = F(x-u)$$ is ergodic (``$F$ is ergodic").

\item Suppose that $F$ is ergodic, and the translation map $\R^{n}\times C(\R^{n})\rightarrow C(\R^{n})$
\begin{equation}
\label{eq:translation 2var}
(u,F)\mapsto \tau_{u}F
\end{equation}
is measurable w.r.t. the product $\sigma$-algebra $\mathscr{B}(\R^{n})\times \mathscr{A}$ and $\mathscr{A}$.
Then every random variable $\Phi(F)$ with finite expectation $\E[|\Phi(F)|]<\infty$
satisfies
\begin{equation*}
\lim\limits_{R\rightarrow\infty}\frac{1}{\vol(B(R))}\int\limits_{B(R)}\Phi(\tau_{u}F)du \rightarrow \E[\Phi(F)],
\end{equation*}
convergence a.s. and in $L^{1}$.

\end{enumerate}

\end{theorem}

To reduce the nodal counting questions into
an ergodic question we formulate the Integral-Geometric Sandwich below. To present it
it we need the following notation first.

\begin{notation}
\label{not:counting vars intsect}

\begin{enumerate}

\item For $u\in \R^{n}$, $\Gamma$ a smooth hypersurface
let $$\nod_{\cdot}(\Gamma,\cdot;R,u):=\nod_{\cdot}(\tau_{u}\Gamma,\cdot;R),$$ i.e.
the centred ball $B(R)$ in Notation \ref{not:counting vars} is replaced by
$B_{u}(R)$, and use the shortcut $$\nod_{\cdot}(F,\cdot;R,u):=\nod_{\cdot}(F^{-1}(0),\cdot;R,u).$$

\item For each of the (random or deterministic) variables $\nod_{\cdot}(\cdots)$
already defined, $\nod_{\cdot}^{*}(\cdots)$ is defined along the same lines with ``lying
entirely in $B(R)$" replaced by ``intersects $B(R)$".

\end{enumerate}

\end{notation}

\begin{remark}
Our approach will eventually show that the difference between
$$\nod_{\cdot}(F,\cdot;R,u)$$ and $$\nod_{\cdot}^{*}(F,\cdot;R,u)$$ is typically {\em negligible}
(see the proof of Theorem \ref{thm:scal invar count} below).
It will imply ``semi-locality", i.e. that ``most" of the nodal components (resp.
tree ends) of $F$ are contained in balls of sufficiently big radius $R\gg 1$;
bearing in mind the natural scaling these correspond to
nodal components (resp. tree ends) of the band-limited functions
$f$ in \eqref{eq:f = sum cj phij} lying in radius-$R/T$ geodesic balls on $\M$
(see also Lemma \ref{lem:E[ND-long/L^{n}]->0}).
\end{remark}

\begin{lemma}[cf. {\cite[Lemma $1$]{So}}]
\label{lem:int geom sandwitch euclidian}
Let $\Gamma$ be a (deterministic) closed hypersurface. Then for $0<r<R$,
$H\in H(n-1)$, $G\in \Tc$, one has
\begin{equation}
\label{eq:geom int scal invar}
\begin{split}
&\frac{1}{\vol (B(r))}\int\limits_{B(R-r)} \nod_{\cdot}(\Gamma,\cdot; r,u)du
\le \nod_{\cdot}(\Gamma,\cdot;R) \\&\le \frac{1}{\vol (B(r))}\int\limits_{B(R+r)} \nod_{\cdot}^{*}(\Gamma,\cdot; r,u)du
\end{split}
\end{equation}
with $\nod_{\cdot}(\Gamma,\cdot;R)$ standing for $\nod_{\Cc}(\Gamma,H;R)$
(resp. $\nod_{X}(\Gamma,G;R)$),
and similarly for the other counts involved in \eqref{eq:geom int scal invar}.
That is, \eqref{eq:geom int scal invar} asserts $2$ different sandwich inequalities corresponding
to $\nod_{\Cc,X}(\cdots)$ respectively.
\end{lemma}

\begin{proof}

We are only going to prove the sandwich inequality corresponding to $\nod_{\Cc}$, namely that
for $H \in H(n-1) $
\begin{equation}
\label{eq:geom int scal invar comp}
\begin{split}
&\frac{1}{\vol (B(r))}\int\limits_{B(R-r)} \nod_{\Cc}(\Gamma,H;r,u)du
\le \nod_{\Cc}(\Gamma,H;R) \\&\le \frac{1}{\vol (B(r))}\int\limits_{B(R+r)} \nod_{\Cc}^{*}(\Gamma,H;r,u)du,
\end{split}
\end{equation}
with the inequality for $\nod_{X}$ proven along similar lines.
Let $\gamma\subseteq B(R)$ be a connected component of $\Gamma$. Put
\begin{equation*}
G_{*}(\gamma) = \bigcap\limits_{v\in\gamma} B(v,r) = \{u:\gamma\subseteq B(u,r) \}
\end{equation*}
and
\begin{equation*}
G^{*}(\gamma) = \bigcup\limits_{v\in\gamma} \overline{B}(v,r) =
\{u:\gamma \cap \overline{B(u,r)}\ne \emptyset \}.
\end{equation*}
We have for every $v\in \gamma$, $$G_{*}(\gamma) \subseteq B_{v}(r) \subseteq G^{*}(\gamma),$$ and thus, in particular,
\begin{equation}
\label{eq:G*<=B(r)<=G*}
\vol(G_{*}(\gamma)) \le \vol(B(r)) \le \vol(G^{*}(\gamma)).
\end{equation}
Summing up \eqref{eq:G*<=B(r)<=G*} for all components $\gamma\subseteq B(R)$ diffeomorphic to $H$, we obtain
\begin{equation}
\label{eq:sum vol*bi(g)}
\begin{split}
\sum\limits_{\substack{\gamma\subseteq B(R)\\ \gamma\cong H}}\vol(G_{*}(\gamma))  &\le
\vol(B(r))\cdot \nod_{\Cc}(\Gamma,H;R) \le
\sum\limits_{\substack{\gamma\subseteq B(R)\\ \gamma\cong H}} \vol(G^{*}(\gamma)) .
\end{split}
\end{equation}
Exchanging the order of summation and the integration
\begin{equation*}
\vol(G_{*}(\gamma)) = \int\limits_{G_{*}(\gamma)}du,
\end{equation*}
we obtain
\begin{equation}
\label{eq:sum vol G_*(g)*bi(g)lower}
\begin{split}
\sum\limits_{\substack{\gamma\subseteq B(R)\\ \gamma\cong H}}\vol(G_{*}(\gamma)) &\ge
\int\limits_{B(R-r)}\left[ \sum\limits_{\substack{\gamma: \gamma\subseteq B_{u}(r)\\ \gamma\cong H}} 1 \right] du
\\&= \int\limits_{B(R-r)} \nod_{\Cc}(u,r;H,\Gamma)  du,
\end{split}
\end{equation}
since if $u\in B(R-r)$ then $B_{u}(r)\subseteq B(R)$.
Similarly,
\begin{equation}
\label{eq:sum vol G^*(g)*bi(g)upper}
\begin{split}
\sum\limits_{_{\substack{\gamma\subseteq B(R)\\ \gamma\cong S}}} \vol(G^{*}(\gamma))
&\le \int\limits_{B(R+r)}\left[ \sum\limits_{\substack{\gamma: \gamma\cap B_{u}(r) \ne \emptyset \\ \gamma\cong H}}1 \right]du
\\&= \int\limits_{B(R+r)}\nod_{\Cc}^{*}(\Gamma,H;r,u)du,
\end{split}
\end{equation}
since if $\gamma\subseteq B(R)$ and for some $u$, $$B_{u}(r)\cap \gamma \ne \emptyset,$$ then
necessarily $u\in B(R+r)$.
The statement \eqref{eq:geom int scal invar comp} of the present lemma for connected components of $\Gamma$
then follows from substituting \eqref{eq:sum vol G_*(g)*bi(g)lower}
and \eqref{eq:sum vol G^*(g)*bi(g)upper} into \eqref{eq:sum vol*bi(g)}, and dividing by $\vol B(r)$.

\end{proof}

The following lemma is instrumental for application of the ergodic theory (Theorem \ref{thm:ergodicity});
its proof will be given in Appendix \ref{apx:measurability}. Recall that, as in the beginning of section
\ref{sec:Int Geom sand}, we understand a Gaussian random field
$F$ as a probability space $(\Delta,\mathscr{A},\prob)$
consisting of a sample space $\Delta=C(\R^{n})$ of continuous functions $f(x)$, equipped with the $\sigma$-algebra $\mathscr{A}$
and the Gaussian measure $\prob$, same as above.

\begin{lemma}
\label{lem:measurability}
Let $F$ be a random field satisfying the assumptions of Theorem \ref{thm:scal invar count}.

\begin{enumerate}

\item Then for every $r>0$, $H\in H(n-1)$ and $\Tc\in X$ the maps $\nod_{\Cc}(F,H;r)$ and $\nod_{X}(F,G;r)$ are random variables
(i.e. the map $\omega \mapsto \nod_{\cdot}(F,\cdot;r)$ is measurable on the sample space $\Delta$).

\item For almost every sample point $\omega\in\Delta$, $r>0$, $H\in H(n-1)$ and $\Tc\in X$,
the function $$x\mapsto\nod_{\cdot}(F,\cdot;x,r)$$ is locally constant, and, in particular, measurable on a compact domain.

\item For every $r,R>0$, $H\in H(n-1)$ and $\Tc\in X$, the function $$(\omega,x)\mapsto\nod_{\cdot}(F,\cdot;x,r)$$ is measurable
on $\Delta\times B(R)$.

\end{enumerate}

\end{lemma}

\subsection{Proof of Theorem \ref{thm:scal invar count}: existence
of the $L^{1}$-limits for topology and nestings counts via ergodicity}

\label{sec:scal invar exist proof}

\begin{proof}

As before, we are only going to prove the result for $\nod_{\Cc}$, the proofs for $\nod_{X}$ being identical
(just replacing the counting variables below respectively).
Let $H\in H(n-1)$, and fix $\epsilon>0$ small, $r>0$ arbitrary, and apply \eqref{eq:geom int scal invar comp} on $\Gamma=F^{-1}(0)$:
\begin{equation*}
\begin{split}
&\left( 1-\frac{r}{R}\right)^{n}\frac{1}{\vol B(R-r)}\int\limits_{B(R-r)} \frac{\nod_{\Cc}(F,H;r,u)}{\vol (B(r))}du
\le \frac{\nod_{\Cc}(F,H;R)}{\vol (B(R))}
\\&\le \left( 1 + \frac{r}{R}\right)^{n}\cdot \frac{1}{\vol B(R+r)} \int\limits_{B(R+r)}
\frac{\nod_{\Cc}^{*}(F,H;r,u)}{\vol (B(r))}du
\\&\le \left( 1 + \frac{r}{R}\right)^{n}\frac{1}{\vol B(R+r)}\int\limits_{B(R+r)} \frac{\nod(F,H;r,u) +
\mathcal{I}(\tau_{u}F;r)}{\vol (B(r))}du,
\end{split}
\end{equation*}
where $\mathcal{I}(\tau_{u}F;r)$ is the total number of components $c$ of $F^{-1}(0)$ intersecting $\partial B_{u}(r)$
(of any topological class), bounded
$$\mathcal{I}(F;r,u) \le \widetilde{\Ac}(\tau_{u}F;r)$$
by the number $\widetilde{\Ac}(\tau_{u}F;r) $ of critical points of $F|_{\partial B_{u}(r)}$
(see Corollary \ref{cor:Kac Rice crit ball sphere} and the remark following it immediately),
and $$\vol(B(R\pm r)) = \vol(B(R))\cdot \left( 1\pm\frac{r}{R} \right)^{n}.$$ Recall the definition
$$\nod_{\Cc}(F,H;r,u) = \nod (\tau_{u} F,H;r),$$ where $\tau_{u}$ is the translation by $u$,
and control $r$ so that $\frac{r}{R}<\epsilon$:
\begin{equation}
\label{eq:sandwitch for N(R;S,F)}
\begin{split}
&\left( 1-\epsilon\right)\frac{1}{\vol B(R-r)}\int\limits_{B(R-r)} \frac{\nod_{\Cc}(\tau_{u}F,H;r)}{\vol (B(r))}du
\le \frac{\nod_{\Cc}(F,H;R)}{\vol (B(R))}
\\&\le \left( 1 + \epsilon\right)\frac{1}{\vol B(R+r)}\int\limits_{B(R+r)} \frac{\nod_{\Cc}(\tau_{u}F,H;r)+
\widetilde{\Ac}(\tau_{u}F;r)}{\vol (B(r))}du.
\end{split}
\end{equation}

Note that, by Corollary \ref{cor:Kac Rice crit ball sphere},
for every $r>0$, $H\in H(n-1)$, the functional
$$F\mapsto \Phi_{H;r}(F):=\frac{\nod_{\Cc} (F,H;r)}{\vol (B(r))}$$
is measurable and is of finite, uniformly bounded, expectation (i.e. $L^{1}$), and hence, by the
stationarity of $F$, so are its translations. Moreover,
the translation map \eqref{eq:translation 2var} is continuous w.r.t. the relevant $\sigma$-algebras
as in part (2) of Theorem \ref{thm:ergodicity}.
It then follows from Theorem \ref{thm:ergodicity} that both $$\frac{1}{\vol B(R+r)}\int\limits_{B(R+r)}
\frac{\nod_{\Cc} (\tau_{u}F,H;r)}{\vol (B(r))}du$$
and $$\frac{1}{\vol B(R-r)}\int\limits_{B(R-r)} \frac{\nod_{\Cc}(\tau_{u}F,H;r)}{\vol (B(r))}du$$ converge
to (the same) limit in $L^1$
$$\frac{1}{\vol B(R)}\int\limits_{B(R)} \frac{\nod_{\Cc}(\tau_{u}F,H;r)}{\vol (B(r))}du \rightarrow c(H;r):=\E[\Phi_{H;r}].$$

Observe that, if we get rid of $\widetilde{\Ac}(\tau_{u}F;r)$ from
the rhs of \ref{eq:sandwitch for N(R;S,F)}, then, up to $\pm \epsilon$, both the lhs and the rhs of
\ref{eq:sandwitch for N(R;S,F)} converge in $L^{1}$ to the same limit $c(H;r)$.
We will be able to get rid of $\widetilde{\Ac}(\tau_{u}F;r)$ for $r$ large; it
will yield that as $r\rightarrow\infty$, we have the limit $$c(H;r)\rightarrow c(H),$$
where the latter constant is the same as $$c_{\Cc;F}(H):=c(H),$$ prescribed
by Theorem \ref{thm:scal invar count}.
To justify the latter we use the same ergodic argument on $$F\mapsto\frac{\widetilde{\Ac}(\tau_{u}F;r)}{\vol (B(r))}:$$
Theorem \ref{thm:ergodicity} yields the $L^{1}$ limit
$$\frac{1}{\vol B(R+r)}\int\limits_{B(R+r)}\frac{\widetilde{\Ac}(\tau_{u}F;r)}{\vol (B(r))}\rightarrow a_{r}$$
as $R\rightarrow\infty$, with
$$a_{r} = \E\left[\frac{\widetilde{\Ac}(\tau_{u}F;r)}{\vol (B(r))}\right] = O\left(\frac{1}{r}\right),$$ by
Corollary \ref{cor:Kac Rice crit ball sphere}. Hence \eqref{eq:sandwitch for N(R;S,F)} implies
\begin{equation*}
\E \left[\left|\frac{\nod_{\Cc}(F,H;R)}{\vol (B(R))} -  c(H;r)\right|\right] = O\left(\epsilon+\frac{1}{r}\right);
\end{equation*}
the latter certainly implies the existence of the $L^{1}$-limits $$c(H)=\lim\limits_{r\rightarrow\infty}c(H;r),$$
the $L^{1}$-convergence $$\frac{\nod_{\Cc}(F,H;R)}{\vol (B(R))}\rightarrow c(H)$$ claimed by Theorem \ref{thm:scal invar count}.

\end{proof}

\section{Topology and nesting not leaking}

\label{sec:top nest not leak}

\subsection{Topology and nesting measures}


Let $F:\R^{n}\rightarrow \R$
be a stationary Gaussian random field, and $R>0$ a big parameter.
We may define the analogous measures to \eqref{eq:muCf def}
and \eqref{eq:muXf def} for $F$ and express them in terms of the counting numbers
$\nod_{\cdot}(F,\cdot;R)$ in Notation \ref{not:counting vars}:
\begin{equation}
\label{eq:muC scal invar def}
\mu_{\Cc(F);R} = \frac{1}{|\Cc(F;R) |} \sum\limits_{c\in \Cc(F;R)}\delta_{t(c)} =
\frac{1}{|\Cc(F;R) |} \sum\limits_{H\in H(n-1)}\nod_{\Cc}(F,H;R)\cdot \delta_{H}
\end{equation}
on $H(n-1)$, and
\begin{equation}
\label{eq:muX scal invar def}
\mu_{X(F);R} = \frac{1}{|\Cc(F;R) |} \sum\limits_{c\in \Cc(F;R)}\delta_{e(c)} =
\frac{1}{|\Cc(F;R) |} \sum\limits_{G\in\Tc} \nod_{X}(F,G;R)\cdot \delta_{G}
\end{equation}
on $\Tc$.

Theorem \ref{thm:scal invar no leak} below first restates Theorem \ref{thm:scal invar count} in terms of convergence of probability measures \eqref{eq:muC scal invar def} and \eqref{eq:muX scal invar def}, and then asserts that there is no mass escape to infinity so that the limiting measures are probability measures.

\begin{notation}
\label{not:scal invar lim top nest meas}

For a Gaussian field $F$ satisfying the assumptions of Theorem \ref{thm:scal invar count}, given
$H\in H(n-1)$, $G\in \Tc$ the latter
theorem yields constants $c_{\cdot;F}(\cdot)$ satisfying \eqref{eq:cnt scal invar lim}.
We may define the measure (cf. \eqref{eq:NS scal invar}):

\begin{equation}
\label{eq:mu C n alpha def scal invar}
\mu_{\Cc(F)} = \frac{(2\pi)^{n}}{\beta_{n,\alpha}\omega_{n}}
\sum\limits_{H\in H(n-1)}c_{\Cc;F}(H)\cdot \delta_{H},
\end{equation}
and similarly
\begin{equation*}
\mu_{X(F)} = \frac{(2\pi)^{n}}{\beta_{n,\alpha}\omega_{n}}
\sum\limits_{G\in \Tc}c_{X;F}(G)\cdot \delta_{G}.
\end{equation*}

\end{notation}

\begin{theorem}
\label{thm:scal invar no leak}

Let $F:\R^{n}\rightarrow\R$ be a stationary Gaussian field whose spectral density $\rho$
satisfies the axioms $(\rho 1)-(\rho 3)$ in Definition \ref{def:rho1-4 axioms}.

\begin{enumerate}

\item
\label{it:conv prob}
For every $H\in H(n-1)$, $G\in \Tc$, and $\epsilon>0$,
\begin{equation}
\label{eq:cnt lim}
\begin{split}
\prob \big\{ \max \big(&\left|\mu_{\Cc(F);R}(H)-\mu_{\Cc(F)}(H) \right|,
\\&\left|\mu_{X(F);R}(G)-\mu_{X(F)}(G) \right|
\big) > \epsilon   \big\} \rightarrow 0
\end{split}
\end{equation}
as $R\rightarrow \infty$.

\item
\label{it:top not leaking}

The limiting topology measure $\mu_{\Cc(F)}$ is a probability measure.

\item

\label{it:nest not leaking}
The limiting nesting measure $\mu_{X(F)}$ is a probability measure.

\end{enumerate}

\end{theorem}

\begin{proof}[Proof of Theorem \ref{thm:scal invar no leak}, part \eqref{it:conv prob}]

To prove the statement of \eqref{eq:cnt lim} on
$$\left|\mu_{\Cc(F);R}(H)-\mu_{\Cc(F)}(H) \right| $$ we notice that
the $L^{1}$-convergence in \eqref{eq:cnt scal invar lim} implies that for every $\epsilon>0$
\begin{equation*}
\prob \left(  \left| \frac{\nod_{\Cc}(F,H;R)}{\vol(B(R))} - c_{\Cc;F}(H) \right| > \epsilon \right) \rightarrow 0
\end{equation*}
via Chebyshev's inequality. This, together with \eqref{eq:NS scal invar} and the definition
\eqref{eq:mu C n alpha def scal invar} of $\mu_{\Cc(F)}$, finally implies the statement \eqref{eq:cnt lim}
of Theorem \ref{thm:scal invar no leak}, part \eqref{it:conv prob} for $\mu_{\Cc(F);R}$ with
the proof for $\mu_{X(F);R}$ being identical to the above.

\end{proof}

The rest of the present section is dedicated to proving the latter parts of Theorem \ref{thm:scal invar no leak},
namely that there is no escape of topology and nesting to infinity. In fact, in the course of the proof we will gain
more information on the possible geometry of typical nodal components, controlling the geometry in terms of
the gradient; in Appendix \ref{apx:top not leaking} we give a shorter proof, at the expense of using more abstract tools
(such as e.g. Monotone Convergence Theorem), and, consequently, more limited understanding of the geometry of nodal components. Note that
part \eqref{it:top not leaking} is equivalent to
\begin{equation*}
\sum\limits_{H\in H(n-1)}c_{\Cc;F}(H) = \beta_{n,\alpha}\omega_{n}(2\pi)^{-n},
\end{equation*}
and similarly for part \eqref{it:nest not leaking}.

\subsection{Proof of Theorem \ref{thm:scal invar no leak}, part \eqref{it:top not leaking}}

\label{sec:proof top not leak}

Let $F:\R^{n}\rightarrow\R$ be a stationary Gaussian field; from this point and until the end
of this section we will assume that $F$ satisfies the assumptions of Theorem \ref{thm:scal invar no leak},
namely axioms $(\rho 1)-(\rho 3)$.
The following proposition, proved in \S\ref{sec:proof grad bnd away},
states that with high probability the
gradient of $F$ is bounded away from $0$ on most of the nodal components of $F$,
i.e. with high probability $F$ is ``stable" in this sense.

\begin{proposition}[Uniform stability of a smooth Gaussian field]
\label{prop:grad bounded away}
For every $\epsilon>0$ and $\eta>0$ there exists a constant $\beta=\beta(\epsilon,\eta)>0$ so that
for $R$ sufficiently big the probability that $|\nabla F(x)|>\beta$ on all but at most $\eta R^{2}$
components of $F^{-1}(0)$ lying in $B(R)$ is $>1-\epsilon$.
\end{proposition}

The following lemma, proved in \S\ref{sec:Cheeger proof},
exploits the ``stability" of a function in the sense of Proposition \ref{prop:grad bounded away}
to yield that in this case the topology of a nodal component is essentially constrained to a finite number of topological
classes.

\begin{lemma}
\label{lem:Cheeger}
Given $\beta>0$, $C<\infty$ and $V<\infty$ there exists a finite subset $$K=K(\beta,C,V,n)\subseteq H(n-1)$$
of $H(n-1)$ with the following property.
Suppose that $G:\R^{n}\rightarrow\R$ is a (deterministic) smooth function, and $\Zf$ is a connected component of $G=0$ which is contained
in a ball $\Bc\subseteq\R^{n}$ and satisfies:
\begin{enumerate}[$(i)$]

\item For all $x\in \Zf$ $$|\nabla G(x)| \ge \beta.$$

\item The volume of $\Zf$ is $\vol_{n-1}(\Zf)\le V.$

\item The $C^{2}$ norm of $G$ on $\Bc$ is bounded $$\|G\|_{C^{2}(\Bc)}\le C.$$

\end{enumerate}
Then $\Zf\in K$.

\end{lemma}

\begin{proof}[Proof of Theorem \ref{thm:scal invar no leak}, part \eqref{it:top not leaking} assuming Proposition
\ref{prop:grad bounded away} and Lemma \ref{lem:Cheeger}]

In order to prove that there is no escape of probability we will prove that there exist $\beta,V>0$ and $C>0$ as in Lemma \ref{lem:Cheeger},
so that the expected number of components of $F$ that do not satisfy the conditions of Lemma \ref{lem:Cheeger} on a fixed Euclidian ball
is arbitrarily small. To make this precise, for a collection $A\subseteq H(n-1)$ of topology classes we define $\nod_{\Cc}(F,A;R)$
to be the number of nodal components $c\in \Cc(F)$ of $F$ lying entirely in $B(R)$ of topology class lying in $A$;
in particular for $H\in H(n-1)$ we have $$\nod_{\Cc}(F,H;R):=\nod_{\Cc}(F,\{H\};R).$$
For the limiting measure $\mu_{\Cc(F)}$ to be a probability measure it is sufficient to prove tightness:
for every $\delta>0$ there exists a {\em finite} $$A_{0}=A_{0}(\delta)\subseteq H(n-1)$$ so that
\begin{equation}
\label{eq:E[A]<epsilon}
\E[\nod_{\Cc}(F,H(n-1)\setminus A_{0};R)] < \delta\cdot R^{n};
\end{equation}
$A_{0}$ will be chosen as $A_{0}=K(\beta,C,V,n)$ with some $(\beta,C,V)=(\beta,C,V)(\delta)$ to be determined.

Now let $A\subseteq H(n-1)$ be arbitrary.
We are going to invoke the Integral-Geometric Sandwich \eqref{eq:geom int scal invar} of Lemma \ref{lem:int geom sandwitch euclidian}
again; to this end we also define $$\nod_{\Cc}^{*}(F,A;R)$$ to be the number of those components $c\in \Cc(F)$ of $F^{-1}$
of topological class in $A$ merely {\em intersecting} $B(R)$, and
$$\nod_{\Cc}^{*}(F,A;r,u):=\nod_{X}^{*}(F(\cdot+u),A;r) $$ is the number
of components as above intersecting in a $u$-centred radius-$r$ ball.
Summing up the rhs of \eqref{eq:geom int scal invar} for all $H\in A$ in this setting, we have
for every $$0<r<R$$ the upper bound
\begin{equation}
\label{eq:geom int sand Cc upper*}
\begin{split}
\nod_{\Cc}(F,A;R) &\le \frac{1}{\vol(B(r))}\int\limits_{B(R+r)}\nod_{\Cc}^{*}(F,A;r,u)du \\&\le
\frac{1}{\vol(B(r))}\int\limits_{B(R+r)}\left(\nod_{\Cc}(F,A;r,u)+\widetilde{\Ac}(\tau_{u}F;r)\right)du
\end{split}
\end{equation}
(see Corollary \ref{cor:Kac Rice crit ball sphere}).

Now we take expectation of both sides of \eqref{eq:geom int sand Cc upper*}.
Since by Corollary \ref{cor:Kac Rice crit ball sphere} and the stationarity of $F$, uniformly $$\E[\widetilde{\Ac}(\tau_{u}F;r)]=O_{r\rightarrow\infty}(r^{n-1})$$
with the constant involved in the `O`-notation depending only on $F$, given $\delta>0$,
we can choose a sufficiently big parameter $r>r_{0}(\delta)$ so that after taking the expectation \eqref{eq:geom int sand Cc upper*} is
\begin{equation}
\label{eq:geom int sand Cc upper}
\E[\nod_{\Cc}(F,A;R)] \le \left(\frac{R}{r}+1\right)^{n}\cdot \E[\nod_{\Cc}(F,A;r)] + \frac{\delta}{2}\cdot R^{n}.
\end{equation}
From \eqref{eq:geom int sand Cc upper} it is clear that in order to prove the tightness
\eqref{eq:E[A]<epsilon} it is sufficient to find a finite $A_{0}\subseteq H(n-1)$
so that $$\E[\nod_{\Cc}(F,H(n-1)\setminus A_{0};r)]<\frac{\delta}{4}r^{n}$$ is arbitrarily small;
the upshot is that $r$ is {\em fixed} (though arbitrarily big).

Take $\epsilon=\epsilon(\delta)>0$ a parameter to be chosen later.
We will now define $A_{0}$ to be of the form $$A_{0}=K(\beta,C,V,n),$$ as in Lemma \ref{lem:Cheeger} applied on $\Bc=B(r)$,
with $(\beta,C,V)$ chosen as follows:
\begin{enumerate}

\item Assuming that $r>r_{0}(\delta)$ is sufficiently big so that we may apply Proposition \ref{prop:grad bounded away}
in $\Bc$ with $\eta=\frac{\delta}{16}$ to obtain a number $\beta > 0$ so that outside of an event of probability $<\frac{\epsilon}{6}$
we have $$|\nabla F|>\beta$$ on all but at most $$\frac{\delta}{16}\cdot r^{n}$$ components in $\Bc$.

\item Since for every $r>0$ the expected $C^{2}$-norm $$\|F\|_{C^{2}(\Bc)}<\infty$$ is finite, given $r>0$ and $\epsilon>0$
we may find $C=C(r)>0$ sufficiently big so that $$\prob\{\|F\|_{C^{2}(\Bc)}>C\} < \frac{\epsilon}{6}.$$

\item Since, by Kac-Rice (Lemma \ref{lem:Kac-Rice precise}), the total expected nodal volume of $F$ inside $B(r)$ is finite,
(equals $$\E[\vol_{n-1}(F^{-1}(0)\cap \Bc)] = c\cdot r^{n}$$ with some $c=c(F)>0$), we may find
$V=V(r)$ sufficiently big so that, outside of an event of probability $<\frac{\epsilon}{6}$, the volume
of all but at most $$\frac{\delta}{16}\cdot r^{n}$$ components of $F$ on $\Bc$ is $$\vol_{n-1}(\Zf)<V$$ is smaller than $V$.

\end{enumerate}

Consolidating all above, with the $(\beta,C,V)$ just chosen, we conclude that outside an event
$\Delta_{0}\subseteq\Delta$ in the ambient probability space $\Delta$ of probability
\begin{equation}
\label{eq:prob(Delta)<eps/2}
\prob(\Delta_{0})<\frac{\epsilon}{2}
\end{equation}
we have $\|F\|_{C^{2}(\Bc)}\le C$, and also
$|\nabla F| > \beta$ and $$\vol_{n-1}(\Zf)<V$$ on all but at most $\frac{\delta}{8}\cdot r^{n}$ exceptional nodal components $\Zf$ of $F$
on $\Bc$. Hence, by Lemma \ref{lem:Cheeger}, choosing $$A_{0}=K(\beta,C,V,n),$$
the topological classes of these ``good" components are all lying in $A_{0}$.
That is, on $\Delta\setminus\Delta_{0}$,
the topologies of at most $\frac{\delta}{8}\cdot r^{n}$ nodal components of $F$ on $B(r)$, are in $H(n-1)\setminus A_{0}$; equivalently, on $\Delta\setminus\Delta_{0}$
we have the pointwise bound
\begin{equation}
\label{eq:nodA0<delta/4r^n outside Delt}
\nod_{\Cc}(F,H(n-1)\setminus A_{0};r) < \frac{\delta}{8}\cdot r^{n}.
\end{equation}

One then has ($d\Pc$ being the underlying probability measure on $\Delta$)
\begin{equation}
\label{eq:exp nodA=intDelta+int outside}
\begin{split}
&\E[\nod_{\Cc}(F,H(n-1)\setminus A_{0};r)] \\&= \int\limits_{\Delta_{0}}\nod_{\Cc}(F,H(n-1)\setminus A_{0};r)d\Pc+
\int\limits_{\Delta\setminus\Delta_{0}}\nod_{\Cc}(F,H(n-1)\setminus A_{0};r)d\Pc \\&<
\int\limits_{\Delta_{0}}\nod_{\Cc}(F,H(n-1)\setminus A_{0};r)d\Pc+
\frac{\delta}{8}r^{n},
\end{split}
\end{equation}
by the pointwise bound \eqref{eq:nodA0<delta/4r^n outside Delt} on $\Delta\setminus\Delta_{0}$.

We claim that the exceptional event $\Delta_{0}$ does not contribute significantly to the expectation
on the rhs of \eqref{eq:exp nodA=intDelta+int outside}, more precisely, that
\begin{equation}
\label{eq:exp on Delta small A0}
\int\limits_{\Delta_{0}}\nod_{\Cc}(F,H(n-1)\setminus A_{0};r)d\Pc \le \frac{\delta}{8}\cdot r^{n}
\end{equation}
for $r>r_{0}(\delta)$ sufficiently big independent of $A$.
In fact, we make the evidently stronger claim for the {\em total} number of nodal components
\begin{equation}
\label{eq:exp on Delta small total}
\int\limits_{\Delta_{0}}\nod_{\Cc}(F;r)d\Pc \le \frac{\delta}{8}\cdot r^{n},
\end{equation}
valid for $r>r_{0}(\delta)$ sufficiently big (here the independence of $A$ is self-evident), and
$\Delta_{0}$ satisfying \eqref{eq:prob(Delta)<eps/2} with $\epsilon<\epsilon_{0}(\delta)$ sufficiently small.
However, \eqref{eq:exp on Delta small total} (implying \eqref{eq:exp on Delta small A0}) follows as a simple
conclusion of Nazarov-Sodin's $L^{1}$-convergence \eqref{eq:NS L1 conv scal invar}.
The tightness \eqref{eq:E[A]<epsilon} finally follows upon substituting
\eqref{eq:exp on Delta small A0} into \eqref{eq:exp nodA=intDelta+int outside}, and then into \eqref{eq:geom int sand Cc upper}.

\end{proof}

The rest of the present section is dedicated to the proofs of Proposition \ref{prop:grad bounded away}
(\S \ref{sec:proof grad bnd away}-\S\ref{sec:prf grad bnd most comp}),
part \eqref{it:nest not leaking} of Theorem \ref{thm:scal invar no leak}
(\S\ref{sec:nest no leak}), and also Lemma \ref{lem:Cheeger} (\S\ref{sec:Cheeger proof});
an estimate on small nodal domains (Lemma \ref{lem:xi-small bound}) will be invoked by the two former of the three.

\subsection{Proof of Proposition \ref{prop:grad bounded away}: uniform stability of smooth random fields}

\label{sec:proof grad bnd away}

First we formulate a different notion of {\em stability}, and
prove that $F$ is stable with arbitrarily high probability; in the end we will prove that
a stable function necessarily satisfies the property in Proposition \ref{prop:grad bounded away}.

\begin{notation}
In what follows the letters, $c_{i}$ and $C_{i}$ will designate various positive (``universal")
constants - depending on $F$ only; $c_{i}$ and $C_{i}$
will stand for ``sufficiently small" and ``sufficiently big" constants respectively
and may be different for different lemmas.
\end{notation}

\begin{definition}
\label{def:stability}

Let $F:\R^{n}\rightarrow \R$ be a smooth random Gaussian field,
$\delta>0$, $\alpha,\beta>0$ small, $R>0$ be the (big) radius
of the ball ($R\rightarrow\infty$), and $T>0$ a sufficiently large constant.
Cover $B(R)$ with approximately $(R/T)^{n}$ balls (or squares) $\Dc_{i}$
of radius $T$ so that the covering multiplicity is bounded by a universal constant $\kappa>0$,
i.e. each point $x\in B(R)$ belongs to at most $\kappa$ of the $\{\Dc_{i}\}$. Denote $\Gc_{i}$ to be balls centred at the same points as $\Dc_{i}$, with radii $3T$. Note that
the covering multiplicity of $\{\Gc_{i}\}$ is bounded by $$c_{0}(n)\cdot \kappa.$$

\begin{enumerate}

\item  We say that $F$ is $(\alpha,\beta)$-{\em stable} on a ball $\Gc_{i}$ if it does not contain
a point $x$ with both $F(x)<\alpha$ and $\| \nabla F(x)\| < \beta$, and otherwise $F$ is
$(\alpha,\beta)$-{\em unstable} on $\Gc_{i}$.

\item We say that $F$ is
$(\delta,\alpha,\beta)$-stable if
$F$ is stable on all of $\Gc_{i}$ except for up to $\delta R^{n}$ ones.

\end{enumerate}

The $(\delta,\alpha,\beta)$-stability notion also depends on $T$, and on our covering (and hence on $R$),
but we will control all the various constants in terms of $T$ and $\kappa$ only. The parameter $T$ will be kept constant until
the very end (see Lemma \ref{lem:stable grad bnd away}), and $\kappa=\kappa(n)$ is absolute. We will suppress the dependence on the various parameters from the definition of stability if the latter are clear; typically $\delta$ will be a given small number, and $\alpha$ and $\beta$
will depend on $\delta$.

\end{definition}

\begin{proposition}
\label{prop:gen stability}
For every $\epsilon>0$ and $\delta>0$ there exist $\alpha,\beta>0$ depending on
$\epsilon,\delta$ and the law of $F$, so that $F|_{B(R)}$ is $(\delta,\alpha,\beta)$-stable
with probability $>1-\epsilon$.
\end{proposition}

\begin{lemma}
\label{lem:stable grad bnd away}
For every $\epsilon>0$ and $\eta>0$, there exist $T_{0}$, $\delta>0$,
and an event $\Ec$ of probability $\prob(\Ec)<\epsilon$
such that for all $T>T_{0}$ and $(\alpha,\beta)$ if $F\notin \Ec$ and $F$ is $(\delta,\alpha,\beta)$-stable, then
$|\nabla F(x)|>\beta$ on all but $\eta R^{n}$ components of $F^{-1}(0)$.
\end{lemma}

The proofs of Proposition \ref{prop:gen stability} and Lemma
\ref{lem:stable grad bnd away} will be given in \S\ref{sec:stability proof} and
\S\ref{sec:prf grad bnd most comp} respectively.

\begin{proof}[Proof of Proposition \ref{prop:grad bounded away} assuming Proposition
\ref{prop:gen stability} and Lemma \ref{lem:stable grad bnd away}]

Given $\epsilon>0$ and $\eta>0$ we invoke Lemma \ref{lem:stable grad bnd away}
with $\epsilon/2$ instead of $\epsilon$ to obtain $T>T_{0}$, $\delta>0$,
and the exceptional event $\Ec$ of probability $\prob(\Ec)<\epsilon/2$, as prescribed. Now
we apply Proposition \ref{prop:gen stability} with $\epsilon$ replaced by $\epsilon/2$ again,
and $\delta$ obtained as above to yield $(\alpha,\beta,\gamma)$ so that $F$ is
$(\delta,\alpha,\beta)$-stable with probability $<1-\epsilon/2$.
It then follows from the way we obtained $\delta$ as result of an application of Lemma \ref{lem:stable grad bnd away} that,
further excising $\Ec$ of probability $\prob(\Ec)<\epsilon/2$ from the stable event of probability $>1-\epsilon/2$,
the number of nodal components $\Gamma$ of $g^{-1}(0)$ failing to satisfy $|\nabla F(x) |>\beta$ is at most
$\eta\cdot R^{n}$, this occurs with probability $>1-\epsilon$.

\end{proof}

\subsection{Proof of Proposition \ref{prop:gen stability}}
\label{sec:stability proof}

We will adopt the standard notation $\partial^{\upsilon}$, $$\upsilon=(\alpha_{1},\ldots,\alpha_{n})\in \Z_{\ge 0}^{n},$$
to denote the corresponding partial derivative; $|\upsilon| = \upsilon_{1}+\ldots+\upsilon_{n}$.
We will need some auxiliary lemmas.

\begin{lemma}
\label{lem:bnd covering separated}
There exists a constant $c_{0}=c_{0}(\kappa)>0 $ depending only on $\kappa$ with
the following property.
Let $\Gc = \{ \Gc_{i}\}_{i\le K}$ be a collection of radius $3T$ balls lying
in $B(R)$ such that each point $x\in B(R)$ lies in at most $\kappa$ of them.
Then $\Gc$ contains $c_{0}\cdot K$ balls that are in addition $4$-separated.
\end{lemma}

\begin{lemma}
\label{lem:Linf bound Sobolev}
For every $\epsilon>0$ and $r\ge 0$ there exist $C_{0}=C_{0}(\epsilon)>0$ with the following property.
With probability $>1-\epsilon$, for every (possibly random) collection $\{x_{i}\}_{i\le K}$ of points satisfying $d(x_{i},x_{j})>4$ for
$i\ne j$, we may choose $K/2$ points, up to reordering $\{x_{i} \}_{i\le K/2}$, satisfying
\begin{equation*}
\sup\limits_{|\upsilon|\le r,B(x_{i},1)} |\partial^{\upsilon} F (x)| \le C_{0} \frac{R^{n/2}}{\sqrt{K}}.
\end{equation*}

\end{lemma}

Typically, $K$ is of the order of magnitude $K\approx R^{n}$; informally, in this case Lemma
\ref{lem:Linf bound Sobolev} states that in this case the derivatives around most of the points are uniformly bounded.

\begin{proof}[Proof of Proposition \ref{prop:gen stability} assuming lemmas
\ref{lem:bnd covering separated} and \ref{lem:Linf bound Sobolev}]

For a given tuple $(\alpha,\beta)$ let $K$ be the number of $(\alpha,\beta)$-unstable balls of $F$.
Our goal is to show that we may choose suitable $\alpha=\alpha(\delta)$
and $\beta=\beta(\delta)$ so that $K<\delta R^{n}$, $\delta>0$ is an arbitrarily small given number.

By Lemma \ref{lem:bnd covering separated} we may choose
\begin{equation}
\label{eq:tildK}
\tilde{K}=c_{1}\cdot K
\end{equation}
unstable balls of $F$ that are in addition $4$-separated, and up to reordering the $\Gc_{i}$
let $\{x_{i}\}_{i\le \tilde{K}}$, $x_{i}\in \Gc_{i}$ be some points satisfying
\begin{equation}
\label{eq:g(xi)<alpha,grad}
|F(x_{i})|<\alpha \text{ and }
|\nabla F(x_{i})| < \beta,
\end{equation}
as postulated by the definition of $F$ being unstable on $\Gc_{i}$.

Now we are going to excise a small neighbourhood around each of the $x_{i}$ where one may control
the values and the gradient, slightly relaxing \eqref{eq:g(xi)<alpha,grad}.
To this end we introduce a small parameter $\gamma=\gamma(\delta)$ to be chosen in the end.
Taylor expanding $F$ around $x_{i}$ shows that on $B(x_{i},\gamma)$
\begin{equation}
\label{eq:g(x)<alpha+beta*gamma}
|F(x)| < \alpha+\beta\cdot\gamma+ C_{2}\sup\limits_{|\upsilon|=2,x\in B(x_{i},\gamma)} | \partial^{\upsilon}F(x) |\cdot \gamma^{2}
\end{equation}
and
\begin{equation}
\label{eq:nabla g(x)<beta+Sobnorm*gamma}
|\nabla F(x)| < \beta+ C_{3} \sup\limits_{|\upsilon|=2,x\in B(x_{i},\gamma)} | \partial^{\upsilon}F(x) |\cdot \gamma.
\end{equation}
Now we invoke Lemma \ref{lem:Linf bound Sobolev} with $\epsilon$ replaced
by $\epsilon/2$ and $r=2$; since the $\Gc_{i}$ are $4$-separated for
$i\le \tilde{K}$ the hypothesis $d(x_{i},x_{j})>2$ for $i\ne j$ of Lemma \ref{lem:Linf bound Sobolev}
is indeed satisfied. Hence, up to reordering, we have for $i\le \tilde{K}/2$:
\begin{equation}
\label{eq:sup bnd Hess Sob}
\sup\limits_{|\upsilon|=2,B(x_{i},\gamma)} |\partial^{\upsilon} F (x)|\le C_{4} \frac{R^{n/2}}{\sqrt{\tilde{K}}},
\end{equation}
with probability $>1-\epsilon/2$.
Substituting \eqref{eq:sup bnd Hess Sob} into \eqref{eq:g(x)<alpha+beta*gamma} and
\eqref{eq:nabla g(x)<beta+Sobnorm*gamma} yields for $x\in B(x_{i},\gamma)$, $i\le \tilde{K}/2$:
\begin{equation*}
|F(x)|< A
\end{equation*}
and
\begin{equation*}
|\nabla F(x)| < B,
\end{equation*}
with
\begin{equation}
\label{eq:A def}
A=\alpha+\beta\cdot \gamma+ c_{3}\frac{R^{n/2}}{\sqrt{\tilde{K}}}\cdot \gamma^{2}
\end{equation}
and
\begin{equation}
\label{eq:B def}
B=\beta+c_{3}\frac{R^{n/2}}{\sqrt{\tilde{K}}}\cdot \gamma.
\end{equation}

Now let $\mathcal{A}_{A,B}$ be the random variable
\begin{equation*}
\mathcal{A}_{A,B} = \vol(\{x:\: |F(x)|<A,\, |\nabla F(x)|<B\}).
\end{equation*}
On recalling \eqref{eq:tildK} and the definition of $K$ as the total number of unstable
balls of $F$ on $B(R)$ (Definition \ref{def:stability}),
our proof above shows that with probability $> 1-\epsilon/2$ for
$A$ and $B$ defined as above we have that
\begin{equation}
\label{eq:A>=gamma^2 K}
\mathcal{A}_{A,B} \ge \frac{1}{2}\tilde{K}\cdot c_{4}\gamma^{n} \ge c_{5}\gamma^{n}\cdot K.
\end{equation}

On the other hand,
by the independence of $F(x)$ and $\nabla F(x)$ for a fixed $x\in\R^{2}$,
and since the distribution of $$(F(x),\nabla F(x))\in \R^{n+1}$$
is non-degenerate Gaussian by axiom $(\rho 3)$ on $F$, for every $x\in B(R)$
\begin{equation*}
\prob(|F(x)|<A,\, |\nabla F(x)|< B) \le C_{4}A\cdot B^{n}.
\end{equation*}
Therefore, as
\begin{equation*}
\mathcal{A}_{A,B}  = \int\limits_{B(R)} \chi_{A,B}(F(x),\nabla F(x))dx
\end{equation*}
with $\chi_{A,B}$ the appropriate indicator, the expectation of $\mathcal{A}_{A,B}$
may be bounded as
$$\E[\mathcal{A}_{A,B}] \le C_{5}AB^{n}\cdot R^{n}.$$
Invoking Chebyshev's inequality, we may find a constant $C_{6}>0$
so that with probability $>1-\epsilon/2$ we may bound
\begin{equation}
\label{eq:A<CR^2}
\mathcal{A}_{A,B} < C_{6}AB^{n}R^{n}.
\end{equation}
Excising both the unlikely events of probability $<\epsilon/2$
as above we may deduce that with probability
$>1-\epsilon$ both \eqref{eq:A>=gamma^2 K} and \eqref{eq:A<CR^2} occur, implying that
\begin{equation*}
\begin{split}
c_{5}\gamma^{2}\cdot K &\le C_{6}AB^{n}R^{n}
\\&\le C_{6}
\left(\alpha+\beta\cdot \gamma+ c_{3}\frac{R^{n/2}}{\sqrt{\tilde{K}}}\cdot \gamma^{2}\right)\cdot
\left(\beta+c_{3}\frac{R^{n/2}}{\sqrt{\tilde{K}}}\cdot \gamma\right)^{n} \cdot R^{n}
\end{split}
\end{equation*}
upon recalling \eqref{eq:A def} and \eqref{eq:B def}.
Using a simple manipulation shows that with probability $>1-\epsilon$
the number of unstable balls of $F$ is bounded by
\begin{equation}
\label{eq:K<=AB^2R^2}
K \le C_{7}\cdot
\gamma^{-n}\left(\alpha+\beta\cdot \gamma+ \frac{R^{n/2}}{\sqrt{\tilde{K}}}\cdot \gamma^{2}\right)\cdot
\left(\beta+\frac{R^{n/2}}{\sqrt{\tilde{K}}}\cdot \gamma\right)^{n} \cdot R^{n},
\end{equation}
valid for every $\gamma < 1$.

Let $\delta>0$ be a given small number, and assume by
contradiction that
\begin{equation}
\label{eq:K>delta R^2 contr}
K > \delta \cdot R^{n}.
\end{equation}
Then \eqref{eq:K<=AB^2R^2} is
\begin{equation}
\label{eq:K<=xi R^{2}}
K \le \xi \cdot R^{n},
\end{equation}
where
\begin{equation*}
\xi = C_{7}\left(\frac{\alpha}{\gamma}+\beta +
\frac{\gamma}{\delta^{1/2}}\right)
\left(\frac{\beta}{\gamma^{(n-1)/n}}+ \frac{\gamma^{1/n}}{\delta^{1/2}}\right)^{2}.
\end{equation*}
It is then easy to make $\xi$ arbitrarily small by first choosing $\gamma$ and
subsequently $\alpha$ and $\beta$ sufficiently small; in particular we
may choose $\alpha$, $\beta$ and $\gamma$ so that $\xi<\delta$,
which, in light of \eqref{eq:K>delta R^2 contr}, contradicts \eqref{eq:K<=xi R^{2}}.

\end{proof}

\subsection{Proofs of the Auxiliary Lemmas
\ref{lem:bnd covering separated}-\ref{lem:Linf bound Sobolev}}
\label{sec:aux lemmas}

\begin{proof}[Proof of Lemma \ref{lem:bnd covering separated}]

Define the graph $G=(V,E)$, where $V=\{ \Gc_{i}\}$ and \\
$(\Gc_{i},\Gc_{j})\in E$ if $d(\Gc_{i},\Gc_{j})\le 4$. Consider an arbitrary node
$\Gc_{i}=B(x_{i},3T)$ and all the balls
$\Gc_{j}$ lying within distance $4$ of $\Gc_{i}$. In this case necessarily
$$\Gc_{i}\subseteq B(x_{i},6T+4)$$ whence, by a volume estimate, there are at most
$$c_{1}=c_{1}(\kappa,T)=\kappa (6T+4)^{n}/(3T)^{n}>0$$ such $\Gc_{j}$, so that the degree of $\Gc_{i}$ in $G$
is at most $d(\Gc_{i})\le c_{1}-1$. Now we start with an arbitrary vertex $v_{1}=\Gc_{i_{1}}$, and delete
all the vertexes connected to $v_{1}$; we then take $v_{2}$ to be any other vertex and continue
this process. When this process terminates (we enumerated all the vertexes of the graph) we end up
we at least $c_{1}^{-1}\cdot K$ vertexes, i.e. we proved our claim with $c:=c_{1}^{-1}$.

\end{proof}

\begin{proof}[Proof of Lemma \ref{lem:Linf bound Sobolev}]

Let $\epsilon>0$ and $r$ be given.
Since there are only finitely many $\upsilon$ with $|\upsilon |\le r$, we may also assume that $\upsilon$ is
given. By Sobolev's Imbedding Theorem ~\cite[Theorem 4.12 on p. 85]{Ad},
there exists an $m\ge 1$ and a constant $C_{1}=C_{1}(m,r)>0$,
so that for every $x\in \R^{n}$:
\begin{equation*}
| \partial^{\upsilon}F(x) | \le
C_{1}\| F\|_{H^{m,2}(B(x,1))},
\end{equation*}
and hence we have
\begin{equation}
\label{eq:sup der Sob}
\sup\limits_{x\in B(x_{i},1)} |\partial^{\upsilon}F(x) | \le C_{1}\| F\|_{H^{m,2}(B(x_{i},2))}.
\end{equation}
Note that the separateness assumption of the present lemma on $\{x_{i}\}$ implies that the
balls $\{B(x_{i},2)\}_{i\le \tilde{K}}$ are disjoint. Therefore summing up the squared rhs of
\eqref{eq:sup der Sob} for $i\le K$, we have:
\begin{equation}
\label{eq:sum sup norm < Sob}
\sum\limits_{i=1}^{K}\sup\limits_{x\in B(x_{i},1)} |\partial^{\upsilon}F(x) |^{2} \le
C_{1}^{2}\sum\limits_{i=1}^{K}\| F\|_{H^{m,2}(B(x_{i},2))} \le C_{1}^{2}\cdot
\| F\|^{2}_{H^{m,2}(B(R+3T))}.
\end{equation}

Now, as
\begin{equation*}
\| F\|^{2}_{H^{m,2}(B(R+3T))} = \sum\limits_{|\upsilon|\le m} \int\limits_{B(R+3T)}|\partial^{\upsilon}F(x)|^{2}dx,
\end{equation*}
bearing in mind the stationarity of $F$, we have
\begin{equation*}
\E[\| F\|^{2}_{H^{m,2}(B(R+3T))}] = C_{2}\cdot \vol(B(R+3T))\le C_{3}\cdot R^{n}
\end{equation*}
where $C_{2}=\sum\limits_{|\upsilon|\le m} \E[|\partial^{\upsilon} F(0)|^{2}] > 0$ is a sum of Gaussian moments.
Therefore, by Chebyshev's inequality, for $C_{4}$ sufficiently big we have
\begin{equation}
\label{eq:exp Sob < R^2}
\| F\|^{2}_{H^{m,2}(B(R+3T))} < C_{4}\cdot R^{n}
\end{equation}
with probability $>1-\epsilon$. Substituting \eqref{eq:exp Sob < R^2} into
\eqref{eq:sum sup norm < Sob} implies via Chebyshev's inequality that at least $K/2$ of the $K$ summands
in \eqref{eq:exp Sob < R^2} are bounded by
\begin{equation}
\label{eq:sup < c R^2/K}
\sup\limits_{x\in B(x_{i},1)} |\partial^{\upsilon}F(x) |^{2} \le 2 C_{1}^{2}\cdot C_{4} R^{n} / K < C_{5} R^{n}/K,
\end{equation}
i.e., up to reordering the indexes, the inequality \eqref{eq:sup < c R^2/K} holds for all $i\le K/2$.
The statement of the present lemma finally follows upon taking the square root on both sides of \eqref{eq:sup < c R^2/K},
bearing in mind that
$$\sup\limits_{x\in B(x_{i},1)} |\partial^{\upsilon}F(x) |^{2} =
\left(\sup\limits_{x\in B(x_{i},1)} |\partial^{\upsilon}F(x) |\right)^{2}.$$

\end{proof}

\subsection{An estimate on the number of small components of smooth random fields}

\label{sec:xi-small bound proof}

In this section we prove an estimate on small components of a field $F$ that is instrumental in pursuing the proof of both
Lemma \ref{lem:stable grad bnd away} and part \eqref{it:nest not leaking} of Theorem \ref{thm:scal invar no leak}.
We start by defining ``small" components of $F$.

\begin{definition}
\label{def:xi-sm def}
\begin{enumerate}

\item We say that a nodal component of $F^{-1}(0)$ is $\xi$-small if it is adjacent
to a domain of volume $<\xi$.

\item For $R>0$ let $\nod_{\xi-sm}(F;R)$ be the number of
$\xi$-small components of $F^{-1}(0)$ lying entirely inside $B(R)$.

\end{enumerate}

\end{definition}

\begin{lemma}[Cf. {\cite[Lemma $9$]{So}}]
\label{lem:xi-small bound}
There exist constants $c_{0},C_{0}>0$ such that
\begin{equation*}
\limsup\limits_{R\rightarrow\infty}\frac{\E[\nod_{\xi-sm}(F;R)]}{R^{n}} \le C_{0}\cdot \xi^{c_{0}}.
\end{equation*}
\end{lemma}

To prove Lemma \ref{lem:xi-small bound} we first formulate the following auxiliary result, whose proof is going to be
given at the end of this section.

\begin{lemma}[Cf. {\cite[Lemma $13$]{So}}]
\label{eq:Nsm det}
Let $X\subseteq \R^{n}$ be a domain and $h\in C^{2}(X)$ a (deterministic) function,
and denote
\begin{equation*}
|\partial^{2}h(x)| := \max\limits_{|\alpha|=2} |\partial^{\alpha}h(x)|.
\end{equation*}
There exist numbers $q>0$, $0<\epsilon<1$, $s>0$, $c_{0}>0$, and a constant $C_{0}>0$ depending only on $n$ and $\xi_{0}$, such that
if $\xi<\xi_{0}$ is sufficiently small, then
\begin{equation}
\label{eq:Nxism<=moments}
\begin{split}
\nod_{\xi-sm}(h;R) \le C_{0} \xi^{c_{0}} &\left(\int\limits_{B(R)} |\partial^{2}h |^{q}dx\right)^{\frac{s}{s+1}}\times
\\&\times \left( \int\limits_{B(R)} |h|^{1-\epsilon} \cdot \| \nabla h \|^{-(n-\epsilon)} dx\right)^{\frac{1}{s+1}}.
\end{split}
\end{equation}
\end{lemma}

\begin{proof}[Proof of Lemma \ref{lem:xi-small bound} assuming Lemma \ref{eq:Nsm det}]
We apply Lemma \ref{eq:Nsm det} with $h$ \\replaced by the random field $F$, and $X=B(R)$, and take the expectation to yield
\begin{equation*}
\begin{split}
&\E[N_{\xi-sm}(F;B(R))] \\&\le C_{1} \xi^{c_{0}}
\E\left[\left(\int\limits_{B(R)} |\partial^{2}F |^{q}dx\right)^{\frac{s}{s+1}} \cdot
\left( \int\limits_{B(R)} |F|^{1-\epsilon} \cdot \| \nabla F \|^{-(n-\epsilon)} dx\right)^{\frac{1}{s+1}}\right]
\\&\le
C_{1} \xi^{c_{0}}
\E\left[\int\limits_{B(R)} |\partial^{2}F |^{q}dx\right]^{\frac{s}{s+1}} \cdot
\E\left[\left( \int\limits_{B(R)} |F|^{1-\epsilon} \cdot \| \nabla F \|^{-(n-\epsilon)} dx\right)\right]
^{\frac{1}{s+1}},
\end{split}
\end{equation*}
by H\"{o}lder's inequality $$\E[XY] \le \E[X^{p}]^{1/p}\cdot \E[Y^{q}]^{1/q}$$ with $\frac{1}{p}+\frac{1}{q}=1$,
$p=\frac{s+1}{s}$, $q=s+1$,
and by the independence of $F(x)$ and $\nabla F(x)$ at each $x\in B(R)$, we have
\begin{equation}
\label{eq:small xic moments}
\begin{split}
&\E[N_{\xi-sm}(F;B(F))] \\ \le
C_{1} \xi^{c}
\E\left[\int\limits_{B(R)} |\partial^{2}F |^{q}dx\right]^{\frac{s}{s+1}} &\cdot
\left(\int\limits_{B(R)} \E\left[|F|^{1-\epsilon}\right]\cdot\E\left[\| \nabla F \|^{-(n-\epsilon)}\right] dx\right) ^{\frac{1}{s+1}}
\\&\le C_{2}\xi^{c}\cdot \vol(B(R)),
\end{split}
\end{equation}
by the stationarity of $F$, its smoothness and non-degeneracy, and the finiteness of all the Gaussian moments
involved in the r.h.s of \eqref{eq:small xic moments}, computing in the spherical coordinates.

\end{proof}

\begin{proof}[Proof of Lemma \ref{eq:Nsm det}]

Let $\mathbb{B}\subseteq \R^{n}$ be the unit ball, and $q=q(n)>n$ be a large constant.
There exists (\cite{EG} p. 143, Theorem 3) a constant $C_{1}=C_{1}(q)$
such that
\begin{equation}
\label{eq:Poincare Morrey}
\sup\limits_{x\in \mathbb{B}} |h(x)-h(0)| \le C_{1} \|\nabla h\|_{L^{q}(\mathbb{B})}
\end{equation}

For each of the $N_{\xi-sm}(h;X)$ small nodal domains lying in $X$, we may find a ball of volume $<\xi$, centred at a critical point
of $h$ with at least one zero at its boundary $h(x_{1})=0$ for some $x_{1}\in \partial B$;
by choosing the minimal such a ball, we may assume that it lies entirely inside the domains,
and hence that the balls corresponding to different domains are disjoint. Let $B=B_{x_{0}}(R)$ be an
arbitrary such a ball centred at $x_{0}$; its radius is
\begin{equation}
\label{eq:R<volB^1/n}
R\le C_{2}\vol(B)^{1/n}.
\end{equation}
Applying the inequality \eqref{eq:Poincare Morrey} above on $h(x_{0}+\frac{\cdot}{R})$ to transform $B$
into a unit ball, we have (since, by assumption, the ball centre $x_{0}$ is a critical point of $h$, we have
$\sup\limits_{B} \| \nabla h(x) \|  = \sup\limits_{B} \| \nabla h(x)-\nabla h(x_{0})\| $),
\begin{equation*}
R \cdot \sup\limits_{B} \| \nabla h(x) \| \le C_{3}(q)  R^{-n/q}\cdot R^{2}\| \partial^{2}h\|_{L^{q}(B)},
\end{equation*}
so that, upon recalling \eqref{eq:R<volB^1/n},
\begin{equation}
\label{eq:sup nabla f << D2}
\sup\limits_{B} \| \nabla h(x) \| \le C_{4} \vol(B)^{1/n-1/q} \| \partial^{2}h\|_{L^{q}(B)};
\end{equation}
we rewrite the latter inequality as
\begin{equation}
\label{eq:D2 << sup nabla(f)}
\| \partial^{2}h\|_{L^{q}(B)}^{-1}  \le C_{4} \vol(B)^{1/n-1/q}  \| \nabla h(x) \|^{-1},
\end{equation}
that holds for every $x\in B$.
In addition, we have
\begin{equation*}
\sup\limits_{B} \| h(x) \| \le \sup\limits_{B}\{ \| h(x)-h(x_{0}) \|\} + |h(x_{0})-h(x_{1})| \le
2\sup\limits_{B}\{ \| h(x)-h(x_{0}) \|\},
\end{equation*}
so that, scaling (and translating) the inequality \eqref{eq:Poincare Morrey} as before, we obtain
\begin{equation*}
\begin{split}
\sup\limits_{B} \| h(x) \| &\le C_{5} R \cdot R^{-n/q} \|\nabla h\|_{L^{q}(B)} \le
C_{6} R \sup\limits_{B} \| \nabla h(x) \| \\&\le C_{7}\vol(B)^{2/n-1/q} \|\partial^{2}h\|_{L^{q}(B)},
\end{split}
\end{equation*}
appealing to \eqref{eq:sup nabla f << D2} for the last inequality. As above, we choose to
rewrite the latter inequality as
\begin{equation}
\label{eq:f << sup nabla f}
\|\partial^{2}h\|_{L^{q}(B)}^{-1}\le C_{8}\vol(B)^{2/n-1/q} | h(x) |^{-1} ,
\end{equation}
for all $x\in B$.

Let $\epsilon>0$ be a small but fixed number.
We multiply \eqref{eq:D2 << sup nabla(f)} raised to the power $(n-\epsilon)$
by \eqref{eq:f << sup nabla f}
raised to the power $(1-\epsilon)$ and integrate the resulting inequality on $B$ to obtain
(note that the l.h.s. is constant)
\begin{equation*}
\|\partial^{2}h\|_{L^{q}(B)}^{-s} \vol(B) \le  C_{9} \vol(B)^{\tilde{t}}
\int\limits_{B} |h(x)|^{1-\epsilon} |\nabla h(x)|^{-(n-\epsilon)}dx,
\end{equation*}
where $s =  n+1-2\epsilon $ and $$\tilde{t}= (n-\epsilon)\left(\frac{1}{n}-\frac{1}{q}\right)+
(1-\epsilon)\left( \frac{2}{n}-\frac{1}{q} \right),$$
or, equivalently,
\begin{equation*}
\|\partial^{2}h\|_{L^{q}(B)}^{-s}  \le  C_{9} \vol(B)^{t}
\int\limits_{B} |h(x)|^{1-\epsilon} |\nabla h(x)|^{-(n-\epsilon)}dx
\end{equation*}
with $t=\tilde{t}-1$. It is easy to choose the parameters $q$ and $\epsilon>0$,
so that both $t,s>0$ are positive.

All in all the above shows that there exist positive constants $t,s>0$, $\epsilon>0$ and $C_{0}>0$,
such that if $B$ is a ball centred at $x_{0}\in \R^{n}$
of volume $$\vol(B)<\xi <\xi_{0}$$
such that $\nabla f(x_{0})=0$ and $h$ has at least one zero on the boundary $\partial B$, then
\begin{equation}
\label{eq:1<=delta^t moments}
1 \le  C_{0}\delta^{t}\cdot \left(\int\limits_{B}|\partial^{2}h |^{q}dx\right)^{s/(1+s)}\cdot
\left( \int\limits_{B}|h|^{-(1-\epsilon)}|\nabla h|^{-(n-\epsilon)}dx\right)^{1/(1+s)}
\end{equation}
For each of the $\nod_{\xi-sm}(h;X)$ nodal domains lying in $X$, we may find a ball of volume $<\xi$, centred at a critical point
of $h$ with at least one zero at its boundary; balls corresponding to different domains are disjoint.
Summing up \eqref{eq:1<=delta^t moments} for the various balls as above, and using H\"older's inequality we finally obtain
the statement \eqref{eq:Nxism<=moments} of the present lemma.
\end{proof}

\subsection{Proof of Lemma \ref{lem:stable grad bnd away}: Gradient bounded away from $0$ on most of the components}

\label{sec:prf grad bnd most comp}

\begin{proof}

First, by Kac-Rice (Lemma \ref{lem:Kac-Rice precise}) and the stationarity of $F$,
\begin{equation*}
\E[\vol(F^{-1}(0)\cap B(R))] = c_{F} \cdot R^{n};
\end{equation*}
hence, by Chebyshev, there exists a $C_{2}>0$ such that with probability $>1-\epsilon/2$ we have
\begin{equation*}
\vol(F^{-1}(0)\cap B(R)) < C_{2}\cdot R^{n}.
\end{equation*}
Therefore, with probability $>1-\epsilon/2$ the number of nodal components $\Gamma\subseteq F^{-1}(0)$
of diameter $\diam(\Gamma) > T$ is
\begin{equation}
\label{eq:T-long comp}
\nod_{\text{diam}>T}(g;R)<\frac{C_{2}}{T} \cdot R^{n},
\end{equation}
where we invoked the isoperimitric inequality.
Next, using Lemma \ref{lem:xi-small bound}, with probability $>1-\epsilon/2$ there are at most
\begin{equation}
\label{eq:xi-sm estimate}
C_{3}\cdot \xi^{c_{1}} R^{n}
\end{equation}
components that are $\xi$-small.

In light of the above we are only to deal with components $\Gamma$ of diameter $<T$ that are not $\xi$-small. Assume that $F$ is
$(\delta,\alpha,\beta)$-stable, and let $\nod_{\beta-unstable}(F;R)$ denote the number of components $\Gamma$ of $F^{-1}(0)$
that fail to satisfy $$|\nabla F|_{\Gamma}|>\beta,$$ and let $\Gamma$ be such
a component. Since $\{\Dc_{i}\}$ covers $B(R)$ there exists a ball $\Dc_{i}$ that intersects $\Gamma$;
in this case necessarily $\Gamma\subseteq \Gc_{i}$. Then the gradient $$|\nabla F(x)|_{\Gamma} |>\beta$$ is bounded away
from zero on $\Gamma$, unless $F$ is unstable on $\Gc_{i}$; the stability assumption on $F$ ensures that there are at most $\delta R^{n}$ of
such $\Gc_{i}$. Therefore the total number of components $\Gamma$ that are contained in one of the unstable $\Gc_{i}$ and
are {\em not} $\xi$-small is
\begin{equation}
\label{eq:not small not big}
\le C_{4}\cdot  \frac{T^{n}}{\xi} \cdot \delta R^{n}.
\end{equation}

Finally, we consolidate the various estimates: \eqref{eq:T-long comp} and
\eqref{eq:xi-sm estimate} (each one occurring with probability $>1-\epsilon/2$), and \eqref{eq:not small not big} to deduce that
outside an event of probability $<\epsilon$, if $F$ is $(\delta,\alpha,\beta)$-stable, then
\begin{equation*}
\nod_{\beta-unstable} \le \frac{C_{2}}{T} \cdot R^{n} + C_{3}\cdot \xi^{c_{1}} R^{n} + C_{4}\cdot  \frac{T^{n}}{\xi} \cdot \delta R^{n}
\le \eta \cdot R^{n},
\end{equation*}
where
\begin{equation*}
\eta = C_{5}\cdot \left(\frac{1}{T}+\xi^{c_{1}}+\frac{T^{n}}{\xi}\cdot \delta\right).
\end{equation*}
The constant $\eta$ may be made arbitrarily small by choosing the parameters $T>T_{0}$ sufficiently big,
$\xi$ sufficiently small,
and then $\delta$ sufficiently small, independent of $(\alpha,\beta)$. This concludes
the proof of the present lemma.

\end{proof}

\subsection{Proof of Theorem \ref{thm:scal invar no leak}, part \eqref{it:nest not leaking}}
\label{sec:nest no leak}

\begin{proof}

For every $m$ let $\Tc_{m}$ be the (finite) set of tree ends with $m$ vertices, so that
$$\Tc=\bigcup\limits_{m\ge 1}\Tc_{m}.$$
For a collection $\Sc\subseteq \Tc$ of tree ends we define $\nod_{X}(F,\Sc;R)$
to be the number of nodal components $c\in \Cc(F)$ of $F$ lying entirely in $B(R)$, whose corresponding tree end $e(c)\in \Sc$ is in $\Sc$
(up to isomorphism), in particular for $G\in\Tc$ we have $$\nod_{X}(F,G;R):=\nod_{X}(F,\{G\};R).$$
For $M\ge 1$ let $$\Sc_{M} = \bigcup\limits_{m\ge M}\Tc_{m}$$ be the collection of all tree ends with at least $M$ vertices.
Proving that the limiting measure $\mu_{X(F)}$ is a probability measure in this setup is equivalent to tightness,
i.e. that for every $\epsilon>0$ there exists an $M\gg 0$ sufficiently big so that
\begin{equation}
\label{eq:E[Sn]<epsilon}
\E[\nod_{X}(F,\Sc_{M};R)] < \epsilon\cdot R^{n}.
\end{equation}

We are going to invoke the Integral-Geometric Sandwich \eqref{eq:geom int scal invar} of Lemma \ref{lem:int geom sandwitch euclidian}
again; to this end we also define $$\nod_{X}^{*}(F,\Sc;R)$$ to be the number of those components $c\in \Cc(F)$ of $F^{-1}$
with $e(c)\in \Sc$ merely {\em intersecting} $B(R)$, and
$$\nod_{X}^{*}(F,\Sc;r,u):=\nod_{X}^{*}(F(\cdot+u),\Sc;r) $$ is the number
of components as above intersecting in a $u$-centred radius-$r$ ball.
Summing up the rhs of \eqref{eq:geom int scal invar} for all $G\in \Sc$ in this setting, we have
for every $$0<r<R$$ the upper bound
\begin{equation}
\label{eq:geom int sand upper*}
\begin{split}
\nod_{X}(F,\Sc;R) &\le \frac{1}{\vol(B(r))}\int\limits_{B(R+r)}\nod_{X}^{*}(F,\Sc;r,u)du \\&\le
\frac{1}{\vol(B(r))}\int\limits_{B(R+r)}\left(\nod_{X}(F,\Sc;r,u)+\widetilde{\Ac}(\tau_{u}F;r)\right)du
\end{split}
\end{equation}
(see Corollary \ref{cor:Kac Rice crit ball sphere}).

Now we take expectation of both sides of \eqref{eq:geom int sand upper*}.
Since by Corollary \ref{cor:Kac Rice crit ball sphere} and the stationarity of $F$, uniformly $$\E[\widetilde{\Ac}(\tau_{u}F;r)]=O_{r\rightarrow\infty}(r^{n-1})$$
with the constant involved in the `O`-notation depending only on $F$, given $\epsilon>0$,
we can choose a sufficiently big parameter $r>r_{0}(\epsilon)$ so that, after taking the expectation
of both sides, \eqref{eq:geom int sand upper*} is
\begin{equation}
\label{eq:geom int sand upper}
\E[\nod_{X}(F,\Sc;R)] \le \left(\frac{R}{r}+1\right)^{n}\cdot \E[\nod_{X}(F,\Sc;r)] + \frac{\epsilon}{2}\cdot R^{n}.
\end{equation}

Following Definition \ref{def:xi-sm def}, given a small parameter $\xi>0$ we denote $$\nod_{\xi-sm}(F;r,u)$$ to be the number of
$\xi$-small nodal components lying entirely inside $B(u,r)$.
Now, if a radius-$r$ ball $B$ contains a tree end
with at least $M$ vertices, then there exist at least $M/2$ domains of volume $$\le 2\frac{\vol B(r)}{M}$$
lying entirely in $B$.
Therefore, for the choice of the parameter
\begin{equation}
\label{eq:delta=r^2/n}
\xi = 2\vol B(r)/M,
\end{equation}
this (since, by the above and the local tree structure of the nesting graph,
we can only have as many tree roots corresponding to domains of volume
$\ge \xi$ as those domains in the subtree with volume $<\xi$) implies that
\begin{equation}
\label{eq:NTcn0<2Ndelta-sm}
\nod_{X}(F,\Sc_{M};r,u) < 2\nod_{\xi-sm}(F;r,u).
\end{equation}



On the other hand, Lemma \ref{lem:xi-small bound} states that there exist constants $c_{0},C_{0}>0$ (depending only on the law of $F$)
such that for $r\gg 0$ sufficiently big
\begin{equation}
\label{eq:E[delta-sm]<Cdelta^c r^2}
\E[\nod_{\xi-sm}(F;r,u)] =\E[\nod_{\xi-sm}(F;r)]  < C_{0}\xi^{c_{0}}r^{n}.
\end{equation}
Now given $\epsilon>0$, choose $r>r_{0}(\epsilon)$ sufficiently big so that
\eqref{eq:geom int sand upper} holds. Substituting
\eqref{eq:E[delta-sm]<Cdelta^c r^2} into \eqref{eq:NTcn0<2Ndelta-sm} and then finally into \eqref{eq:geom int sand upper}
now yields
\begin{equation*}
\E[\nod_{X}(F,\Sc_{M};R)] < (R+r)^{n}\cdot 2C_{0}\xi^{c_{0}}  + \frac{\epsilon}{2}\cdot R^{n} < \epsilon\cdot R^{n}
\end{equation*}
provided that $\xi$ is sufficiently small, which is the case if $M$ is sufficiently big \eqref{eq:delta=r^2/n}.

\end{proof}

\subsection{Proof of Lemma \ref{lem:Cheeger}: Cheeger Finiteness Theorem}

\label{sec:Cheeger proof}

\begin{proof}
The version of Cheeger's finiteness theorem given in ~\cite[Theorem 7.11 on p. 340]{Cha} states that, given numbers $d,V,\Lambda>0$,
there exists only finitely
many diffeomorphism classes of compact $(n-1)$-dimensional Riemannian manifolds $\Zf$ with diameter $\diam(\Zf) \le d$, volume
$\vol_{n-1}(\Zf)\ge V$ and whose sectional curvatures $\Lambda(x)=\Lambda_{\Zf;u,v}(x)$ corresponding to a $(u,v)$-plane
at a point $x\in \Zf$ satisfy $|\Lambda(x)|\le \Lambda$.
Here we verify the requisite conditions for this version of Cheeger's finiteness theorem. First
endow $\Zf$ with the Riemannian metric induced as a submanifold of $\R^{n}$, the latter with its standard Euclidian metric.
We need to show that $\beta,V,C$ above control the sectional curvatures (point-wise), diameter and the $(n-1)$-dimensional
volume (from below) of $\Zf$.

For the sectional curvatures one can express them in terms of $G$ and its first two derivatives. For example, for $n=3$
a classical formula ~\cite[pp. 139--140]{Sp} for the (Gauss) curvature $\Lambda$ of $\Zf$ at $x\in\Zf$ is given by
\begin{equation}
\label{eq:Lamb curv n=3}
\Lambda(x) = \frac{-\left| \begin{matrix}
H(G)(x) &\nabla G^{t}(x) \\ \nabla G(x) &0
\end{matrix}\right|}{|\nabla G(x)|^{4}},
\end{equation}
where $H(x)$ is the Hessian
\begin{equation*}
H(G)(x) = \left(\frac{\partial ^{2}G}{\partial x_{i}\partial x_{j}}   \right)_{i,j=1,2,3}.
\end{equation*}
From \eqref{eq:Lamb curv n=3} it is clear that our assumptions imply that
\begin{equation}
\label{eq:sup curv <= B(b,C)}
\sup\limits_{z\in\Zf} |\Lambda(x)| \le B(\beta,C),
\end{equation}
where $B(\beta,C)$ depends explicitly on $\beta$ and $C$.

For dimensions $n\ge 4$ there is a similar formula for the curvatures in terms of the first
and the second derivatives of $G$, the only division being by $|\nabla G(x)|$. The explicit
formula ~\cite{Am} for the Riemannian curvatures at $\Zf$ at a point $x\in\Zf$ shows that
the analogue of \eqref{eq:sup curv <= B(b,C)} is valid in any dimension. That is, the sectional
curvatures $\Lambda_{u,v}(x)$ in the $(u,v)$-plane at a point $x$ satisfy
\begin{equation*}
\max\limits_{x\in\Zf}\max\limits_{u,v} \left|\Lambda_{u,v}(x) \right| \le B_{n}(\beta,C),
\end{equation*}
where again $B_{n}(\beta,C)$ depends explicitly on $\beta$ and $C$.

To bound the diameter of $\Zf$ from above and the $\vol_{n-1}(\Zf)$ from below, we examine $\Zf$
locally near a point $x$. After an isometry of $\R^{n}$ we can assume that $x=0$ and $$ \nabla G(0)=(0,0,\ldots,\xi),$$
where by assumption $|\xi| \ge \beta$. The hypersurface $\Zf$ near $0$ is a graph of $x_{n}$ over $(x_{1},\ldots,x_{n-1})$.
So using these first $n-1$ coordinates to parameterize $\Zf$ we have its line element (first fundamental form)
\begin{equation*}
ds^{2}=\sum\limits_{i,j=1}^{n-1}g_{ij}dx^{i}dx^{j},
\end{equation*}
where
\begin{equation*}
g_{ij}=\delta_{ij}+\frac{\frac{\partial G}{\partial x_{i}}\cdot \frac{\partial G}{\partial x_{j}}}{\left(  \frac{\partial G}{\partial x_{n}} \right)^{2}},
\end{equation*}
$i=1,2,\ldots n-1$, $j=1,2,\ldots n-1$. It follows that for $\eta>0$ there is a $\gamma=\gamma(\beta,C)>0$ such that for $x=(x_{1},\ldots, x_{n-1})$
with $|x|\le \gamma$, the metric $g$ and the Euclidian metric satisfy
\begin{equation*}
(1+\eta)^{-1}E\le g(x)\le (1+\eta)E.
\end{equation*}
That is, this radius-$\gamma$ Euclidian ball is $(1+\eta)$ quasi-isometric to its image $\Zf$.
Thus this image has $(n-1)$-dimensional volume bounded from below by $c_{n-1}\gamma^{n-1}$
(with $c_{n-1}$ a dimensional constant), so that the required lower bound for $\vol_{n-1}(\Zf)$ is satisfied.

For the diameter, we cover $\Zf$ with $N$ such balls which are $(1+\eta)$ quasi-isometric to a Euclidian $(n-1)$-ball of
radius $\gamma$. We can do this in such a way so that each point of $\Zf$ is covered at most $c_{n-1}'$ times
(again, $c_{n-1}'$ depending only on $n-1$). From this it follows that $N$ is at most $c_{n-1}'\vol_{n-1}(\Zf)$,
which in turn is at most $c_{n-1}'\cdot V$. The diameter of $\Zf$ is then at most
$N(1+\eta)\gamma$, which is a quantity depending only on $V,C$ and $\beta$. With this we have all
the requirements to apply Cheeger's Theorem ~\cite[Theorem 7.11 on p. 340]{Cha}, and Lemma \ref{lem:Cheeger} follows.

\end{proof}

\section{Support of the limiting measures}

\label{sec:full supp}

Recall that $\mathfrak{g}_{\alpha}:\R^{n}\rightarrow\R$ are the isotropic Gaussian fields defined in section \ref{sec:bas conventions}.
As the spectral density of $\mathfrak{g}_{\alpha}$ satisfies axioms $(\rho 1)-(\rho 3)$,
Theorem \ref{thm:scal invar no leak} implies that the measures
$$\mu_{\Cc,n,\alpha}=\mu_{\Cc(\mathfrak{g}_{\alpha})}$$ and $$\mu_{\Cc,n,\alpha}=\mu_{X(\mathfrak{g}_{\alpha})},$$
on $H(n-1)$ and $\Tc$ respectively (Notation \ref{not:scal invar lim top nest meas}) are probability measures
satisfying \eqref{eq:cnt lim};
these are the same as in Theorem \ref{thm:main thm lim meas}, as established in section \ref{sec:glob res Riem}.
Theorem \ref{thm:top nest meas supp} below asserts that both have full support for all $n\ge 2$, $\alpha\in [0,1]$.

\begin{theorem}
\label{thm:top nest meas supp}
For $n\ge 2$, $\alpha\in [0,1]$ let $\mu_{\Cc,n,\alpha}$ and $\mu_{X,n,\alpha}$ be the limiting topology and
nesting probability measures corresponding
to $\mathfrak{g}_{\alpha}$, via Theorem \ref{thm:scal invar no leak}.

\begin{enumerate}

\item For all $n\ge 2$, $\alpha\in [0,1]$ the support of $\mu_{\Cc,n,\alpha}$ is $H(n-1)$.

\item For all $n\ge 2$, $\alpha\in [0,1]$ the support of $\mu_{X,n,\alpha}$ is $\Tc$.

\end{enumerate}

\end{theorem}

To prove Theorem \ref{thm:top nest meas supp} we formulate the following couple of propositions proven below;
the former is applicable on $\mathfrak{g}_{\alpha}$ with $\alpha\in [0,1)$, whereas the latter
deals only with $\mathfrak{g}_{1}$.

\begin{proposition}
\label{prop:liminf>0 interior}
Let $F:\R^{n}\rightarrow\R$ be a stationary random field with spectral density $\rho$ satisfying axioms $(\rho 1)-(\rho 3)$,
$H\in H(n-1)$ and $G\in \Tc$. Assume that the interior of the support $\supp\rho$ of $\rho$ is non-empty. Then
\begin{equation*}
\liminf\limits_{R\rightarrow\infty}\frac{\E[\nod_{\Cc}(F,H;R)]}{\vol B(R)} >0
\end{equation*}
and
\begin{equation*}
\liminf\limits_{R\rightarrow\infty}\frac{\E[\nod_{X}(F,G;R)]}{\vol B(R)} >0
\end{equation*}
\end{proposition}

\begin{proposition}
\label{prop:liminf>0 g1}
For every $n\ge 2$, $H\in H(n-1)$ and $G\in \Tc$ we have
\begin{equation*}
\liminf\limits_{R\rightarrow\infty}\frac{\E[\nod_{\Cc}(\mathfrak{g}_{1},H;R)]}{\vol B(R)} >0
\end{equation*}
and
\begin{equation*}
\liminf\limits_{R\rightarrow\infty}\frac{\E[\nod_{X}(\mathfrak{g}_{1},G;R)]}{\vol B(R)} >0
\end{equation*}
\end{proposition}

\begin{proof}[Proof of Theorem \ref{thm:top nest meas supp} assuming propositions \ref{prop:liminf>0 interior} and
\ref{prop:liminf>0 g1}]

Recall that \newline $\mu_{\Cc(\mathfrak{g}_{\alpha})}$ and $\mu_{X(\mathfrak{g}_{\alpha})}$ are as in
Notation \ref{not:scal invar lim top nest meas}. Propositions \ref{prop:liminf>0 interior} and
\ref{prop:liminf>0 g1} imply that the numbers $$\{c_{\Cc,\mathfrak{g}_{\alpha}}(H)\}_{H\in H(n-1)}$$
and $$\{c_{X,\mathfrak{g}_{\alpha}}(G)\}_{G\in \Tc}$$ are all positive for either $\alpha\in [0,1)$ or $\alpha=1$ respectively.

\end{proof}

\subsection{Towards the proof of propositions \ref{prop:liminf>0 interior} and
\ref{prop:liminf>0 g1}}

Here we formulate a result (Lemma \ref{lem:F repr pos prop} below)
asserting that if it is possible to represent a certain topology or nesting at all for a random field $F$,
then it will be represented by a positive proportion of components of $F$. First a bit of notation.

\begin{notation}

\begin{enumerate}

\item Let $\widehat{\Sigma}\subseteq\R^{n}$ be a symmetric set (i.e. invariant w.r.t. $\xi\mapsto -\xi$). We define
the space of $\widehat{\Sigma}$-band limited real-valued functions
\begin{equation}
\label{eq:Fcrho fin sum def}
\Fc_{\widehat{\Sigma}} =
\left\{ h(x)=\sum\limits_{\substack{\xi\in\widehat{\Sigma}\\ \text{finite}}} c_{\xi}e^{2\pi i \langle \xi,x\rangle}: \:
\forall\xi\in\widehat{\Sigma}.\, c_{-\xi}=\overline{c_{\xi}} \right\}
\end{equation}
of functions on $\R^{n}$.

\item Let $F:\R^{n}\rightarrow\R$ be a Gaussian field, and $\rho$ its spectral measure. Denote
$$\Fc_{F} := \Fc_{\supp{\rho}}, $$ where $\supp{\rho}$ is the support of $\rho$.

\end{enumerate}

\end{notation}

\begin{lemma}

\label{lem:F repr pos prop}

Let $F:\R^{n}\rightarrow\R$ be a smooth Gaussian field, and $H\in H(n-1)$ (resp. $G\in\Tc$)
such that there exists a
ball $D\subseteq \R^{n}$ and a (deterministic) function $h \in \Fc_{F}$
with a nodal component $c \in \Cc(h)$, $c\cong H$ (resp. $e(c) \cong G$)
lying entirely in $D$, and, in addition, $\nabla h$ does not vanish on $h^{-1}(0)\cap D$.
Then
\begin{equation*}
\liminf\limits_{R\rightarrow\infty}\frac{\E\left[\nod_{\Cc}(F,H;R)\right]}{\vol(B(R))} > 0
\end{equation*}
(resp.
$$\liminf\limits_{R\rightarrow\infty}\frac{\E[\nod_{X}(F,G;R)]}{\vol(B(R))}> 0).$$
\end{lemma}

\begin{proof}[Proof of Lemma \ref{lem:F repr pos prop}]

Let $r_{0}>0$ be the radius of $D$.
We claim that the assumptions of the present
lemma imply that for some $r_{0}>0$, the expected number of nodal components of type $H$ inside a radius $r_{0}$ ball is
\begin{equation}
\label{eq:E[nodC(F,H;r)>0]}
\E[\nod_{\Cc}(F,H;r_{0})]>0,
\end{equation}
and
\begin{equation*}
\E[\nod_{X}(F,G;r_{0})]>0.
\end{equation*}
With \eqref{eq:E[nodC(F,H;r)>0]} established, we may find $$I>\kappa(r_{0}) R^{n}$$ radius-$r_{0}$ disjoint balls
$$\{B(x_{i},r_{0})\}_{i\le I},$$ so that
$$\nod_{\Cc}(F,H;R) \ge \sum\limits_{i=1}^{I} \nod_{\Cc}(F,H;x_{i},r_{0}).$$
Hence, by the linearity of the expectation and the translation invariance of $F$ we have
\begin{equation*}
\E[\nod_{\Cc}(F,H;R)] \ge \sum\limits_{i=1}^{I}\E[\nod_{\Cc}(F,H;x_{i},r_{0})]
=I\cdot \E[\nod_{\Cc}(F,H;r_{0})],
\end{equation*}
and therefore
\begin{equation*}
\liminf\limits_{R\rightarrow\infty}\frac{\E[\nod_{\Cc}(F,H;R)]}{R^{n}} \ge \kappa(r_{0})\cdot \E[\nod_{\Cc}(F,H;r_{0})]>0,
\end{equation*}
and the same holds for $\nod_{X}(F,G;R)$.

\vspace{2mm}

Now to see \eqref{eq:E[nodC(F,H;r)>0]} let $\Hc(\rho)$ be the reproducing kernel Hilbert space, i.e.
$$\Hc(\rho)=\Fc(L^{2}(\rho)),$$ the image under Fourier transform of the space of square summable
Hermitian functions
$$g:\R^{n}\rightarrow\R$$ with
\begin{equation}
\label{eq:seminorm rho def}
\| g\|_{L^{2}(\rho)}:=\int\limits_{\R^{n}}|g(x)|^{2}d\rho(x)<\infty,
\end{equation}
$g(-x)=\overline{g(x)}$, endowed with the inner product
$$\langle \widehat{g_{1}},\widehat{g_{2}} \rangle_{\Hc(\rho)} = \langle g_{1},g_{2}\rangle_{L^{2}(\rho)}$$
(see section \ref{sec:bas conventions}). Let $\{ e_{k}\}_{k\ge 1}$ be any orthonormal basis of $\Hc(\rho)$,
so that for every $g\in \Hc(\rho)$ we have the equality
\begin{equation}
\label{eq:g=Fourier conv Hilbert}
g = \sum\limits_{k=1}^{\infty}\langle g, e_{k}\rangle \cdot e_{k},
\end{equation}
with the series on the r.h.s. of \ref{eq:g=Fourier conv Hilbert} converging in $\Hc(\rho)$; the equality \eqref{eq:g=Fourier conv Hilbert}
is $\Hc(\rho)$, i.e. modulo the equivalence relation induced by $\|\cdot \|_{L^{2}(\rho)} = 0$ with the semi-norm \eqref{eq:seminorm rho def}.

Since by the axiom $(\rho 2)$
(equivalent to the a.s. smoothness of $F$), as $k\rightarrow\infty$, the $\{e_{k}(x)\}$
are sufficiently rapidly decaying uniformly on compact subsets of $\R^{n}$, the equality \eqref{eq:g=Fourier conv Hilbert} also holds
in $C^{m}(B)$, where $m\ge 1$ and $B\subseteq \R^{n}$ is an arbitrary compact domain.
We may write
\begin{equation}
\label{eq:F=sum xik ek}
F(x)=\sum\limits_{k=1}^{\infty}\xi_{k}e_{k}(x)
\end{equation}
with $\{\xi_{k}\}_{k\ge 1}$ i.i.d. standard Gaussians. While the series on the r.h.s. of \eqref{eq:F=sum xik ek}
a.s. does not converge in the Hilbert space $\Hc(\rho)$, by the
aforementioned uniform rapid decay of $e_{k}$ on compacta, the series on the r.h.s. of
\eqref{eq:F=sum xik ek} converges uniformly on compacta, together with all the derivatives, i.e. here
we can differentiate the equality \eqref{eq:F=sum xik ek} term-wise.

Now, given a function $h\in \Fc_{F}$ and a ball $D\subseteq\R^{n}$ as appear in Lemma \ref{lem:F repr pos prop},
and $\epsilon>0$, using a standard mollifier, we may find an element $g\in \Hc(\rho)$ of the Hilbert space such that
\begin{equation}
\label{eq:||h-g||<eps/2}
\|h-g\|_{C^{1}(D)}<\frac{\epsilon}{2}.
\end{equation}
Taking into account the rapid decay of $\{e_{k}\}$
on $D$, and comparing \eqref{eq:g=Fourier conv Hilbert} to the representation \eqref{eq:F=sum xik ek}, we obtain that
\begin{equation*}
\prob(\|F-g \|_{C^{1}(D)}<\epsilon/2)>0,
\end{equation*}
and, combining it with \eqref{eq:||h-g||<eps/2}, finally
\begin{equation}
\label{eq:|F-h|C1<eps}
\prob(\|F-h \|_{C^{1}(D)}<\epsilon)>0.
\end{equation}

Let $c\in\Cc(h)$ be as in the assumptions of Lemma \ref{lem:F repr pos prop}. Since $\nabla h$ does not vanish on $c$, by an application
of the standard Morse theory, any sufficiently small $C^{1}$-perturbation of $h$ would admit a nodal component
diffeomorphic to $c$ (that is, of diffeomorphism class $H$), still lying in $D$. In other words,
there exists an $\epsilon_{0}>0$, such that if for some smooth function $g$ defined on $D$ we have
$\|g-h \|_{C^{1}(D)}<\epsilon_{0}$, then there exists a component $c'\in\Cc(g)$, $c'\cong c\cong H$,
$c'$ is lying in $D$. An application of \eqref{eq:|F-h|C1<eps} with $\epsilon=\epsilon_{0}>0$ as above
yields that the probability of $F^{-1}(0)$ containing a nodal component $c\cong H$ diffeomorphic to $H$ lying in $D$
is strictly positive, which, in its turn certainly implies \eqref{eq:E[nodC(F,H;r)>0]},
sufficient for the conclusion of the present lemma.

\end{proof}

\subsection{Proof of Proposition \ref{prop:liminf>0 interior}}

\label{sec:liminf>0 inter prf}

\begin{proof}

According to Lemma \ref{lem:F repr pos prop} it suffices to produce a $C^{2}$-function $h$ in $\Fc_{F}$ with the required properties.
We are assuming that $\supp{\rho}$ has non-empty interior (specifically for $F=\gfr_{\alpha}$, $0\le \alpha <1$).
We first show that in this case the restriction of $\Fc_{F}$ to $\overline{B}$, where $B$ is a ball centred at $0$
and of some (finite) radius, and $m\ge 0$, is dense in $C^{m}(\overline{B})$. Let $B(\xi_{0},r_{0})$ be an open ball
contained in $\supp\rho$, and let $\phi$ be a smooth non-negative function supported in $B(0,1)$ with
$$\int\limits_{\R^{n}}\phi(\xi)d\xi = 1.$$

For $0<\epsilon\le 1$ and $\beta$ a multi-index the function
\begin{equation*}
\frac{\partial^{\beta}}{\partial\xi^{\beta}} \left[ \frac{1}{(r_{0}\epsilon)^{n}}
\phi\left( \frac{\xi-\xi_{0}}{\epsilon r_{0}} \right)   \right]
\end{equation*}
are supported in $B(\xi_{0},r_{0})$. Now
\begin{equation}
\label{eq:phi exp deriv}
\begin{split}
&\int\limits_{\R^{n}}\frac{\partial^{\beta}}{\partial \xi^{\beta}}
\left[ \frac{1}{(r_{0}\epsilon)^{n}} \phi\left( \frac{\xi-\xi_{0}}{\epsilon r_{0}}  \right) \right]
e^{2\pi i \langle\xi,x \rangle} d\xi \\&
=(-1)^{\sum \beta_{j}} (2\pi x_{1})^{\beta_{1}}\cdot \ldots\cdot (2\pi x_{n})^{\beta_{n}}
\int\limits_{\R^{n}}\frac{1}{(r_{0}\epsilon)^{n}} \phi\left( \frac{\xi-\xi_{0}}{\epsilon r_{0}}  \right)
e^{2\pi i \langle\xi,x \rangle} d\xi.
\end{split}
\end{equation}
As $\epsilon\rightarrow 0$ the rhs of \eqref{eq:phi exp deriv} converges to
\begin{equation}
\label{eq:lim func 2pi xi}
(-1)^{\sum \beta_{j}} (2\pi x_{1})^{\beta_{1}}\cdot \ldots\cdot (2\pi x_{n})^{\beta_{n}}\cdot e^{2\pi i \langle\xi_{0},x \rangle},
\end{equation}
uniformly on compacta in $x$.

From the above it follows by replacing the integral on the r.h.s. of \eqref{eq:phi exp deriv}
by Riemann sums that the functions in \eqref{eq:lim func 2pi xi} can be approximated uniformly on $\overline{B}$
by elements of $\Fc_{F}$. In fact, the same holds in $C^{m}(\overline{B})$ for any fixed $m\ge 0$.
On the other hand, it is well known that finite linear combinations of $x^{\beta}$, that is polynomials, are dense in
$C^{m}(\overline{B})$.

With this the required $h$ can be found as follows. Given an $H\in H(n-1)$ we can construct a tubular neighbourhood of $H$
in $\R^{n}$ (we first realize $H$ as differentiably embedded by definition), and then a $C^{2}$-extension $f$ to $\R^{n}$,
such that $V(f)=H$ and $\nabla f \ne 0$ on $H$. Now apply the approximation above to obtain an $h\in\Fc_{F}$ for which
$V(h)$ has a nonsingular component diffeomorphic to $H$. The argument for constructing an $h$ with a given tree end is the same.
This completes the proof of Proposition \ref{prop:liminf>0 interior}.

\end{proof}

\subsection{Proof of Proposition \ref{prop:liminf>0 g1} for $n=2$: monochromatic waves attain all nesting trees}

In the case that $\alpha = 1$ it is no longer true that the restrictions of the functions in
$$E_{1} :=\Fc_{\gfr_{1}} $$ are dense in $C^{m}(\overline{B})$. In fact, any member of
$\Fc_{\gfr_{1}}$ satisfies $$\Delta u + u = 0,$$ and hence any uniform limit of such functions will
satisfy the same equation. Now for $\alpha=1$ and $n=2$, $H(1)$ consists of a single point
and the only issue, as discussed in ~\cite{NS}, in proving that their constant $\beta_{2,1}$
is positive, is to produce one function in $E_{1}$ with a nonsingular component. As they note
the $J$-Bessel function does the job. What remains for $n=2$ is the case of tree ends and to show
that we can find an $h\in E_{1}$ with a given tree end. The construction is in two steps.
First we need a modified Approximation Lemma for restriction of $E_{1}$ functions to finite sets;
this result follows from the general result in ~\cite{CS},
necessary for dealing with the higher dimensional cases,
but for dimension $2$ and finite sets $K$, one can give a simple proof. The one given
below was suggested by the referee.

\begin{lemma}[~\cite{CS}]
\label{lem:admissability CS}
If $K\subseteq \R^{2}$ is finite, then the restrictions of functions in $E_{1}$ to $K$
attain the whole of $C(K)$.
\end{lemma}

\begin{figure}[ht]
\centering
\ifanswers
\includegraphics[height=60mm]{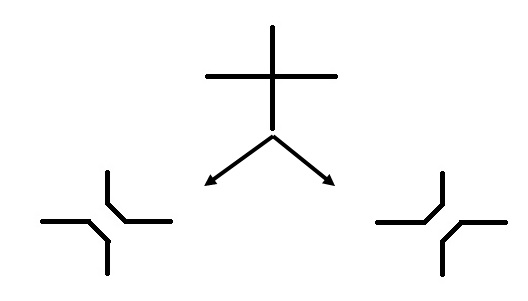}
\fi
\caption{A singularity resolution.}
\label{fig:singularity resolution}
\end{figure}

\begin{figure}[ht]
\centering
\ifanswers
\includegraphics[height=40mm]{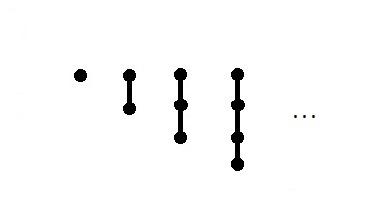}
\fi
\caption{Chain tree ends.}
\label{fig:chain nesting}
\end{figure}

\begin{proof}

By linear algebra, the statement of Lemma \ref{lem:admissability CS} is equivalent to
showing that there are no non-trivial linear relations between point evaluations at different points. Suppose that
\begin{equation}
\label{eq:sum ak f(xk)=0}
\sum\limits_{k=1}^{n}a_{k}f(x_{k}) = 0
\end{equation}
is satisfied for every $f\in E_{1}$, for some $\{a_{k}\}_{1\le k\le n}\subseteq \R$. Now define the function $A(\cdot):\R^{2}\rightarrow\R$ as
\begin{equation*}
A(\xi):= \sum\limits_{k=1}^{n}a_{k}e^{2\pi i \langle \xi, x_{k}\rangle};
\end{equation*}
$A(\cdot)$ naturally extends to an analytic function $A:\C^{2}\rightarrow\C$.
Recalling the standard notation $e(y)=e^{2\pi i y}$,
since, by the definition of $E_{1}$, the equality \eqref{eq:sum ak f(xk)=0} holds for every function of the form
\begin{equation}
\label{eq:f monochrom fin}
f(x) = \sum\limits_{\substack{\xi\in \Sc^{1}\\ \text{finite}}} c_{\xi}e(\langle \xi,x\rangle),
\end{equation}
for some $\{c_{\xi}\} \subseteq \C$ satisfying $c_{-\xi}=\overline{c_{\xi}}$, by appropriately choosing $f$'s of the form \eqref{eq:f monochrom fin}
we may deduce that $A(\cdot)$ vanishes on the unit circle $\Sc^{1}\subseteq \R^{2}$.

In what follows we show that $A$ is the zero function on $\C^{2}$,
sufficient to yield the statement of Lemma \ref{lem:admissability CS}. First
consider the connected curve $$\Cc=\{\xi= (\xi_{1},\xi_{2})\in\C^{2}: \xi_{1}^{2}+\xi_{2}^{2}=1\};$$ since the function
$A$ is analytic on $\C$ and vanishes on $\Sc^{1}=\Cc\cap \R^{2}$, it must vanish on the whole of $\Cc$.
For $1\le k\le n$ we write $x_{k}=(b_{k},c_{k}) \in \R^{2}$, and with no loss of generality we may assume that all the $b_{k}$ are
distinct (otherwise we rotate the plane), and that $\max\limits_{k\le n}b_{k}=b_{n}$. Now choose
$\mu\in\R$, take $\xi=(-i\mu,\sqrt{1+\mu^{2}})\in \Cc$, and let $\mu\rightarrow \infty$. As, by above, $A(\xi)=0$ on $\Cc$,
we have
\begin{equation*}
\begin{split}
0&=A(\xi) = \sum\limits_{k=1}^{n}a_{k}e^{2\pi \mu b_{k}} \cdot e( \sqrt{1+\mu}c_{k}) \\&=
(1+o_{\mu\rightarrow \infty}(1)) a_{n}e^{2\pi \mu b_{n}} e( \sqrt{1+\mu}c_{n}).
\end{split}
\end{equation*}
This certainly implies that $a_{n}=0$, and, continuing by induction, we may conclude that all the $a_{k}=0$ must vanish,
as claimed above.

\end{proof}

To prove Proposition \ref{prop:liminf>0 g1} we will apply Lemma \ref{lem:admissability CS}. To produce our $h$ (this being the second step below) we perturb a specific function $\phi(x_{1},x_{2})$ in $E_{1}$:
\begin{equation}
\label{eq:phi sinsin def}
\phi(x_{1},x_{2}) = \sin(\pi x_{1})\cdot \sin(\pi x_{2}).
\end{equation}

\begin{proof}[Proof of Proposition \ref{prop:liminf>0 g1}, $n=2$]

For any finite $K\subseteq\Z^{2}$ and $$\eta:K\rightarrow \{-1,1\}$$ we can find $\psi\in E_{1}$ such that
\begin{equation}
\label{eq:psi(k)=eta(k) signs}
\psi(k)=\eta(k)
\end{equation}
for every $k\in K$.
For $\epsilon>0$ sufficiently small
the function
\begin{equation}
\label{eq:phi+eps*psi}
\varphi(x_{1},x_{2})=\varphi_{\phi,\psi;\epsilon}(x_{1},x_{2}):=\phi(x_{1},x_{2})+\epsilon\cdot\psi(x_{1},x_{2}),
\end{equation}
with $\phi$ given by \eqref{eq:phi sinsin def}
will have its nodal lines
in a big compact ball containing $K$ close to those of $\phi$. The manner in which the simple crossing
in Figure \ref{fig:singularity resolution}, above, of the nodal lines at each $k\in K$
will resolve for $\epsilon$ small, that is into one of those in Figure \ref{fig:singularity resolution} below will
depend on the sign of $\psi(k)$ (and the sign of $\phi(k)$). In what follows we show that by prescribing the signs of $\psi$
at the cross points it is possible to impose that the function \eqref{eq:phi+eps*psi} attains a given $G\in\Tc$, for $\epsilon>0$ sufficiently
small.

\begin{figure}[ht]
\centering
\ifanswers
\includegraphics[height=80mm]{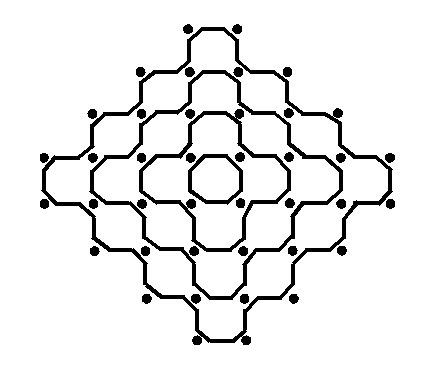}
\fi
\caption{Chain of length $4$ implementation.}
\label{fig:chain implement}
\end{figure}

\begin{figure}[ht]
\centering
\ifanswers
\includegraphics[height=20mm]{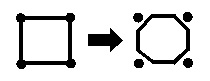}
\fi
\caption{Creating the trivial tree.}
\label{fig:perturbed square}
\end{figure}

More precisely, we prove by induction on $m\ge 1$ the following statement: for every $G\in\Tc$ with $m$ vertices there exist
a finite $K\subseteq \Z^{2}$ and a selection of signs $\{\eta_{k}\}_{k\in K}$, and a compact domain $D\subseteq \R^{2}$, so that
prescribing the signs \eqref{eq:psi(k)=eta(k) signs} for $\psi$ on $K$ yields, for $\epsilon>0$ sufficiently small, a tree end of
$\varphi$ (as in \eqref{eq:phi+eps*psi}) restricted to $D$, isomorphic to $G$.
First we build any end of the ``chain" form, as in Figure \ref{fig:chain nesting}; that includes the induction basis $m=1$
(the trivial tree as in Figure \ref{fig:perturbed square}). As it is obvious, this is
clearly possible in view of the picture in Figure \ref{fig:chain implement}, where our chain is
grown with the set of $k$'s involved highlighted.

Now we assume by induction hypothesis that all the tree ends $G\in \Tc$ with $m<M$ vertices could be attained (in the sense
of the induction statement above), and we are going to prove now that the same is true for $G\in \Tc$ with $M$ vertices.
To this end we introduce two operations: engulf and join, that would be instrumental in order for ``constructing" $G$
from trees with smaller number of vertices (that is, prescribing the appropriate signs via \eqref{eq:psi(k)=eta(k) signs}, provided that
the corresponding signs were readily constructed for smaller trees). These operations are carried out on certain figures connected to
finite subsets, $K\subseteq\Z^{2}$ and can be achieved by choosing $\{\eta(k)\}_{k\in K}$ suitably.

\begin{figure}[ht]
\centering
\ifanswers
\includegraphics[height=60mm]{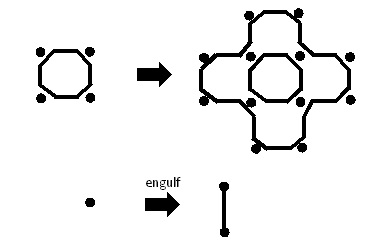}
\fi
\caption{Engulfing a single square (above). The corresponding trees are exhibited below.}
\label{fig:engulf}
\end{figure}

\begin{figure}[ht]
\centering
\ifanswers
\includegraphics[height=60mm]{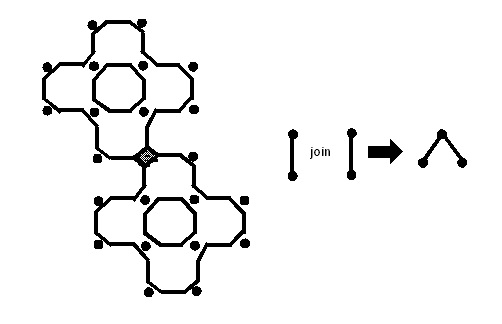}
\fi
\caption{Joining two figures. The corresponding trees are exhibited to the right.}
\label{fig:join}
\end{figure}

Start with the box of $4$ lattice points (see Figure \ref{fig:perturbed square}, to the left), which by choosing $\eta(k)$
to be suitable $\pm 1$ at the vertices yields the picture in Figure \ref{fig:perturbed square} to the right,
represented by a single point in $\Tc$ (the leftmost tree in Figure \ref{fig:chain nesting}). Engulf is the
operation of enclosing the figure at hand with one new oval. This is done by choosing (uniquely) the squares
on the boundary of the figure as indicated in the picture in Figure \ref{fig:engulf}.
The join is done by taking two figures, and joining the right lowest corner of one to the left highest corner of
the other, as in Figure \ref{fig:join}.

\begin{figure}[ht]
\centering
\ifanswers
\includegraphics[height=30mm]{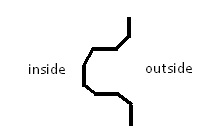}
\fi
\caption{A ``kink". It does not pose a serious problem.}
\label{fig:kink}
\end{figure}

\begin{figure}[ht]
\centering
\ifanswers
\includegraphics[height=60mm]{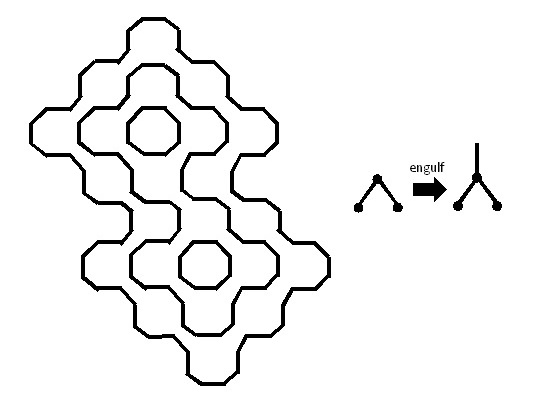}
\fi
\caption{Engulfing a figure with a ``kink".}
\label{fig:engulf kink}
\end{figure}

The figures formed will always have a highest single square at the top and a lowest single square at the
bottom, so that there is no ambiguity (this property is true to start off, and is preserved by the two operations).
A second point is that engulfing any figure that arises is always possible. The only potential problem
is that the join may introduce a non-convex ``kink" of the shape, as in Figure \ref{fig:kink}. This could lead
to a block in engulfing. However, as indicated in the example in Figure \ref{fig:engulf kink},
this does not cause a problem. At any further stage these kinks don't interact, and one can always engulf.

Let us now formally perform the induction step.
Given a tree end $G\in \Tc$ with $M\ge 2$ vertices, it is either the engulf of a smaller tree ends $G'$ (with $M-1$ vertices), or the join
of two tree ends $G'$ of $M'\ge 2$ vertices and $G''$ of $M''\ge 2$ vertices whence $M=M'+M''-1$ and we have $M',M''<M$; denote
$\{\eta_{G'}(k)\}_{k\in K'}$ (resp. $\{\eta_{G'}(k)\}_{k\in K'}$ and $\{\eta_{G''}(k)\}_{k\in K''}$) the corresponding signs obtained from the induction hypothesis applied on $G'$ (resp. $G'$, $G''$). Then, by the definition of the engulf (resp. join) procedure we obtain a prescription
of the signs $\{\eta_{G}(k)\}_{k\in K}$ on a bigger set $K$ that yields a tree end isomorphic to $G$ on a corresponding domain $D\subseteq \R^{2}$,
which concludes the induction step, and therefore also the present proof.

\end{proof}

\section{Semi-local nodal counts on $\M$}

\label{sec:loc res Riem}

\subsection{Local results}

Here we formulate a local result (Theorem \ref{thm:loc scal} below)
around a point $x\in \M$, after blowing up the coordinates according to the natural
scaling of $$f=f_{\alpha;T}:\M\rightarrow\R,$$ the band-limited Gaussian functions \eqref{eq:f = sum cj phij}.
Recall that for $H\in H(n-1)$ (resp. $G\in\Tc$) Theorem \ref{thm:loc scal} yields constants $c_{\Cc;\mathfrak{g_{\alpha}}}(H)$
(resp. $c_{X;\mathfrak{g_{\alpha}}}(G)$) corresponding to the limiting random fields $\mathfrak{g}_{\alpha}$
of $f$, under the same natural scaling around $x$.

\begin{notation}

\label{not:counting var manifold}

\begin{enumerate}

\item For $x\in\M$, $r>0$ let $B(x,r)\subseteq \M$ be the geodesic ball in $\M$ of radius $r$.

\item For $H\in H(n-1)$ (resp. $G\in\Tc$) let $$\nod_{\Cc}\left(f_{\alpha;T},H; x,r\right)$$
(resp. $\nod_{X}\left(f_{\alpha;T},G; x,r\right)$)
be the number of components of $f_{\alpha;T}^{-1}(0)$ lying in $B(x,r)$ of class $H$ (resp. corresponding
to nesting tree isomorphic to $G$).

\item In the situation as above $\nod_{\cdot}^{*}(f_{\alpha;T},\cdot;x,r)$ is the number of those merely
{\em intersecting} $B(x,r)$.

\end{enumerate}

\end{notation}

\begin{theorem}[Cf. {\cite[Theorem 5]{So}}]
\label{thm:loc scal}
For every $H \in H(n-1)$, $G\in\Tc$, and $x\in\M$ we have
\begin{equation*}
\lim\limits_{R\rightarrow\infty}\limsup\limits_{T\rightarrow\infty} \prob \left\{
\left| \frac{\nod_{\Cc}\left(f_{\alpha;T},H; x,\frac{R}{T}\right)}{\vol(B(R))} -c_{\Cc;\mathfrak{g_{\alpha}}}(H)  \right| >\epsilon \right\} = 0,
\end{equation*}
\begin{equation*}
\lim\limits_{R\rightarrow\infty}\limsup\limits_{T\rightarrow\infty} \prob \left\{
\left| \frac{\nod_{\Cc}\left(f_{\alpha;T},G; x,\frac{R}{T}\right)}{\vol(B(R))} -c_{X;\mathfrak{g_{\alpha}}}(G)  \right| >\epsilon \right\} = 0.
\end{equation*}
\end{theorem}

\subsection{Proof of Theorem \ref{thm:loc scal}}

Let $f$ be defined as in \eqref{eq:f = sum cj phij}. For a fixed point $x\in \M$ we blow up the coordinates
around $x$, and consider $f$ on a small geodesic ball $B(x,\epsilon_{0})\subseteq \M$. That is, we define the Gaussian field
\begin{equation*}
f_{x;T}(u) = f\left(\exp_{x}(uT^{-1})\right),
\end{equation*}
$u\in T_{x}(\M)\cong \R^{n}$ lying a the Euclidian ball $$u\in B(\epsilon_{0}\cdot T)\subseteq\R^{n}.$$
Since $\M$ is a {\em compact} manifold, for $r$ sufficiently small, independent of $x$,
\begin{equation}
\label{eq:bi(x,r,fL)=bi(rL,fxL)}
\nod_{\cdot} (f,\cdot;x,r) = \nod_{\cdot}(f_{x;T},\cdot; rT),
\end{equation}
the rhs of the latter equality being of a random field defined on the Euclidian space.

The random field $f_{x;T}$ is Gaussian with covariance kernel
\begin{equation*}
\begin{split}
K_{x;T}(u,v)&:=\E[f_{x;T}(u)\cdot f_{x;T}(v)] = \E[f(\exp_{x}(uT^{-1}))\cdot f(\exp_{x}(vT^{-1}))]
\\&= K_{\alpha}(T;uT^{-1},vT^{-1}),
\end{split}
\end{equation*}
(cf. \eqref{eq:Kalpha covar f}), and \eqref{eq:Kalpha asymp} implies that, as $T\rightarrow\infty$,
\begin{equation}
\label{eq:rf(u,v) aprx covar galp}
r_{f_{x;T}}(u,v) \approx B_{n,\alpha}(|u-v|) = \E[\gfr_{n,\alpha}(u)\cdot \gfr_{n,\alpha}(v)].
\end{equation}

Hence, for every $x\in \M$ fixed, the fields $\{f_{x;T}\}_{T}$, defined on growing domains,
{\em converge} on $\R^{n}$ to $\gfr_{\alpha}$, to be formulated more precisely.
In particular, as $f$ and $\gfr_{\alpha}$
are defined on different probability spaces, we need to couple them, i.e. define on the same probability space.
We now formulate the following proposition that will imply Theorem \ref{thm:loc scal} proved immediately after.

\begin{proposition}
\label{prop:bi(R-1,Fx)<=bi(R/L,fL)<=bi(R+1,Fx)}
Let $x\in\M$, $R>0$ be sufficiently big, $\delta>0$ be given, $H\in H(n-1)$ and $G\in \Tc$.
Then there exists a coupling of $f$ and $\gfr_{\alpha}$ and $T_{0}=T_{0}(R,\delta)$ sufficiently big, so that for
all $T>T_{0}$ outside an event of probability $<\delta$ we have
\begin{equation*}
\nod_{\cdot}(\gfr_{\alpha},\cdot;R-1) \le
\nod_{\cdot} \left(f,\cdot; x,\frac{R}{T}\right) \le \nod_{\cdot} (\gfr_{\alpha},\cdot;R+1).
\end{equation*}
\end{proposition}

Proposition \ref{prop:bi(R-1,Fx)<=bi(R/L,fL)<=bi(R+1,Fx)} will be proven in \S\ref{sec:bi(R-1,Fx)<=bi(R/L,fL)<=bi(R+1,Fx)}
immediately below.

\begin{proof}[Proof of Theorem \ref{thm:loc scal} assuming Proposition \ref{prop:bi(R-1,Fx)<=bi(R/L,fL)<=bi(R+1,Fx)}]

Theorem \ref{thm:loc scal} follows from Proposition \ref{prop:bi(R-1,Fx)<=bi(R/L,fL)<=bi(R+1,Fx)} and
Theorem \ref{thm:scal invar count} applied on $\gfr_{\alpha}$ at once.

\end{proof}

\subsection{Some preparatory results towards the proof of Proposition \ref{prop:bi(R-1,Fx)<=bi(R/L,fL)<=bi(R+1,Fx)}}

\label{sec:bi(R-1,Fx)<=bi(R/L,fL)<=bi(R+1,Fx)}

The proof of Proposition \ref{prop:bi(R-1,Fx)<=bi(R/L,fL)<=bi(R+1,Fx)} is quite similar to the proof of
the analogous statement on nodal count from ~\cite[Lemmas $6-7$]{So}; here we need to check that
the topological and the nesting nodal structure rather than merely the nodal count
is stable under the perturbation. In order
to prove Proposition \ref{prop:bi(R-1,Fx)<=bi(R/L,fL)<=bi(R+1,Fx)}
we need to excise the following exceptional events $\Delta_{i}$, $1\le i \le 4$.

Let $\delta>0$ be a small parameter that will control the probabilities
of the discarded events, $\beta>0$ a small parameter that will control the quality of the various approximations,
and $M>0$ a large parameter. Given $R$ and $T$ big we define
\begin{equation*}
\Delta_{1}=\Delta_{1}(R,T;\beta)= \{ \| f_{x;T}-\gfr_{\alpha}\|_{C^{1}(\overline{B}(2R))} \ge \beta\},
\end{equation*}
\begin{equation*}
\Delta_{2}=\Delta_{2}(R,T;\delta,M) = \left\{ \| f_{x;T}\|_{C^{2}(\overline{B}(2R))} \ge \delta^{-1}M  \right\} ,
\end{equation*}
\begin{equation*}
\Delta_{3}=\Delta_{3}(R;\delta,M) = \left\{ \| \gfr_{\alpha}\|_{C^{2}(\overline{B}(2R))} \ge \delta^{-1}M  \right\}
\end{equation*}
and the ``unstable event"
\begin{equation*}
\Delta_{4}(R,T;\beta) =
\left\{ \min\limits_{u\in \overline{B}(2R)}\max \{|f_{x;T}(u)|, |\nabla f_{x;T}(u)|\} \le 2\beta  \right\}.
\end{equation*}

The following lemma from ~\cite{So} is instrumental in proving that
for $f$ and $\gfr_{\alpha}$ suitably coupled the exceptional events have small probability.
Recall that we have \eqref{eq:rf(u,v) aprx covar galp}; more precisely
\eqref{eq:Kalpha asymp} implies that the covariance function of $f_{x;T}$ together with its derivatives converge uniformly to
the covariance function of $\gfr_{\alpha}$ and its respective derivatives.

\begin{lemma}[{\cite[Lemma $4$]{So}}]
\label{lem:E[|fxL-Fx|<alpha]}
Given $x\in\M$, $R>0$ and $\upsilon>0$ there exists a coupling of $f$ and $\gfr_{\alpha}$ and
$T_{0}=T_{0}(R,\upsilon)$ with
\begin{equation*}
\E\left[\|f_{x;T}-\gfr_{\alpha}\|_{C^{1}(\overline{B}(2R))}\right] < \upsilon
\end{equation*}
for all $T\ge T_{0}$.
\end{lemma}

From this point we will avoid mentioning coupling that will be understood implicitly.
The following is a simple corollary from Lemma \ref{lem:E[|fxL-Fx|<alpha]}.

\begin{lemma}
\label{lem:Delta1 small}
For $R>0$ sufficiently big given, and $\beta,\delta>0$ there exists $T_{0}=T_{0}(R,\beta,\delta)$ so
that for $T>T_{0}$ the probability of $\Delta_{1}$ is
\begin{equation}
\label{eq:Delta1 small}
\prob(\Delta_{1}(R,T;\beta))<\delta
\end{equation}
arbitrarily small.
\end{lemma}

The following lemma yields a bound on the probability $\prob(\Delta_{2})$
and $\prob(\Delta_{3})$.

\begin{lemma}
\label{lem:Delta2-3 small}

\begin{enumerate}
\item
For every $R>0$ sufficiently big there exists $M=M(R)>0$ so that for all $\delta>0$
\begin{equation}
\label{eq:Delta3 small}
\prob(\Delta_{3}(R;\delta,M)) < \delta.
\end{equation}

\item
For $R>0$ there exists $M(R)>0$ and $T_{0}=T_{0}(R)$ so that for all $T>T_{0}$ and $\delta>0$
\begin{equation}
\label{eq:Delta2 small}
\prob(\Delta_{2}(R,T;\delta,M))< \delta.
\end{equation}

\end{enumerate}

\end{lemma}

\begin{proof}

For \eqref{eq:Delta3 small} we may choose $M$ to be $$M=\E [\| \gfr_{\alpha}\|_{C^{2}(\overline{B}(2R))} ]  < \infty,$$
finite by \cite{Adler-Taylor}, Theorem $2.1.1$. The estimate \eqref{eq:Delta3 small} with $M$ as here then
follows by Chebyshev's inequality.

In order to establish \eqref{eq:Delta2 small} we observe that
by \cite{Adler-Taylor}, Theorem $2.2.3$ (``Sudakov-Fernique comparison inequality")
and \eqref{eq:Kalpha asymp} applied to both $K_{x;T}$ and its derivatives
for all $M_{1}>M$ there exists $T_{0}=T_{0}(R,M_{1})$ such that for all $T>T_{0}$
$$\E \left[ \| f_{x;T}\|_{C^{2}(\overline{B}(2R))}\right] < M_{1}.$$ Hence
\eqref{eq:Delta2 small} follows from using Chebyshev's inequality as before.
Using $M_{1}$ instead of $M$ will also work with \eqref{eq:Delta3 small}.

\end{proof}

Finally, for the unstable event $\Delta_{4}$ we have the following bound:

\begin{lemma}[{\cite[Lemma $5$]{So}}]
\label{lem:Delta4 small}
For $R>0$ sufficiently big given, $M>0$ and $\delta>0$ there exist $\beta=\beta(R,M,\delta)>0$ and $T_{0}=T_{0}(R,M,\delta)>0$
so that for all $T>T_{0}$ outside $\Delta_{2}\cup\Delta_{3}$ the probability of $\Delta_{4}$ is
\begin{equation}
\label{eq:Delta4 small}
\prob\left(\Delta_{4}(R,T;\beta) \setminus (\Delta_{2}(R,T;\delta,M)\cup \Delta_{3}(R;\delta,M))\right)< \delta
\end{equation}
arbitrarily small.
\end{lemma}

\subsection{Proof of Proposition \ref{prop:bi(R-1,Fx)<=bi(R/L,fL)<=bi(R+1,Fx)}}

\label{sec:bi(R-1,Fx)<=bi(R,fxL)<=bi(R+1,Fx)}

\begin{proposition}[cf. {\cite[Lemma $6$]{So}}]
\label{prop:bi(R-1,Fx)<=bi(R,fxL)<=bi(R+1,Fx)}
Outside of $\Delta_{1}\cup \Delta_{4}$ we have
\begin{equation*}
\nod_{\cdot}(\gfr_{\alpha},\cdot;R-1) \le
\nod_{\cdot} \left(f_{x;T},\cdot; R\right) \le \nod_{\cdot} (\gfr_{\alpha},\cdot;R+1).
\end{equation*}
\end{proposition}

\begin{proof}[Proof of Proposition \ref{prop:bi(R-1,Fx)<=bi(R/L,fL)<=bi(R+1,Fx)} assuming Proposition
\ref{prop:bi(R-1,Fx)<=bi(R,fxL)<=bi(R+1,Fx)}]

Since the probability of the \\ events $\Delta_{i}$, $i=1\ldots 4$ is $<4\delta$ for $T>T_{0}(R,\beta,\delta)$ by lemmas
\ref{lem:Delta1 small}-\ref{lem:Delta4 small}, the statement of Proposition \ref{prop:bi(R-1,Fx)<=bi(R/L,fL)<=bi(R+1,Fx)}
follows from Proposition
\ref{prop:bi(R-1,Fx)<=bi(R,fxL)<=bi(R+1,Fx)} at once upon replacing $\delta$ by $\frac{\delta}{4}$,
bearing in mind \eqref{eq:bi(x,r,fL)=bi(rL,fxL)}.

\end{proof}

\begin{proof}[Proof of Proposition \ref{prop:bi(R-1,Fx)<=bi(R,fxL)<=bi(R+1,Fx)}]

Modifying the proofs of \cite[Lemmas $6,7$]{So}, here we only prove the somewhat more complicated
case of tree ends: outside of $\Delta_{1}\cup \Delta_{4}$ for $G\in \Tc$ one has
\begin{equation}
\label{eq:nod Fx, R-1 <= nod fxL, R}
\nod_{X} (\gfr_{\alpha},G;R-1) \le \nod_{X} (f_{x;T},G;R),
\end{equation}
with the inequality
$$\nod_{X} \left(f_{x;T},G; R\right) \le \nod_{X} (\gfr_{\alpha},G;R+1),$$
and their analogues for nodal components topology following along the same lines.
Outside $\Delta_{1}\cup \Delta_{4}$ we have
\begin{equation}
\label{eq:||fxL-Fx||<beta}
\|f_{x;T}-\gfr_{\alpha}\|_{C^{1}(\overline{B}(2R))} <\beta,
\end{equation}
and also
\begin{equation}
\label{eq:maxfxL,nabla>beta}
\min\limits_{\overline{B}(2R)}\max\{ |f_{x;T}|,|\nabla f_{x;T} | \}>\beta ; \;
\min\limits_{\overline{B}(2R)}\max\{ |\gfr_{\alpha}|,|\nabla \gfr_{\alpha} | \} > \beta.
\end{equation}

For $t\in [0,\beta]$ consider $$g_{t} := \gfr_{\alpha}+\frac{t}{\beta} \cdot (f_{x;T}-\gfr_{\alpha}).$$
We claim that for all $t\in [0,\beta]$, $g_{t}$ has no critical zeros in $B(2R)$, i.e.
points $y\in B(2R)$ such that $g_{t}(y) =0$ and $\nabla g_{t}(y) = 0$.
Otherwise let $t_{0}\in [0,\beta]$ and $y_{0}\in B(2R)$ such that
$g_{t_{0}}(y_{0}) =0$, $\nabla g_{t_{0}}(y_{0}) = 0$. This contradicts \eqref{eq:maxfxL,nabla>beta},
as then
\begin{equation*}
|\gfr_{\alpha}(y_{0})| = \left| g_{t_{0}}(y_{0}) -  \frac{t_{0}}{\beta} \cdot (f_{x;T}(y_{0})-\gfr_{\alpha}(y_{0}))\right|
=\frac{t_{0}}{\beta} |f_{x;T}(y_{0})-\gfr_{\alpha}(y_{0})| < \beta
\end{equation*}
by \eqref{eq:||fxL-Fx||<beta}, and
\begin{equation*}
\begin{split}
|\nabla \gfr_{\alpha}(y_{0})| &= \left|\nabla g_{t_{0}}(y_{0}) -  \frac{t_{0}}{\beta} \cdot \nabla (f_{x;T}-\gfr_{\alpha})(y_{0})\right|
\\&=\frac{t_{0}}{\beta}\cdot |  \nabla (f_{x;T}-\gfr_{\alpha})(y_{0})| < \beta,
\end{split}
\end{equation*}
again, by \eqref{eq:||fxL-Fx||<beta}. That concludes the proof of the non-existence of
critical zeros of $g_{t}$, $t\in [0,\beta]$ in $B(2R)$.

Now, since under the assumptions of Proposition \ref{prop:bi(R-1,Fx)<=bi(R,fxL)<=bi(R+1,Fx)},
we excluded the event $\Delta_{1}\cup \Delta_{4}$, that implies that the various components $\gamma$ of $\gfr_{\alpha}^{-1}(0)$
are regular and bounded away from each other. Moreover Nazarov-Sodin \cite[Lemmas $6,7$]{So} showed that each
component $\gamma$ lies in an ``annulus" inside $$\gamma_{1}=\{y\in \R^{n}:d(y,\gamma)< 1\},$$ bounded by the two
hypersurfaces $\gfr_{\alpha}(x)=\pm \beta$, where $\beta>0$ is assumed to be sufficiently small; different components
$\gamma$ correspond to different, pairwise disjoint annuli.

Since $\Delta_{1}$ was excluded, for every $t\in [0,\beta]$, $g_{t}(y)>0$ is positive on $\gfr_{\alpha}^{-1}(\beta)$
and $g_{t}(y)<0$ is negative on $\gfr_{\alpha}^{-1}(-\beta)$. Therefore for every $\gamma$ component of
$\gfr_{\alpha}$ lying in $B(R)$ and $t\in [0,\beta]$, $g_{t}^{-1}(0)$ contains at least one component lying in $\gamma_{1}$.
A standard result from the Differential Topology asserts that a $1$-parameter family $g_{t}:\gamma_{1}\rightarrow\R$, $t\in [0,\beta]$ defined
on the open bounded annulus $\gamma_{1}$, so that for all $t\in [0,a]$, the function $g_{t}$ admits no non-degenerate critical points (i.e. it is a Morse function), the zero sets $\{g_{t}^{-1}(0)\}$ are diffeomorphic; note that, by the above, there is no nodal intersection with the boundary
$\partial \gamma_{1}$, as here $g_{t}$ is strictly positive or negative.

By the above, for every $\gamma$ component of $\gfr_{\alpha}^{-1}(0)$ lying in $B(R)$ and $t\in [0,\beta]$ there exists a unique component $\gamma^{t}$ of $g_{t}^{-1}(0)$ lying in $\gamma_{1}\subseteq B(R+1)$
(by the above, components cannot merge or split in $\gamma_{1}$; new components would also generate a critical point). In particular,
the correspondence $\gamma \mapsto \gamma^{t}$ is well-defined and injective between components of $\gfr_{\alpha}^{-1}(0)$ lying in $B(R)$
and the components of $g_{t}^{-1}(0)$ lying in $B(R+1)$, and moreover $\gamma$ and $\gamma^{t}$ are diffeomorphic.

Furthermore, the nesting tree of $g_{t}$ is preserved: there exists
an injective map $\phi^{t}:\Omega_{0} \rightarrow \Omega_{t}$ between the vertices of the nesting trees of
$g_{0}=\gfr_{\alpha}$ and $g_{t}$ respectively
such that $\omega,\omega'\in \Omega_{0}$ are connected by an oval $\gamma$ of $\gfr_{\alpha}^{-1}(0)$, if and
only if $\phi^{t}(\omega), \phi^{t}(\omega') \in \Omega_{t}$ are connected by the oval $\gamma^{t}$ of $g_{t}^{-1}(0)$.
Equivalently, if $\omega$ is the domain lying inside $\gamma_{1}$ and outside
$\gamma_{2}$, then $\phi^{t}(\omega)$ is the domain lying inside $\gamma_{1}^{t}$ and
outside $\gamma_{2}^{t}$ (by Jordan's Theorem in this setting, see section \ref{sec:bas conventions}).
No new ovals are created inside ovals corresponding to $\gamma\subseteq B(R-1)$, as otherwise there would
exist critical zeros, that were already ruled out.

Now let $c\in \Cc(\gfr_{\alpha})$ be a nodal component of $\gfr_{\alpha}$ lying in $B(R-1)$,
with $e(c)$ isomorphic a given rooted tree $G$.
By the above, for every $t\in [0,\beta]$ the nesting tree end $e_{g_{t}^{-1}(0)}(c^{t})$ is isomorphic to
$e(c)$, and hence to $G$. As $$c^{t}\subseteq \{y:d(y,c)\le 1\} \subseteq B(R),$$ it implies
\eqref{eq:nod Fx, R-1 <= nod fxL, R} as claimed.
\end{proof}

\section{Proof of Theorem \ref{thm:main thm lim meas}: gluing local results on $\M$}

\label{sec:glob res Riem}

\subsection{Proof of Theorem \ref{thm:main thm lim meas}}

In this section we finally prove Theorem \ref{thm:main thm lim meas} with the measures
\begin{equation}
\label{eq:muCnalpha=muCCgfr}
\mu_{\Cc,n,\alpha} = \mu_{\Cc(\gfr_{\alpha})}
\end{equation}
and
\begin{equation}
\label{eq:muXnalpha=muCCgfr}
\mu_{X,n,\alpha} = \mu_{X(\gfr_{\alpha})},
\end{equation}
as in Notation \ref{not:scal invar lim top nest meas}; both $\mu_{\Cc(\gfr_{\alpha})}$ and $\mu_{X(\gfr_{\alpha})}$
are probability measures
by the virtue of Theorem \ref{thm:scal invar no leak}, parts \eqref{it:top not leaking} and \eqref{it:nest not leaking} respectively.
First we formulate the following theorem that will imply Theorem \ref{thm:main thm lim meas} at once;
for $H\in H(n-1)$ (resp. $G\in \Tc$) let $\nod_{\Cc}(f,H)$ (resp. $\nod_{X}(f,H)$) the
the total number of components $c\in\Cc(f)$ of topological class $H$ (resp. such that $e(c)$ isomorphic
to $G$).

\begin{theorem}
\label{thm:tot cnt asymp}
For every $H\in H(n-1)$, $G\in \Tc$ we have
\begin{equation*}
\E\left[\left|\frac{\nod_{\Cc}(f,H)}{T^{n}} - \vol\M\cdot c_{\Cc;\gfr_{\alpha}}(H)\right|\right] \rightarrow 0
\end{equation*}
and
\begin{equation*}
\E\left[\left|\frac{\nod_{X}(f,G)}{T^{n}} - \vol\M\cdot c_{X;\gfr_{\alpha}}(G)\right|\right] \rightarrow 0.
\end{equation*}

\end{theorem}

\begin{proof}[Proof of Theorem \ref{thm:main thm lim meas} assuming Theorem \ref{thm:tot cnt asymp}]

As it was mentioned above we take the postulated limiting measures
$\mu_{\Cc,n,\alpha}$ and $\mu_{X,n,\alpha}$ to be given by \eqref{eq:muCnalpha=muCCgfr} and \eqref{eq:muXnalpha=muCCgfr}
respectively. As these are probability measures having the full support $H(n-1)$ and $\Tc$ respectively by the virtue
of theorems \ref{thm:scal invar no leak} and \ref{thm:top nest meas supp}, the only thing
remaining to be proven is \eqref{eq:lim muf exist}.

Here we only prove \eqref{eq:lim muf exist} for $\mu_{\Cc}$,
i.e. that for all $\epsilon>0$
\begin{equation}
\label{eq:D(mu(f)-mulim >epsilon ->0}
\prob\left\{f\in\Ec_{\M,\alpha}(T):\: D(\mu_{\Cc(f)},\mu_{\Cc,n,\alpha})>\epsilon \right\}\rightarrow 0,
\end{equation}
the proof for $\mu_{X(f)}$ being identical. Let $\epsilon>0$ be given.
Bearing in mind the definition \eqref{eq:mu C n alpha def scal invar} of $\mu_{\Cc,n,\alpha} = \mu_{\Cc(\gfr_{\alpha})}$
and \eqref{eq:NS scal invar},
Theorem \ref{thm:tot cnt asymp} implies that for every $H\in H(n-1)$ we have
\begin{equation}
\label{eq:mu(H)-muC(H)->0}
\prob\left\{ \left|\mu_{\Cc(f)}(H) - \mu_{\Cc,n,\alpha}(H)\right|>\epsilon \right\} \rightarrow 0,
\end{equation}
as $T\rightarrow\infty$.
Let $$H(n-1)=\{H_{k} \}_{k\ge 1}$$ be any enumeration of the
(countable) family $H(n-1)$ of $(n-1)$-dimensional topological classes embeddable in $\Sc^{n}$ (equivalently,
embeddable in $\R^{n}$), and for $K\ge 1$ let  $$H(n-1;K) = \{H_{k} : k\le K \}.$$

Now, since $\mu_{\Cc,n,\alpha}$ is a probability measure, there exists $K=K(\epsilon)$ sufficiently big so that
\begin{equation}
\label{eq:muC>K1<eps/2}
\mu_{\Cc,n,\alpha}(H(n-1)\setminus H(n-1;K)) < \frac{\epsilon}{4}.
\end{equation}
For every $\delta>0$ and $$H_{k}\in H(n-1;K)$$ we employ \eqref{eq:mu(H)-muC(H)->0} with $\epsilon$ replaced by $\epsilon/2K$
to obtain $$\widetilde{T_{0}}(H_{k})=\widetilde{T_{0}}(H_{k},K,\epsilon,\delta)$$ such that
\begin{equation}
\label{eq:mu(H)-muC(H)>eps/2K<delta/2K}
\prob\left\{ \left|\mu_{\Cc(f)}(H) - \mu_{\Cc,n,\alpha}(H)\right|>\frac{\epsilon}{4K} \right\} <\frac{\delta}{2K}.
\end{equation}
Let $$T_{0}=T_{0}(\epsilon,\delta)=\max\limits_{k\le K} \widetilde{T_{0}}(H_{k}),$$ and $$C\subseteq H(n-1;K)$$ a sub-collection
of topology classes.
Summing up \eqref{eq:mu(H)-muC(H)>eps/2K<delta/2K} for $H_{k}\in C$
and on using the triangle inequality we obtain
\begin{equation}
\label{eq:muC(C)-mun(C) > eps <delta/2}
\prob\left\{ \left|\mu_{\Cc(f)}(C) - \mu_{\Cc,n,\alpha}(C)\right|>\frac{\epsilon}{4} \right\} <\frac{\delta}{2}
\end{equation}
holding for every $T>T_{0}$, with the exceptional event of probability $<\frac{\delta}{2}$ independent of $C$.
In particular, for $C=H(n-1,K)$, \eqref{eq:muC(C)-mun(C) > eps <delta/2} is
\begin{equation}
\label{eq:muC(B)-mun(B) > eps <delta/2}
\prob\left\{ \left|\mu_{\Cc(f)}(H(n-1,K)) - \mu_{\Cc,n,\alpha}(H(n-1,K))\right|>\frac{\epsilon}{4} \right\} <\frac{\delta}{2},
\end{equation}
and since both $\mu_{\Cc(f)}$ and $\mu_{\Cc,n,\alpha}$ are probability measures, by taking the complement $$H(n-1)\setminus H(n-1;K),$$
\eqref{eq:muC(B)-mun(B) > eps <delta/2} is equivalent to
\begin{equation}
\label{eq:muC(H-B)-mun(H-B) > eps <delta/2}
\begin{split}
\prob\bigg\{ \bigg|&\mu_{\Cc(f)}(H(n-1)\setminus H(n-1;K)) \\&-
\mu_{\Cc,n,\alpha}(H(n-1)\setminus H(n-1;K))\bigg|>\frac{\epsilon}{4} \bigg\} <\frac{\delta}{2}.
\end{split}
\end{equation}
Taking into account \eqref{eq:muC>K1<eps/2}, \eqref{eq:muC(B)-mun(B) > eps <delta/2} implies that
\begin{equation}
\label{eq:muC(B) > eps <delta/2}
\prob\left\{ \mu_{\Cc(f)}(H(n-1)\setminus H(n-1;K))>\frac{\epsilon}{2} \right\} <\frac{\delta}{2}.
\end{equation}

In light of all the above we may use \eqref{eq:muC(C)-mun(C) > eps <delta/2}, \eqref{eq:muC(B) > eps <delta/2} and \eqref{eq:muC>K1<eps/2} to finally write for every $$A\subseteq H(n-1)$$ and $T>T_{0}(\epsilon,\delta)$,
\begin{equation}
\label{eq:muCf(A)-muCnalpha(A) bnd}
\begin{split}
\prob\big\{ \big|\mu_{\Cc(f)}(A) - &\mu_{\Cc,n,\alpha}(A)\big|>\epsilon \big\} \le \\\le
\prob\big\{ \big|\mu_{\Cc(f)}&(A\cap H(n-1;K)) - \mu_{\Cc,n,\alpha}(A\cap H(n-1;K))\big|>\epsilon/4 \big\}
\\+ \prob\big\{ &\mu_{\Cc(f)}(A\setminus H(n-1;K)) > \epsilon/2\big\}
< \frac{\delta}{2}+\frac{\delta}{2} = \delta,
\end{split}
\end{equation}
where the exceptional event of probability $<\delta$ independent of $A$.
Since $\delta>0$ was arbitrary, and $T_{0}$ is independent of $A$, recalling the definition
\eqref{eq:D dist prob meas def} of the distance $D(\cdot,\cdot)$ between probability measures on $H(n-1)$,
\eqref{eq:muCf(A)-muCnalpha(A) bnd} implies the convergence result
claimed in \eqref{eq:D(mu(f)-mulim >epsilon ->0}, which, as it was mentioned above, was the only thing missing detail
for completing the proof of Theorem \ref{thm:main thm lim meas}.

\end{proof}

\subsection{Proof of Theorem \ref{thm:tot cnt asymp}}

\begin{proof}
We write the absolute value as
\begin{equation*}
|\cdot| = |\cdot|_{+}+|\cdot|_{-}
\end{equation*}
with $|b|_{+}=\max(b,0)$ and $|b|_{-}=\max(-b,0)$.
To prove Theorem \ref{thm:tot cnt asymp} it is then sufficient to bound each of
$$\E\left[\left| \frac{\nod_{\cdot}(f,\cdot)}{T^{n}} - \vol\M\cdot c_{\cdot ;\gfr_{\alpha}}(\cdot) \right|_{\pm}\right].$$
The latter is the content of Proposition \ref{prop:glob upper bound} to follow immediately.
\end{proof}

\begin{proposition}
\label{prop:glob upper bound}

For every $H\in H(n-1)$ and $G\in \Tc$
\begin{equation}
\label{eq:E[nod+-]->0}
\E\left[\left| \frac{\nod_{\cdot}(f,\cdot)}{T^{n}} - \vol\M\cdot c_{\cdot ;\gfr_{\alpha}}(\cdot) \right|_{\pm} \right]\rightarrow  0.
\end{equation}

\end{proposition}

Proposition \ref{prop:glob upper bound} will be proved in \S\ref{sec:glob upper bound}, only in the (most subtle) case
\begin{equation}
\label{eq:E[nodX+]->0}
\E\left[\left| \frac{\nod_{\cdot}(f,\cdot)}{T^{n}} - \vol\M\cdot c_{X ;\gfr_{\alpha}}(G) \right|_{+} \right]\rightarrow  0.
\end{equation}
the other $3$ cases (choosing $X$ or $\Tc$, $|\cdot |_{+}$ or $|\cdot |_{-}$) being proved along similar (but somewhat easier) lines.

\subsection{Some preparatory lemmas towards the proof of Proposition \ref{prop:glob upper bound}}

\subsubsection{Excising very small and very long domains}

\begin{definition}

Let $\xi, D>0$ be parameters.

\begin{enumerate}

\item A component $c\in \Cc(f)$ is $\xi$-{\em small} if it is a boundary of
a nodal domain whose volume in $\M$ is $<\xi T^{-n}$.
Let $\nod_{\Cc;\xi-sm}(f)$ be the total number of $\xi$-small components of $f$ on $\M$.

\item For $D>0$ a nodal component $c\in\Cc(f)$ is $D$-{\em long} if its diameter is $>D/T$.
Let $\nod_{\Cc;D-long}(f)$ be their total number.

\item Given parameters $D,\xi > 0$ a nodal component $c\in\Cc(f)$ is ($D,\xi$)-{\em normal},
if it is not $\xi$-small nor $D$-long.

\item For $G\in \Tc$ let $\nod_{X;norm}(f,G)$ be the total number of
$(\xi,D)$-normal components $c\in \Cc(f)$ of $f$ such that $e(c)$ is isomorphic to $G$.

\item For $x\in\M$, $r>0$ let $\nod_{X;norm}(f,G;x,r)$ (resp. $\nod_{X;norm}^{*}(f,G;x,r)$) be the number of
those $c$ contained in the geodesic ball $B(x,r)\subseteq \M$
(resp. intersecting $\overline{{B}(x,r)}$).
(Here we use Jordan's Theorem on the sphere to choose those vertices of $e(c)$ lying inside $B(x,r)$, or intersect
it respectively).

\end{enumerate}

\end{definition}

By the definition of normal ovals, we have
\begin{equation}
\label{eq:Nnorm*(r)<=N(r+D/L)}
\nod_{X;norm}^{*}(f,G;x,r) \le \nod_{X;norm}\left(f,G;x,r+\frac{D}{T}\right)
\end{equation}
(as we discarded the very long ovals), and uniformly
\begin{equation}
\label{eq:Nnorm<=delta^{-1}vol(B)}
\nod_{X;norm}(f,G;x,r) \le \xi^{-1}T^{n} \vol_{\M}B(x,r),
\end{equation}
by a volume estimate (as we discarded the very small domains).

\begin{lemma}[Cf. {\cite[Lemma $8$]{So}}]
\label{lem:E[ND-long/L^{n}]->0}
There exists a constant
$C_{0}>0$ such that the following bound holds for the number of $D$-long components:
\begin{equation*}
\limsup\limits_{T\rightarrow\infty}\frac{\E[\nod_{\Cc;D-long}(f)]}{T^{n}} \le C_{0}\cdot \frac{1}{D}
\end{equation*}
\end{lemma}

A proof, given in this generality in ~\cite{So} and omitted here, is by
taking a $D/T$-net on $\M$ and using the Kac-Rice estimate \eqref{eq:Kac-Rice upper bnd gen}.

\begin{lemma}[Cf. {\cite[Lemma $9$]{So}},
and Lemma \ref{lem:xi-small bound} in the scale-invariant case]
\label{lem:E[Nxi-sm/L^{n}]->0}
Then there exist constants $c_{0},C_{0}>0$ so that the following estimate on the number of $\xi$-small components:
\begin{equation*}
\limsup\limits_{T\rightarrow\infty}\frac{\E[\nod_{\Cc;\xi-sm}(f)]}{T^{n}} \le C_{0}\cdot \xi^{c_{0}}.
\end{equation*}

\end{lemma}

The proof of Lemma \ref{lem:E[Nxi-sm/L^{n}]->0}, omitted here,
is very similar to the proof of Lemma \ref{lem:xi-small bound}, given in
\S\ref{sec:xi-small bound proof} of the present manuscript.

\subsubsection{Integral geometric sandwich on $\M$ for normal components}

\begin{lemma}[Integral-geometric sandwich on a Riemannian manifold,
cf. {\cite[Lemma $1$]{So}}, and Lemma \ref{lem:int geom sandwitch euclidian}
in the scale-invariant case]
\label{lem:int geom sand Riemann}

Let $G\in \Tc$. Given $\epsilon > 0$, there exists $\eta>0$ such that for every $r<\eta$,
\begin{equation}
\label{eq:Nnorm<=(1+eps)*int loc + add}
\begin{split}
\nod_{X;norm}(f,G) &\le (1+\epsilon)\int\limits_{\M}\frac{\nod_{X;norm}^{*}(f,G;x,r)}{\vol(B(r))}dx
\\&+\xi^{-1}T^{n}\vol_{\M}(B_{x}(r))+O_{G}(1).
\end{split}
\end{equation}

\end{lemma}

The proof of Lemma \ref{lem:int geom sand Riemann} is very similar to the one
of Lemma \ref{lem:int geom sandwitch euclidian}, and we omit it here.

\begin{remark}

\begin{enumerate}

\item Note the differences between \eqref{eq:Nnorm<=(1+eps)*int loc + add} and the analogous statement
\eqref{eq:geom int scal invar} in the scale invariant case. First, the $(1+\epsilon)$ factor
in \eqref{eq:Nnorm<=(1+eps)*int loc + add} manifests the perturbed volumes of small geodesic balls on
$\M$ as compared to the Euclidian balls. We then need to account for a (exclusive to $\nod_{X}$)
situation where a component $c\in \Cc(f_{L})$ is contained in some geodesic ball $B_{x}(r)$, but the corresponding
tree end is not, which may be bounded by the volume estimate \eqref{eq:Nnorm<=delta^{-1}vol(B)} responsible
for the first extra term in \eqref{eq:Nnorm<=(1+eps)*int loc + add}.
The other peculiar $O_{G}(1)$ term in \eqref{eq:Nnorm<=(1+eps)*int loc + add}
may be taken as at most the vertex number of $G$.

\item For the lower bound the following inequality holds outside an event of arbitrarily small
probability only (The reason that this may not always hold is again, that in the non-Euclidean
case the graph end $e(c)$ may fail to be lying in a geodesic ball, even if $c$ is)
\begin{equation*}
(1-\epsilon)\int\limits_{\M}\frac{\nod_{X;norm}(f,G;x,r)}{\vol(B(r))}dx  \le \nod_{X;norm}(f,G).
\end{equation*}

\end{enumerate}

\end{remark}

\subsubsection{Uniform bound for the number of components in small balls}

\begin{lemma}
\label{lem:Kac-Rice exp nod bnd}
For $r$ sufficiently small, depending only on $\M$ we have the following uniform bound
\begin{equation*}
\E[\nod_{\Cc}(f;x,r)]= O(r^{n}\cdot T^{n}),
\end{equation*}
with constant involved in the $`O'$-notation depending on $\M$ only.
\end{lemma}

\begin{proof}
First, the nodal components count is bounded
\begin{equation*}
\Cc(f) \le \E[\Ac(f)]
\end{equation*}
by the number of critical points of $f$. To bound the latter we use Lemma \ref{lem:Kac-Rice upper bnd gen} on $$G=\nabla f,$$
understood in a coordinate patch around $x$.
We claim that \eqref{eq:Kalpha asymp}, applied on $K_{\alpha}(T;x,y)$ and its derivatives,
implies that the expression
\begin{equation}
\label{eq:Kac-Rice weak upper G=nabla f}
\frac{\E[|\nabla G(x)|^{2}]^{n/2}}{\left(\det\E[G(x)\cdot G(x)^{t}]\right)^{1/2}}
= \frac{\E[|(f_{ij}(x))_{i,j}|^{2}]^{n/2}}{\left(\det\E[\nabla f(x)\cdot \nabla f(x)^{t}]\right)^{1/2}} = O(T^{n}),
\end{equation}
on the rhs of \eqref{eq:Kac-Rice upper bnd gen} is uniformly bounded, with constant involved in the $`O'$-notation
depending only on $n$.

First, we observe that the denominator of \eqref{eq:Kac-Rice weak upper G=nabla f}
is the determinant of the covariance matrix $\Sigma(x)$ of $\nabla f(x)$. We know from
\eqref{eq:Kalpha asymp} that, after scaling by $T$, the entries of $\Sigma(x)$
converge uniformly to the entries of the non-degenerate covariance matrix of $\nabla \gfr$. Hence, properly scaled,
$$\det\E[\nabla f(x)\cdot \nabla f(x)^{t}]$$ is bounded away from $0$. Concerning the numerator
$\E[|f_{ij}(x)|]$ of \eqref{eq:Kac-Rice weak upper G=nabla f}, we may use the triangle inequality
to bound $$\E[|(f_{ij}(x))|^{2}] \le  \sum\limits_{i,j} \E\left[f_{ij}(x)^{2}\right];$$ after scaling by $T$
(consistent with the scaling of the denominator), the latter is uniformly bounded, due to
\eqref{eq:Kalpha asymp} applied to the corresponding $4$-th order mixed derivative.
Hence \eqref{eq:Kac-Rice upper bnd gen} implies the statement of the present lemma.

\end{proof}

\subsection{Proof of Proposition \ref{prop:glob upper bound}}
\label{sec:glob upper bound}

\begin{proof}
Throughout the course of this proof we will assume with no loss of generality that $\M$ is unit volume $\vol\M=1$,
and for a given $G\in\Tc$ we will use the shorthand
\begin{equation}
\label{eq:cX=grfalpha}
c_{X}:=c_{X ;\gfr_{\alpha}}(G);
\end{equation}
and as it was mentioned above we only show \eqref{eq:E[nodX+]->0} here.
Let $R>0$ be a large number, so that $D/R$ is sufficiently small, and $R/T<\eta$ as in Lemma \ref{lem:int geom sand Riemann},
sufficiently small, so that
\begin{equation}
\label{eq:volMB(r)/volB(r)<1+eps}
\begin{split}
&\left|\frac{\vol_{\M}(B_{x}(R/T))}{\vol(B(R/T))} -1\right| < \epsilon,
\\  &\frac{\vol(B(R+D))}{\vol(B(R))} < 1+\epsilon
\end{split}
\end{equation}
uniformly for $x\in\M$.

Apply Lemma
\ref{lem:int geom sand Riemann} with $r=\frac{R}{T}$;
by the triangle inequality for $|\cdot |_{+}$ we have (recall \eqref{eq:cX=grfalpha})
\begin{equation}
\label{eq:sand apply nodT}
\begin{split}
&\E\left[\left| \frac{\nod_{\Tc;norm}(f,G)}{T^{n}} - \vol\M\cdot c_{X} \right|_{+} \right]
\\&\le \E\left[ \int\limits_{\M}\left|(1+2\epsilon)\frac{\nod_{X;norm}^{*}(f,G;x,R/T)}{\vol B(R+D)}-c_{X}\right|_{+}dx \right]
\\&+O(T^{-n}+\xi^{-1}\vol_{\M}(B_{x}(R/T)))
\\&\le
\E\left[ \int\limits_{\M}\left|\frac{\nod_{X;norm}(f,G;x,(R+D)/T)}{\vol B(R+D)}-c_{X}\right|_{+}dx \right]
\\&+O\left(\epsilon\cdot \int\limits_{\M}\frac{\E[\nod_{X;norm}(f,G;x,(R+D)/T)]}{\vol B(R+D)}dx \right)
\\&+O(T^{-n}+\xi^{-1}\vol_{\M}(B_{x}(R/T))),
\end{split}
\end{equation}
by \eqref{eq:Nnorm*(r)<=N(r+D/L)} and \eqref{eq:volMB(r)/volB(r)<1+eps}.
Observe that the integrand $$\frac{\E[\nod_{X;norm}(f,G;x,(R+D)/T)]}{\vol B(R+D)}$$ is uniformly bounded by
Lemma \ref{lem:Kac-Rice exp nod bnd}. Hence \eqref{eq:sand apply nodT} is
\begin{equation*}
\begin{split}
&\E\left[\left| \frac{\nod_{X;norm}(f,G)}{T^{n}} - \vol\M\cdot c_{X} \right|_{+} \right]
\\&\le\E\left[\int\limits_{\M} \left| \frac{\nod_{X;norm}(f,G;x,(R+D)/T)}{\vol B(R+D)}-c_{X}\right|_{+}dx\right]\\& +
O(\epsilon+T^{-n}+\xi^{-1}\vol_{\M}(B_{x}(R/T))),
\end{split}
\end{equation*}

Since the latter error term $O(\cdots)$ could be made arbitrarily small, it is then sufficient to prove that
\begin{equation}
\begin{split}
\label{eq:int Omega,M N(S,x)-cS(x)}
&\E\left[\int\limits_{\M}\left| \frac{\nod_{X;norm}(f,G;x,(R+D)/T)}{\vol B(R+D)}-c_{X}(G,x)\right|_{+}dx\right] \\&=
\int\limits_{\Delta}\int\limits_{\M} \left|\frac{\nod_{X;norm}(f,G;x,(R+D)/T)}{\vol B(R+D)}-c_{X}\right|_{+}dx d\Pc(\omega)
\rightarrow 0,
\end{split}
\end{equation}
where $\Delta$ is the underlying probability space, and $\Pc$ is the probability measure on $\Delta$.
Now consider the event
$$\Delta_{T,G;x,R} = \left\{  \left| \frac{\nod_{X}(f,G;x,R/T)}{\vol B(R+D)}-c_{X} \right|>\epsilon\right\};$$
note that though formally the values of $T$ can attain a continuum variety in $\R$, only countably many of them
would yield genuinely different functions $f$ in \eqref{eq:f = sum cj phij} and their by-products, such as $\Delta_{T,G;x,R}$,
so we may assume that any limit $T\rightarrow\infty$ is along this countable system; from this point on we will neglect
this difference, which will save us from dealing with various measurability issues.
Then by Theorem \ref{thm:loc scal}, for every $x\in\M$
\begin{equation}
\label{eq:P(Omega)->0}
\lim\limits_{R\rightarrow\infty}\limsup\limits_{T\rightarrow\infty} \Pc(\Delta_{T,G;x,R+D}) =0.
\end{equation}

We claim that there exists a sequence\footnote{As above, this will simplify our treatment of measurability;
with some more effort we could work with an arbitrary countable sequence, or even a
continuum of $\{R\}$} $\{R_{j}\}_{j\rightarrow\infty}$ (e.g. $R_{j}\in\Z$ integers) so that
the limit \eqref{eq:P(Omega)->0} is almost uniform w.r.t. $x\in X$, that is, for every $\eta>0$
there exists $\M_{\eta} \subseteq \M$ with $\vol \M_{\eta}>1-\eta$, such that
\begin{equation}
\label{eq:P(Omega) Egorov}
\varliminf\limits_{R_{j}\rightarrow\infty}\varlimsup\limits_{T\rightarrow\infty} \sup\limits_{x\in\M_{\eta}}\Pc(\Delta_{T,G;x,R_{j}+D}) =0.
\end{equation}
To see \eqref{eq:P(Omega) Egorov} we first apply an Egorov-type theorem on the limit in \eqref{eq:P(Omega)->0} w.r.t.
$R\rightarrow\infty$: working with the sets $$E_{n,k} =
\bigcup\limits_{R>n \text{ integer}}\left\{x\in \M:\: \Pc(\Delta_{T,G;x,R+D}) > \frac{1}{k} \text{ for } T \text{ sufficiently big} \right\}$$
yields that for some $\M_{\eta}$ with $\vol(\M_{\eta})>1-\frac{\eta}{2}$
\begin{equation*}
\lim\limits_{R_{j}\rightarrow\infty}\sup\limits_{x\in\M_{\eta}}\varlimsup\limits_{T\rightarrow\infty} \Pc(\Delta_{T,G;x,R_{j}+D}) =0;
\end{equation*}
this is not quite the same as the claimed result \eqref{eq:P(Omega) Egorov}, as the order of $\sup\limits_{x\in\M_{\eta}}$
and the $\limsup$ w.r.t. $T\rightarrow\infty$ is wrong.
We use an Egorov-type argument once again, w.r.t. the limit $\varlimsup\limits_{T\rightarrow\infty}$
to mollify this.
Fix an integer $r>0$, and let $R_{j}=R_{j(r)}>0$ sufficiently big
so that
\begin{equation}
\label{eq:sup(limsup)<1/r}
\sup\limits_{x\in\M_{\eta}}\varlimsup\limits_{T\rightarrow\infty} \Pc(\Delta_{T,G;x,R_{j}+D})<\frac{1}{r}.
\end{equation}
Define the monotone decreasing sequence of sets
$$F_{m} = \bigcup\limits_{T>m}\left\{x\in \M_{\eta}:\: \Pc(\Delta_{T,G;x,R_{j}+D}) > \frac{2}{r} \right\}.$$
Since, by \eqref{eq:sup(limsup)<1/r}, $$\bigcap\limits_{m\ge 1} F_{m}=\emptyset,$$
we may find $m=m(r)$ sufficiently
big so that $\vol(F_{m(r)})<\frac{\eta}{2^{r+1}}$. Therefore the claimed result \eqref{eq:P(Omega) Egorov} holds on
$$M_{\eta}\setminus \bigcup\limits_{r\ge 1}F_{m(r)},$$ i.e.
further excising the set $\bigcup\limits_{r\ge 1}F_{m(r)} $ of volume $<\frac{\eta}{2}$ from $\M_{\eta}$.

We then write the integral \eqref{eq:int Omega,M N(S,x)-cS(x)} as
\begin{equation}
\label{eq:int OmegaM [NV-eta]+}
\begin{split}
&\int\limits_{\Delta}\int\limits_{\M} \left|\frac{\nod_{X;norm}(f,G;x,(R+D)/T)}{\vol B((R+D))}-c_{X}\right|_{+}dx d\Pc (\omega)\\&=\int\limits_{\M}\int\limits_{\Delta_{T,G;x,R+D}} +\int\limits_{\M}\int\limits_{\Delta\setminus\Delta_{T,G;x,R+D}}.
\end{split}
\end{equation}
First, on $\Delta\setminus\Delta_{T,G;x,R+D}$, the integrand of \eqref{eq:int OmegaM [NV-eta]+} is
\begin{equation*}
\begin{split}
&\left|\frac{\nod_{X;norm}(f,G;x,(R+D)/T)}{\vol B((R+D)/T)}-c_{X}(G,x)\right|_{+} \\&\le
\left|\frac{\nod_{X}(f_{L},G;x,(R+D)/T)}{\vol B((R+D)/T)}-c_{X}(G,x)\right|   \le \epsilon,
\end{split}
\end{equation*}
and hence the contribution of this range is
\begin{equation}
\label{eq:int Nnorm /Omega}
\begin{split}
&\int\limits_{\M}\int\limits_{\Delta\setminus\Delta_{T,G;x,R+D}}
\left|\frac{\nod_{X;norm}(f,G;x,(R+D)/T)}{\vol B((R+D)/T)}-c_{X}\right|_{+}dx d\Pc (\omega)
\\&\le \int\limits_{\M}\int\limits_{\Delta\setminus\Delta_{T,G;x,R+D}}\epsilon dxd\Pc(\omega) \le \epsilon.
\end{split}
\end{equation}

On $\Delta_{T,G;x,R+D}$ we use the volume estimate \eqref{eq:Nnorm<=delta^{-1}vol(B)} yielding uniformly on $x\in\M$
\begin{equation}
\label{eq:NV triv estimate Omega}
\begin{split}
&\int\limits_{\Delta_{T,G;x,R+D}}
\left|\frac{\nod_{X;norm}(f,G;x,(R+D)/T)}{\vol B(R+D)}-c_{X}\right|_{+} d\Pc (\omega)
\\&\le \int\limits_{\Delta_{T,G;x,R+D}}
\left|\frac{\nod_{X;norm}(f,G;x,(R+D)/T)}{\vol B(R+D)}\right|_{+}d\Pc (\omega)
\\&\le \int\limits_{\Delta_{T,G;x,R+D}}\xi^{-1}T^{n}\frac{\vol_{\M}(B_{x}((R+D)/T))}{\vol B(R+D)}d\Pc (\omega)  \\&\le (1+\epsilon)\xi^{-1}\Pc(\Delta_{T,G;x,R+D}).
\end{split}
\end{equation}
Similarly to the above, uniformly on $\omega\in\Delta$
\begin{equation}
\label{eq:NV triv estimate M/Meta}
\int\limits_{\M\setminus \M_{\eta}}
\left|\frac{\nod_{X;norm}(f,G;x,(R+D)/T)}{\vol B(R+D)}-c_{X}\right|_{+} dx \le (1+\epsilon)\xi^{-1}\eta.
\end{equation}

The uniform estimates \eqref{eq:NV triv estimate Omega} and \eqref{eq:NV triv estimate M/Meta} imply that
\begin{equation*}
\begin{split}
&\int\limits_{\M}\int\limits_{\Delta_{T,G;x,R+D}}
\left|\frac{\nod_{X;norm}(f,G;x,(R+D)/T)}{\vol B(R+D)}-c_{X}\right|_{+}
dx\Pc(\omega) \\&\le (1+\epsilon)\xi^{-1}
(\sup\limits_{x\in \M_{\eta}}\Pc(\Delta_{T,G;x,R+D}) + \eta),
\end{split}
\end{equation*}
and upon substituting the latter estimate and \eqref{eq:int Nnorm /Omega} into \eqref{eq:int OmegaM [NV-eta]+},
and then to the integral \eqref{eq:int Omega,M N(S,x)-cS(x)}, we finally obtain
\begin{equation*}
\begin{split}
&\E\left|\int\limits_{\M} \frac{\nod_{X;norm}(f,G;x,(R+D)/T)}{\vol B(R+D)}-c_{X}\right|_{+}dx
\\&\le \epsilon + (1+\epsilon)\xi^{-1}
(\sup\limits_{x\in \M_{\eta}}\Pc(\Delta_{T,G;x,R+D}) + \eta),
\end{split}
\end{equation*}
which could be made arbitrarily small for each sufficiently small choice of $\xi$ excising the very small
components, and using \eqref{eq:P(Omega) Egorov}. This concludes the proof of \eqref{eq:int Omega,M N(S,x)-cS(x)},
sufficient to yield the conclusion of the present proposition.
\end{proof}

\appendix

\section{Measurability of nodal counts}

\label{apx:measurability}

\begin{proof}[Proof of Lemma \ref{lem:measurability}]

We address briefly the issue of the measurability of various counts, such as $\nod_{\Cc}(F,H;r)$ as functions of the Gaussian field $F$.
These functions are refinements of the counting functions $\nod(G,\cdot)$ in ~\cite{NS2015}, for which the measurability is discussed
on pages 31, 43, and one can extend their arguments to deal with our $\nod$'s. Rather than doing that we give a direct analysis for our Gaussian
fields $\gfr_{n,\alpha}$ as defined in \S\ref{sec:scal invar gfr}, and we take this opportunity to explicate their meaning.
We are only going to prove the measurability statements for $\nod_{\Cc}$ w.r.t. $\omega\in\Delta$,
the proof of the measurability of $\nod_{X}$ being identical.
First we establish part (1) of Lemma \ref{lem:measurability}.

The functions on $\R^{n}$ are given as ``$a$-linear functions"
\begin{equation}
\label{eq:F Karl exp}
F_{\omega}(x) = \sum\limits_{j=1}^{\infty} a_{j}\xi_{j}(x),
\end{equation}
where the $\xi_{j}(x)$ are the Fourier transforms of $\psi_{j}$ (as in \S\ref{sec:scal invar gfr}) and $\omega=(a_{1},a_{2},\ldots)\in \R^{\N}$
with $a_{j}$'s i.i.d. $N(0,1)$ Gaussian variables on $\R$ (see \cite{Adler-Taylor}). The $\xi_{j}$'s are smooth (even analytic) functions on $\R^{n}$ which together with their derivatives are rapidly decreasing in $j$ for $x$ in compact subdomains of $\R^{n}$. The series \eqref{eq:F Karl exp} converges for almost all $\omega$ and defines a function on $\R^{n}$ which is our Gaussian field.

In more detail, we equip $\R^{\N}=\{\omega:\: \omega=(a_{1},a_{2},\ldots )\}$ with the $\sigma$-algebra $\mathscr{A}$ generated by
the cylinder sets $\{\omega:\: a_{j}\in A_{j},\, j=1,\ldots, k\}$, where $A_{j}$ are subintervals of $\R$. We form the probability space
$\mathbb{P}=(\R^{\N},\mathscr{A},\mu)$, where $\mu$ is the product Gaussian; $\mu=\mu_{1}\times \mu_{2}\times\ldots$ and $\mu_{j}$ is the
standard $N(0,1)$ Gaussian on each factor. It is clear that for a fixed $x\in \R^{n}$, $F_{\omega}(x)$ is a Gaussian on $\R$
with mean $0$ and variance
\begin{equation*}
\E_{\omega}\left[F_{\omega}(x)^{2}\right] = \sum\limits_{j=1}^{\infty}\xi_{j}(x)^{2}.
\end{equation*}

More generally, for distinct $x_{1},x_{2},\ldots, x_{k}\in \R^{n}$ the $k$-dimensional vector \\
$(F_{\omega}(x_{1}),\ldots, F_{\omega}(x_{k}))$ is
a Gaussian with mean $(0,0,\ldots, 0)$ and covariance
\begin{equation*}
\E_{\omega}[F_{\omega}(x)\cdot F_{\omega}(y)] = \sum\limits_{j=1}^{\infty}\xi_{j}(x)\cdot \xi_{j}(y) = \widehat{\nu_{\alpha}}(x-y)
\end{equation*}
as in \eqref{eq:reprod kern covar} for our $\gfr_{n,\alpha}$'s.
This yields a concrete realization through $F_{\omega}(x)$, $\omega\in\mathbb{P}$ of the mean zero stationary (isotropic) Gaussian field
with covariance $\widehat{\nu_{\alpha}}(x-y)$.

In order to examine the typical $F_{\omega}(x)$ and to define the functions $\nod$ we remove various $\mu$-null sets. To this end
let $W=\{w_{j}\}_{j\ge 1}$ be the weights $w_{j}=j^{-A}$ (with $A>0$ fixed). The function $f:\R^{\N}\rightarrow [0,\infty]$
given by $f(\omega) = \sum\limits_{j=1}^{\infty}|a_{j}|^{2} w_{j}$ is defined and $\mathscr{A}$-measurable, and by the monotone
convergence theorem
\begin{equation*}
\E_{\omega}[f(w)] = \sum\limits_{j=1}^{\infty}w_{j}<\infty.
\end{equation*}
Hence $f$ is finite except on a $\mu$-null set, and we restrict to this almost full set of $\omega$'s
\begin{equation*}
l^{2}(W) := \left\{\omega:\: \sum\limits_{j=1}^{\infty}|a_{j}|^{2}w_{j}<\infty\right\}.
\end{equation*}
The set $l^{2}(W)$ is also a Hilbert space, being a $l^{2}$ sequence space; whether we view $l^{2}(W)$ as a measure
space $\mathbb{P}\cap l^{2}(W)$ with measure $\mu$ or a linear Hilbert space will be made clear.

Let $\Omega\subseteq \R^{n}$ be a domain with $\overline{\Omega}$ compact, then for $r\ge 1$, $B>\frac{A+1}{2}$ there is
a number $C_{r,\Omega,B}<\infty$ such that
\begin{equation*}
\sup\limits_{\substack{k\le r \\ x\in\Omega}}|D^{k}\xi_{j}(x)| \le C_{r,\Omega,B}j^{-B}.
\end{equation*}
Hence for $x\in\Omega$, we have
\begin{equation*}
\begin{split}
&|D^{k}F(x)| \le \sum\limits_{j=1}^{\infty} |a_{j}| C_{r,\Omega,B}j^{-B}
\\&\le C_{r,\Omega,B}\cdot \left( \sum\limits_{j=1}^{\infty}|a_{j}|^{2}w_{j} \right)^{1/2} \cdot
\left(\sum\limits_{j=1}^{\infty}  j^{-2B+A}\right)^{1/2}.
\end{split}
\end{equation*}
Hence the linear map
\begin{equation}
\label{eq:T bounded l2->Ct}
T:l^{2}(W)\rightarrow C^{t}(\overline{\Omega})
\end{equation}
with $t\ge 0$, given by \eqref{eq:F Karl exp}, is bounded
(i.e. is a continuous map between these Banach spaces).
Now the open sets in $l^{2}(W)$ (with its topology) are measurable as subsets of $\R^{\N}$
(i.e. they are in $\mathscr{A}$), hence we conclude that
\begin{equation}
\label{eq:Fomega(x) meas}
\omega\mapsto F_{\omega}(x)
\text{ is measurable as a map }l^{2}(W) \rightarrow C^{t}(\overline{\Omega}).
\end{equation}
Since $\Omega$ is arbitrary,
we conclude that for $\omega\in l^{2}(W)$, we have $F_{\omega}(x)\in C^{\infty}(\R^{n})$.

In order to study (or even define) our functions $\nod_{\Cc}(F_{\omega},H;r)$ we first examine
$\nod$ as a function from $C^{t}(\overline{\Omega})$ to $\R$. Let $\Omega=B(0,r)$ be the ball
centred at $0$ of radius $r$ in $\R^{n}$. The sets
\begin{equation*}
B_{1}=\{f\in C^{1}(\overline{\Omega}):\: f(x)=0,\, \nabla f(x)=0 \text{ for some } x\in \overline{\Omega}\}
\end{equation*}
and
\begin{equation*}
B_{2}=\{f\in C^{1}(\overline{\Omega}): f(x)=0,\, \text{ and } \nabla f(x)\perp T_{x}(\partial \Omega)
\text{ for some } x\in\partial\Omega\}
\end{equation*}
(i.e. $B_{2}$ consists of $f$ so that there exists some zero $x\in \partial\Omega$ of $f$ on the boundary, such that
$\nabla f(x)$ is orthogonal to the tangent space to $\partial \Omega$ at $x$),
are closed subsets of $C^{1}(\overline{\Omega})$. For $f$ in the open subset
$$\Theta = C^{1}(\overline{\Omega})\setminus (B_{1}\cup B_{2}),$$ the zero set $ V(f)$ has finitely many connected components
that are fully contained in $\Omega$, and these are nonsingular. Of these, let $\nod_{\Cc}(f,H;r)$ be the number of such
components diffeomorphic to $H$. Having removed $B_{1}$ and $B_{2}$ it follows
by the stability argument presented within the proof of Proposition \ref{prop:bi(R-1,Fx)<=bi(R,fxL)<=bi(R+1,Fx)}
in \S\ref{sec:bi(R-1,Fx)<=bi(R,fxL)<=bi(R+1,Fx)}, and the boundedness of the map \eqref{eq:T bounded l2->Ct},
that each of the topologies of the fully contained connected components of $V(f)$ are unchanged for $g$ in a small enough neighbourhood of $f$ in
$\Theta$.

We are ready to define $\nod_{\Cc}(F_{\omega},H;r)$ for almost all $\omega$. According to \eqref{eq:Fomega(x) meas}
we have a measurable decomposition
\begin{equation*}
l^{2}(W) = T^{-1}(B_{1}\cup B_{2}) \sqcup T^{-1}(\Theta).
\end{equation*}
By the generalization of Bulinskaya's Lemma for higher dimensions (see e.g. ~\cite[Proposition 6.11]{AW}),
$T^{-1}(B_{1}\cup B_{2})$ is $\mu$-null. We define $\nod_{\Cc}(F_{\omega},H;r)$ on the full set $T^{-1}(\Theta)$
by the composition
\begin{equation*}
\nod_{\Cc}: T^{-1}(\Theta)\xrightarrow{T} \Theta \xrightarrow{\nod_{\Cc}(\cdot, H;\Omega)}\R.
\end{equation*}
The second map is continuous and the first measurable, hence $\nod_{\Cc}(F_{\omega},H;r)$ is defined for almost
all $\omega$ and is measurable.

\vspace{1mm}

As a final remark, note that although $\nod_{\Cc}$ is measurable, its determination on a particular good $\omega$
is in general undecidable. The reason is that in high dimensions there is no decision procedure to decide if a given component
of $V(f)$ is a given $H$ in $H(n-1)$ (see \cite{Nab}).

\vspace{2mm}

For part (2) note that almost surely $F$ is smooth on $B(R)$. Therefore for almost all
$\omega\in\Delta$, $\nod_{\Cc}(F_{\omega},H;x,r)$ is locally constant w.r.t. $x\in B(R)$ outside of
a measure zero set, and, in particular, measurable.
Finally, in order to establish part (3) of Lemma \ref{lem:measurability} we combine both parts (1) and (2). Namely, since, by part (2),
for almost all $\omega\in\Delta$, $x\mapsto \nod_{\Cc}(F_{\omega},H;x,r)$ is locally constant on a set of full measure in $B(R)$,
we may write for almost all $(\omega,x)\in \Delta\times B(R)$:
$$\nod_{\Cc}(F_{\omega},H;x,r)=\lim\limits_{n\rightarrow\infty}\nod_{\Cc}(F_{\omega},H;\lfloor nx\rfloor/n,r),$$ measurable as a limit of measurable functions; here for a vector $x\in\R^{n}$ we denote $\lfloor x \rfloor := (\lfloor x_{1}\rfloor ,\ldots ,\lfloor x_{n}\rfloor)$.

\end{proof}

\section{An alternative proof for topology not leaking}

\label{apx:top not leaking}

Here we present an alternative, shorter, proof of Theorem \ref{thm:scal invar no leak} part \eqref{it:top not leaking}
invoking some abstract tools not employing the uniform stability of Proposition \ref{prop:grad bounded away}.
A similar argument yields a proof
of Theorem \ref{thm:scal invar no leak} part \eqref{it:nest not leaking}.

\begin{proof}[Proof of Theorem \ref{thm:scal invar no leak} part \eqref{it:top not leaking}]
We reuse all the notation of \S\ref{sec:proof top not leak}, and, as in the proof presented in \S\ref{sec:proof top not leak},
we aim at proving tightness: for every $\delta>0$
there exists a {\em finite} $$A_{0}=A_{0}(\delta)\subseteq H(n-1)$$ so that for $R>0$ sufficiently big we have
\begin{equation}
\label{eq:tightness reprise}
\E[\nod_{\Cc}(F,H(n-1)\setminus A_{0};R)] < \delta\cdot R^{n};
\end{equation}
instead of constructing $A_{0}$ explicitly we will merely show its existence using some abstract tools.
Likewise, we apply \eqref{eq:geom int sand Cc upper*} the Integral-Geometric Sandwich, and take
the expectation of both sides to yield that
\begin{equation}
\label{eq:geom int sand Cc upper reprise}
\E[\nod_{\Cc}(F,A;R)] \le \left(\frac{R}{r}+1\right)^{n}\cdot \E[\nod_{\Cc}(F,A;r)] + \frac{\delta}{2}\cdot R^{n},
\end{equation}
valid for $r>r_{0}(\delta)$ sufficiently big.
From \eqref{eq:geom int sand Cc upper reprise} it follows that in order to prove the tightness
\eqref{eq:tightness reprise} it is sufficient to find a finite $A_{0}\subseteq H(n-1)$
so that
\begin{equation}
\label{eq:E[F;H/A0] < del/4 r^n}
\E[\nod_{\Cc}(F,H(n-1)\setminus A_{0};r)]<\frac{\delta}{4}r^{n}
\end{equation}
is arbitrarily small for $r$ fixed (though arbitrarily big).

Now, the expectation
\begin{equation*}
E:=\E[\nod_{\Cc}(F;r)] < \infty
\end{equation*}
of the total number of nodal components (of unrestricted topology), is finite.
Using the fact that the collection $H(n-1)$ of diffeomorphism types is countable, let $$H(n-1)=\{H_{j}\}_{j\ge 1}$$
be an enumeration of $H(n-1)$. We then have
\begin{equation*}
\nod_{\Cc}(F;r) = \sum\limits_{j\ge 1} \nod_{\Cc}(F,H_{j};r),
\end{equation*}
where all of $\{\nod_{\Cc}(F,H_{j};r)\}_{j\ge 1}$ are non-negative random variables (i.e. measurable on
functions on the sample space). By the Monotone Convergence Theorem we then have
\begin{equation*}
E=\E[\nod_{\Cc}(F;r)] = \sum\limits_{j\ge 1}\E[\nod_{\Cc}(F,H_{j};r)],
\end{equation*}
so that the series $$\sum\limits_{j\ge 1}\E[\nod_{\Cc}(F,H_{j};r)] < \infty$$ is convergent.

Hence, by taking a tail of a convergent series,
there exists a number $j_{0} = j_{0}(\delta,r) = j_{0}(r)$ sufficiently big so that
\begin{equation}
\label{eq:sum exp j>j0<delta/4 r^n}
\sum\limits_{j> j_{0}}\E[\nod_{\Cc}(F,H_{j};r)] < \frac{\delta}{4}r^{n}
\end{equation}
(recall that $r>r_{0}$ is fixed but sufficiently big).
Now we choose $$A_{0} = \{H_{j}\}_{j\le j_{0}};$$ we then have
\begin{equation*}
\nod_{\Cc}(F,H(n-1)\setminus A_{0};r) = \sum\limits_{j>j_{0}} \nod_{\Cc}(F,H_{j};r),
\end{equation*}
so that, upon applying the Monotone Convergence in a fashion similar to the above (again, using the measurability
of all $\nod_{\Cc}(F,H_{j};r)$ on the sample space), we obtain
\begin{equation*}
\E[\nod_{\Cc}(F,H(n-1)\setminus A_{0};r)] = \sum\limits_{j>j_{0}} \E[\nod_{\Cc}(F,H_{j};r)] < \frac{\delta}{4}r^{n},
\end{equation*}
by \eqref{eq:sum exp j>j0<delta/4 r^n}, which is precisely \eqref{eq:E[F;H/A0] < del/4 r^n}, that is readily proven
to be sufficient for the tightness \eqref{eq:tightness reprise} via \eqref{eq:geom int sand Cc upper reprise}.

\end{proof}

\end{document}